\documentclass[a4paper, 12pt, final]{amsart}

\makeatletter
\numberwithin{equation}{section}
\numberwithin{figure}{section}

\usepackage{amsmath,amsfonts,amssymb,amsthm,epsfig}

\usepackage{color}
\usepackage[final,linkcolor = blue,citecolor = blue,colorlinks=true]{hyperref}
\usepackage[leqno]{amsmath}

\usepackage{amscd}
\usepackage{mathrsfs}
\usepackage{}
\usepackage[OT1]{fontenc}
\usepackage[latin1]{inputenc}
\usepackage{amssymb,latexsym}
\usepackage{type1cm,ae}
\usepackage[notref,notcite]{showkeys}
\usepackage{graphicx}
\usepackage{xspace}
\usepackage{setspace}
\usepackage[english]{babel}
\usepackage[all,cmtip]{xy}

\theoremstyle{definition}
		\newtheorem{theorem}{Theorem}[section]
				\newtheorem{proposition}[theorem]{Proposition}
				\newtheorem{lemma}[theorem]{Lemma}
				\newtheorem{thmA}{Theorem}[section]
				\newtheorem{corollary}[theorem]{Corollary}

     	        \newtheorem{definition}[theorem]{Definition}
	            \newtheorem{remark}[theorem]{Remark}
\numberwithin{equation}{section}

\newcommand*{\bR}{\ensuremath{\mathbb{R}}}

\newcommand*{\loc}{\mathrm{loc}}
\newcommand*{\closure}[1]{\overline{#1}}
\newcommand*{\bdary}[1]{\partial #1}

\newcommand*{\Wert}{\mathord{\mbox{|\kern-1.5pt|\kern-1.5pt|}}}
\newcommand*{\ie}{\mbox{i.e.}\xspace}

\newcommand\numberthis{\addtocounter{equation}{1}\tag{\theequation}}

\newcommand{\N}{\mathbb{N}}	
\newcommand{\R}{\mathbb{R}}	

\newcommand{\Co}{\mathscr{C}}	

\renewcommand{\Vec}{\mathrm{Vec}}	
\newcommand{\Span}{\mathrm{span}}	

\newcommand{\dd}{\,\mathrm{d}}	
\newcommand{\mD}{\mathcal{D}}

\DeclareMathOperator{\dist}{dist}
\DeclareMathOperator{\Modd}{Mod}

\DeclareMathOperator{\diam}{diam}

\DeclareMathOperator{\modulus}{Mod}

\DeclareMathOperator{\lip}{lip}

\DeclareMathOperator{\Vol}{Vol}

\DeclareMathOperator{\esssup}{esssup}
\DeclareMathOperator{\ess}{ess}

\newcommand{\scr}[1]{\mathscr{#1}}

\def\Xint#1{\mathchoice
  {\XXint\displaystyle\textstyle{#1}}%
  {\XXint\textstyle\scriptstyle{#1}}%
  {\XXint\scriptstyle\scriptscriptstyle{#1}}%
  {\XXint\scriptscriptstyle\scriptscriptstyle{#1}}%
  \!\int}
\def\XXint#1#2#3{{\setbox0=\hbox{$#1{#2#3}{\int}$}
  \vcenter{\hbox{$#2#3$}}\kern-.5\wd0}}

\def\dashint{\Xint-}

\dedicatory{Dedicated to Seppo Rickman and Jussi V\"ais\"al\"a on the occasion of their $80^\text{th}$ birthdays}
\makeatother

\usepackage{babel}
\begin{document}

\title[Geometric function theory: the art of pullback factorization]{Geometric function theory: the art of pullback factorization}

\author{Chang-Yu Guo}
\date{\today}
\address[Chang-Yu Guo]{Department of Mathematics and Statistics, University of Jyv\"askyl\"a, P.O. Box 35, FI-40014, Jyv\"askyl\"a, Finland and Department of Mathematics, University of Fribourg, Fribourg, Switzerland}
\email{changyu.guo@unifr.ch}

\author{Marshall Williams}

\subjclass[2010]{53C17,30C65,58C06,58C25}
\keywords{quasiregular mappings, BLD mappings, branched covering, locally linearly connected, Poincar\'e inequality, locally bounded geometry}
\thanks{C.Y.Guo was supported by the Finnish Cultural Foundation--Central Finland Regional Fund No.~30151735, the Academy of Finland No.~131477, and the Swiss National Science Foundation grant No.~153599.}

\begin{abstract}
 In this paper, we develop the foundations of the theory of quasiregular mappings in general metric measure spaces. In particular, nine definitions of quasiregularity for a discrete open mapping with locally bounded multiplicity are proved to be quantitatively equivalent when the metric measure spaces have locally bounded geometry. We also demonstrate that some, though not all, of these implications remain true under fairly general hypotheses. 
 
 The major new tool appeared in our approach is a powerful factorization, termed the pullback factorization, of a quasiregular mapping into the composition of a 1-BLD mapping and a quasiconformal mapping in locally compact complete metric measure spaces. This factorization also brings fundamental new point of view of the theory of quasiregular mappings in Euclidean spaces, in particular, a branched counterpart of quasisymmetric mappings is introduced and is shown to be locally equivalent with quasiregular mappings, quantitatively.
 
 As applications of our new techniques, we answer some well-known open problems in this field and characterize BLD mappings in metric spaces with locally bounded geometry. In particular, a conjecture of Heinonen--Rickman is shown to be false and an open problem of Heinonen--Rickman gets solved affirmatively.

\end{abstract}

\maketitle
\tableofcontents{}

\section{Introduction}\label{sec:introduction}

A continuous mapping $f\colon X\to Y$ between topological spaces is said to be a \textit{branched covering} if $f$ is \textit{discrete}, \textit{open} and \textit{of locally bounded multiplicity}. Recall that $f$ is open if it maps each open set in $X$ to an open set $f(X)$ in $Y$; $f$ is discrete if for each $y\in Y$ the preimage $f^{-1}(y)$ is a discrete subset of $X$; $f$ has locally bounded multiplicity if for each $x\in X$, there exists a neighborhood $U_x$ of $x$ and a positive constant $M_x$ such that $N(f,U_x)\leq M_x<\infty$, where $N(f,U_x):=\sup_{y\in Y}\sharp\{f^{-1}(y)\cap U_x\}$ is the multiplicity of $f$ in $U_x$. It is also clear that $f$ is discrete whenever it has locally bounded multiplicity and when $X$ and $Y$ are \textit{generalized $n$-manifolds}, a continuous, discrete and open mapping $f\colon X\to Y$ will necessarily have locally bounded multiplicity.

For a branched covering $f\colon X\to Y$ between two metric spaces, $x\in X$ and $r>0$, set
\begin{equation*}
H_f(x,r)=\frac{L_f(x,r)}{l_f(x,r)},
\end{equation*}
where
\begin{equation*}
L_f(x,r):=\sup_{y\in Y}{\{d(f(x),f(y)):d(x,y)= r\}},
\end{equation*}
and
\begin{equation*}
l_f(x,r):=\inf_{y\in Y}{\{d(f(x),f(y)):d(x,y)= r\}}.
\end{equation*}
Then the \textit{linear dilatation function} of $f$ at $x$ is defined pointwise by 
\begin{equation*}
H_f(x)=\limsup_{r\to0}H_f(x,r).
\end{equation*}

\begin{definition}\label{def:metric quasiregular map}
A branched covering $f\colon X\to Y$ between two metric measure spaces is termed \textit{metrically $H$-quasiregular} if the linear dilatation function $H_f$ is finite everywhere and essentially bounded from above by $H$. 
\end{definition}

If $f\colon X\to Y$, in Definition~\ref{def:metric quasiregular map}, is additionally assumed to be a homeomorphism, then $f$ is called \textit{metrically $H$-quasiconformal}. We will call $f$ a metrically quasiregular/quasiconformal mapping if it is metrically $H$-quasiregular/quasiconformal for some $H\in [1,\infty)$. Note that our definition of metric quasiconformality allows an exceptional set for the boundedness of the linear dilatation, which is in general weaker than the more commonly used definition in literature, where the linear dilatation is required to be bounded everywhere. However, for mappings between sufficiently regular spaces, for instance, \textit{spaces of locally bounded geometry}, the two definitions coincide. It seems for us more reasonable to require everywhere finiteness and essentially boundedness, rather than everywhere boundedness, for quasiconformal mappings between general metric measure spaces.

The importance of quasiconformal mappings in complex analysis was realized by Ahlfors, Teichm\"uller and Morrey in the 1930s. Ahlfors used quasiconformal
mappings in his geometric approach to Nevanlinna's value distribution theory that earned him one of the first two Fields medals. Teichm\"uller used quasiconformal mappings to measure a distance between two conformally inequivalent compact Riemann surfaces, starting what
is now called Teichm\"uller theory. Morrey proved a powerful existence theorem, called the measurable Riemann mapping theorem, which has had tremendous impact on complex analysis and dynamics, Teichm\"uller theory, and low dimensional
topology, inverse problems and partial differential equations. 

The higher-dimensional theory of quasiconformal mappings was
initiated in earnest by Reshetnyak, Gehring and
V\"ais\"al\"a in the early 1960s~\cite{g62,g63,re89}. The generalisations to non-injective quasiregular mappings was initiated with Reshetnyak, 
and the basic theory was comprehensively laid and significantly advanced
in a sequence of papers from the Finnish school of Martio, Rickman and
V\"ais\"al\"a in the late 1960s~\cite{mrv69,mrv70,mrv71}.

Historically, there are three different definitions of quasiconformality, namely, the metric definition, the analytic definition, and the geometric definition (see Section~\ref{subsec:Definitions of quasireguarity in general metric measure spaces} below for the detailed description). Apparently, the metric definition is of infinitesimal flavor, the analytic definition is a point-wise condition, and the geometric definition is more of global nature. It is a rather deep fact, due to Gehring and V\"ais\"al\"a, that all the three definitions of quasiregularity are quantitatively equivalent. The interplay of all three aspects of quasiconformality/quasiregularity is an
important feature of the theory; see~\cite{bi83,im01,re89,r93,v71} for more information on the Euclidean theory of these mappings. 

In the Euclidean spaces, another remarkable result of Heinonen and Koskela~\cite{hk95} implies that in the definition of metrically quasiconformal mappings, the linear dilation $H_f(x)$ can be replaced with $h_f(x)=\liminf_{r\to 0} H_f(x,r)$, which we will refer as \textit{weak metrically quasiconformal mappings} (see also Section~\ref{subsec:Definitions of quasireguarity in general metric measure spaces} below for the precise formulation). This result was later generalized to mappings between metric spaces with locally bounded geometry in~\cite{bkr07}.

The study of quasiconformal mappings beyond Riemannian spaces was first appeared in the celebrated work of Mostow~\cite{m73} on \textit{strong rigidity of locally symmetric spaces}. The boundary of rank-one symmetric spaces can be identified as certain Carnot groups of step two, and Mostow has developed the basic metrically quasiconformal theory in these groups. Inspired by Mostow's work, Pansu~\cite{p89} used the theory of quasiconformal mappings to study \textit{quasi-isometries of rank-one symmetric spaces}. In particular, he has shown that the metrically quasiconformal mappings are absolutely continuous on almost every lines. The systematic study of metrically quasiconformal mappings on the Heisenberg group was later done by Kor\'anyi and Reimann~\cite{kr85,kr95}. Margulis and Mostow~\cite{mm95} studied the absolute continuity of metrically quasiconformal mappings along horizontal lines in the setting of \textit{equiregular subRiemannian manifolds} and proved that metrically quasiconformal mappings between two equiregular subRiemannian manifolds are $P$-differentiable\footnote{$P$-differentiability was refered as cc-differentiability in their paper} almost everywhere.
By the break-through work of Heinonen and Koskela~\cite{hk98}, a full-fledged metrically quasiconformal mappping theory exists in rather general metric measure spaces. In particular, the equivalence of metrically quasiconformal mappings and \textit{quasisymmetric mappings} was derived under very mild assumptions.  This theory has subsequently been applied to new rigidity studies in geometric group theory, geometric topology and geometric paratrization of metric spaces; see for instance~\cite{bk02,bk05,hr02,hk11,r14} and the references therein. This theory also initiated a new way of looking at \textit{weakly differentiable mappings} between non-smooth metric measure spaces. Based on this theory, in~\cite{hkst01}, the Sobolev class of Banach space valued mappings was studied and several characterizations of quasiconformal mappings between metric spaces of locally bounded geometry were established. In particular, the equivalences of all the three definitions of quasiconformality were proved in their setting; see also~\cite{t98,t00,t01}. 

The study of quasiregular mappings beyond Riemannian spaces was initiated by Heinonen and Holopainen~\cite{hh97}. The importance of the study of quasiregular mappings beyond the Riemannian spaces was recognized later in the fundamental paper of Heinonen and Rickman~\cite{hr02} as well as in the sequential remarkable papers of Heinonen and Sullivan~\cite{hs02} and Heinonen and Keith~\cite{hk11}. Initiated by these works, there are  some recent advances of the theory of quasiregular mappings in various other settings; see for instance~\cite{or09,gnw15,gl15,a15}.

The main obstacle in establishing the theory of quasiregular mappings in general metric spaces lies in the \textit{branch set} $\mathcal{B}_f$, i.e., the set of points in $X$ where $f\colon X\to Y$ fails to be a local homeomorphism. The difficulty is somehow hidden in the Euclidean planar case, as the celebrated \textit{Stoilow factorization theorem} (cf.~\cite{aim09}) asserts that a quasiregular mapping $f\colon \Omega\to \R^2$
admits a factorization $f=\varphi\circ g$, where $g\colon \Omega\to g(\Omega)$ is quasiconformal and $\varphi\colon g(\Omega)\to \R^2$ is analytic. This factorization, together with our more or less complete understanding of the structure of analytic functions in the plane, connects quasiregular and quasiconformal mappings strongly. In particular, the branch set $\mathcal{B}_f$ of $f$ must be discrete. In higher dimensions or more general metric measure spaces, the branch set of a quasiregular mappings can be very wild. This makes the homeomorphic theory and the non-homeomorphic theory substantially different. Indeed,  the most delicate part of establishing the theory of quasiregular mappings in various settings as mentioned above is to show that the branch set and its image have null measure. For instance, in the very recent paper~\cite{flp14}, in order to derive the basic properties of quasiregular mappings from that of quasiconformal mappings in equiregular subRiemannian manifolds, it was assumed a priori in the definition of a quasiregular mapping that both the branch set and its image are null sets with respect to the underlying Hausdorff measure. 

Somewhat surprisingly, in our previous paper~\cite{gw15}, we have successfully generalized the well-known Bonk--Heinonen--Sarvas Theorem, regarding the \textit{quantitative porosity} of the branch set and its image of a quasiregular mapping, to a large class of metric spaces.  One important point there was that we had an argument to prove the porosity results on $\mathcal{B}_f$ and $f(\mathcal{B}_f)$ directly, without knowing a priori the deep analytic facts (such as differentiability, Condition $N$ and Condition $N^{-1}$) of quasiregular mappings, so that many basic properties (in particular those local properties) follow directly from that of quasiconformal mappings. In other words, we have provided a quick approach to establish the basic properties of quasiregular mappings from that of quasiconformal mappings. The result we have obtained also promotes the important \textit{V\"ais\"al\"a's inequality} in most general settings. On the other hand, to obtain the preceding null property of the branch set and its image with respect to the given measures, one necessarily has to impose certain extra assumptions from \textit{quantitative topology} for the underlying metric measure spaces that we do want to dispose of. Indeed, our prime motivation of this paper is to prove the mutual implications of different definitions of quasiregularity in the most general metric space setting.

Our starting point is to find a powerful \textit{factorization} result, similar as the Stoilow factorization theorem in the plane, for quasiregular mappings in general metric measure spaces, so that the theory of quasiconformal mappings can be employed. Fortunately, such kind of factorization do exist and we will call it the \textit{pullback factorization} (see~Section~\ref{sec:the pullback factorization} below). More precisely, given a quasiregular mapping $f\colon X\to Y$ between two locally compact complete metric measure spaces, there exists a canonical factorization $f=\pi\circ g$, where $g\colon X\to X^f$ is a quasiconformal mapping and $\pi\colon X^f\to Y$ is a \textit{mapping of 1-bounded length distortion} (1-BLD for short). Due to the fine properties of the projection mapping $\pi$, $f$ and $g$ share many common analytic and geometric properties. Indeed, one of the key steps in our main equivalence theorems in Section~\ref{sec:foundations of QR mappings in mms} is to characterize the \textit{geometric $K_O$-inequality} and \textit{$K_I$-inequality} (also known as the \textit{Poletsky's inequality}) of $f$ via the corresponding inequalities of $g$.

There are two new observations on the theory of quasiregular mappings, even in Euclidean spaces, out from the pullback factorization. First, there are two natural groups, namely, the \textit{forward group} and the \textit{inverse group}, each consists of four different definitions of quasiregularity. To gain some intuition of what is going on here, let us, for simplicity, only look at the quasiconformal case. The forward group consists of the following four  definitions: weak metric definition ($m$), metric definition $(M)$, analytic definition ($A$), and the $K_O$-inequality ($G$). While the inverse group consists of the ``inverse" definitions: the weak inverse metric definition $(m^*)$, the inverse metric definition ($M^*$), the inverse analytic definition ($A^*$), and the $K_I$-inequality ($G^*$). It is useful to observe that the inverse group is exactly the forward group for $f^{-1}$. From the technical point of view, the group in this way makes the implications clear, namely, $(A)\Longleftrightarrow (G)$ and $(A^*)\Longleftrightarrow (G^*)$ in arbitrary metric measure spaces, a stronger version of $(m)$ or $(m^*)$ will quantitatively imply both $(A)$ and $(A^*)$ in Ahlfors regular spaces, and all these eight definitions will be quantitatively equivalent when the metric measure spaces have locally bounded geometry. We will use exactly the same groups in the quasiregular case, but some care need to pay to define the inverse analytic definition ($A^*$) of $f$ due to the branching. This will be done with the help of the pullback factorization. The novel here is that the equivalence of $(A)^*$ and $G^*$ provides an analytic characterization of the Poletsky's inequality, which together with $K_O$-inequality and the assumption $\nu(f(\mathcal{B}_f))=0$ will imply the stronger V\"ais\"al\"a's inequality (see~Theorem~\ref{thm:update poletsky to vaisala}).

Secondly, in Ahlfors regular spaces, for the implication of (a stronger form of) $(m)$ or $(m^*)$ to $(A)$ and $(A^*)$, we \textit{do not} really need the fact that the branch set (and its image) are null with respect to the Ahlfors regular measures, but that for the \textit{essential branch set}, which is a subset of the branch set and consists of all points with \textit{local essential index} strictly larger than one. If the underlying metric measure spaces are \textit{doubling}, then one can show that the essential branch set of $f$ is null with respect to certain \textit{pullback measure}, which is sufficient to conclude our implications in Ahlfors regular spaces. 

\subsection{Main results}

Our first main result concerns the relations between all the different definitions of quasiregularity. It can be regarded as a complete analog of that of quasiconformal mappings~\cite{hkst01,bkr07}.

\begin{thmA}\label{thm:thmA}
	Let $f\colon X\to Y$ be a onto branched covering between two metric measure spaces and $Q>1$. Then the following conclusions hold:
	\begin{itemize}
		\item[i).] $f$ is analytically $K_O$-quasiregular with exponent $Q$ if and only if it is geometrically $K_O$-quasiregular with exponent $Q$. Similarly, $f$ is inverse analytically $K_I$-quasiregular with exponent $Q$ if and only if it is strong inverse geometrically $K_I$-quasiregular with exponent $Q$.
		
		\item[ii).] If $X$ and $Y$ are locally Ahlfors $Q$-regular, and $Y$ has $c$-bounded turning, then either of the following two conditions
		\begin{itemize}
			\item[a).] $h_f(x)\leq h$ for all $x\in X$;
			\item[b).] $h_f^*(x)\leq h$ for all $x\in X$,
		\end{itemize}
		implies that $f$ is analytically $K_O$-quasiregular with exponent $Q$ and inverse analytically $K_I$-quasiregular with exponent $Q$, with both constants $K_O$ and $K_I$ depending only on the constant of Ahlfors $Q$-regularity, and on $c$ and $h$.
		
		\item[iii).] If both $X$ and $Y$ have locally $Q$-bounded geometry, then all of the metric, geometric and analytic definitions are quantitatively equivalent.
		
		\item[iv).] If $f$ satisfies the $K_I$-inequality with exponent $Q$ and if $\nu(f(\mathcal{B}_f))=0$, then $f$ satisfies the V\"ais\"al\"a's inequality with exponent $Q$ and with the same constant.
	\end{itemize}
\end{thmA}

A fundamental observation that we have made use of, when trying to move from the metric definition to the analytic one, is that the use of \textit{balls} in the typical covering argument (see e.g.~\cite{hk98,bkr07,w14}) to contruct an appropriate upper gradient is \emph{not essential}. Alternatively, one can use certain \emph{normal neighborhoods}, which are obtained as preimages of balls from the target, to run a similar covering type argument. This observation naturally motivates the study of pullback geometry/topology as we have developed in Section~\ref{sec:basic pullback studies}, which plays a crucial role in the discover of the pullback factorization.

Another point we would like to point out here is that based on the pullback factorization, we have introduced the class of \emph{branched quasisymmetric mappings}, which is branched version of the well-known quasisymmetric mappings. As in the case of quasiconformal mappings, we have shown in Section~\ref{subsec:Branched quasisymmetric mappings} that quasiregular mappings are locally branched quasisymmetric, quantitatively, in metric spaces of bounded geometry.  

As the first main application of the pullback factorization, we get our second main result of this paper (see Section~\ref{sec:the pullback factorization} below for the definition of BDD branched coverings). 

\begin{thmA}\label{thm:thmB}
	Let $f\colon X\to Y$ be a BDD branched covering, with $N=N(f,X)<\infty$, and suppose that $Y$ is doubling and that $X$ and $Y$ have bounded turning. Then there is a bi-Lipschitz embedding of $X$ into $Y\times \R^{c_d(N-1)}$, where $c_d$ depends only on the data of $Y$.
\end{thmA}

An immediate corollary of Theorem~\ref{thm:thmB} (cf.~Corollary~\ref{coro:biLip embedding II}) is that every BLD covering of a locally compact quasiconvex metric space over $\R^n$ with finite maximal multiplicity embeds bi-Lipschitzly into some Euclidean space. This answers in the affirmative an open question of Heinonen and Rickman~\cite{hr02}. With the same technique, we also construct in Section~\ref{subsec:some natural examples} some natural examples of BLD mappings (even $1$-BLD), between sufficiently nice metric spaces, with positive branch sets (as well as their images). This shows the sharpness of the main results of~\cite{gw15} and also disproves a conjecture of Heinonen and Rickman~\cite{hr02}.

 As another application of the pullback techniques, we fully characterize BLD mappings, via quasiregular mappings, in metric spaces with (locally) bounded geometry.
 
\begin{thmA}\label{thm:thmC}
	Let $f\colon X\to Y$ be a branched covering between two Ahlfors $Q$-regular $Q$-Loewner spaces wtih $Q>1$. Then the following statements are equivalent:
	
	1). $f$ is $L$-BLD;
	
	2). For each $x\in X$, there exists $r_x>0$ such that 
	\begin{align*}
	\frac{d(x,y)}{c}\leq d(f(x),f(y))\leq cd(x,y)
	\end{align*}
	for all $y\in B(x,r_x)$;
	
	3). $L_f(x)\leq c$ and $l_f(x)\geq \frac{1}{c}$ for each $x\in X$;
	
	4). $f$ is metrically $H$-quasiregular, locally $M$-Lipschitz, and $J_f(x)\geq c$ for a.e. $x\in X$.
	
	Moreover, all the constants involved depend quantitatively only on each other and on the data associated to $X$ and $Y$. 
\end{thmA}

As a consequence of Theorem~\ref{thm:thmC}, we obtain the important compactness properties of BLD mappings between certain generalized manifolds (see~Theorem~\ref{thm:convergence result}).

\subsection{Structure of the paper}
After the introduction, we collect some basic definitions and notation in Section~\ref{sec:preliminaries}.

In Section~\ref{sec:basic pullback studies}, we study some basic properties of pullback geometry. The main task is to introduce the essential local index and prove some basic estimates on the size of the essential branch set.

Section~\ref{sec:the pullback factorization} is the heart of this paper, from the point of view of techniques. We first present the general construction of the pullback factorization for branched coverings between locally compact complete metric measure spaces. Then we introduce the so-called \textit{mappings of bounded diameter distortion} (BDD for short) and  investigate the nice properties of the projection mapping $\pi$ and also the \textit{pullback metric} in details. Some basic properties between our branched covering $f$ and the lifting mapping $g$ are also studied.

Section~\ref{sec:Sobolev spaces on metric measure spaces} contains a brief introduction of the theory of Sobolev spaces on metric measure spaces based on the \textit{upper gradients}. 

The main body of this paper lies in Section~\ref{sec:foundations of QR mappings in mms}, where we establish the foundations of the theory of quasiregular mappings in general metric measure spaces and prove our first main result Theorem~\ref{thm:thmA}. More precisely, we introduce nine different definitions of quasiregularity in Section~\ref{subsec:Definitions of quasireguarity in general metric measure spaces} and then show in Section~\ref{subsec:Analytic and geometric definitions} the equivalences of $(A)$ and $(G)$ (and also $(A^*)$ and $(G^*)$) in arbitrary metric measure spaces. Section~\ref{subsec:Ahlfors regularity and the metric definitions} is the most technical part of the whole Section~\ref{sec:foundations of QR mappings in mms}, where we show that when the metric measure spaces are Ahlfors regular with the target space having bounded turning, then the (everywhere finite) weak metric/metric* definitions of quasiregularity implies both the analytic and the inverse analytic definitions, quantitatively. The key for our argument lies in Theorem~\ref{thm:for analytic regularity}, which essentially dates back to Balogh--Koskela--Rogovin~\cite{bkr07}. The principle idea is that the (weak) metric dilatation gives good control of balls with small radius that is sufficient, via a standard covering argument similar as~\cite{bkr07}, to construct appropriate weak upper gradients for the mapping itself in Ahlfors regular spaces. Moreover, one can quantifies the precise bound, which will result in the analytic or inverse analytic definitions of quasiregularity. One technical point here is that we need to run Theorem~\ref{thm:for analytic regularity} twice to get rid of the dependences   on the local essential index for the quasiregularity constant. Finally, we show the equivalences of the nine definitions of quasiregularity in metric measure spaces with locally bounded geometry in Section~\ref{subsec:Spaces of locally $Q$-bounded geometry}. Section~\ref{subsec:Size of the branch set and V\"ais\"al\"a's inequality} contains a brief discussion on the size of the branch set, mainly the results from~\cite{gw15}, and the V\"ais\"al\"a's inequality. In particular, we show that when $f$ satisfies the Poletsky's inequality, and the image of the branch set has zero measure, then $f$ satisfies the V\"ais\"al\"a's inequality. In Section~\ref{subsec:Geometric modulus inequalities for mappings of finite linear dilatation}, we study homeomorphisms with finite (weak) linear dilatation and derive some weighted $K_O$ or $K_I$-inequalities for such mappings. In the final section, Section~\ref{subsec:Branched quasisymmetric mappings}, we introduce the class of \textit{branched quasisymmetric mappings} and prove that metrically quasiregular mappings are locally branched quasisymmetric, quantitatively, in spaces of locally bounded geometry.

Section~\ref{sec:Quasiregular mappings between subRiemannian manifolds} contains a brief overview of the foundations of quasiregular mappings in equiregular subRiemannian manifolds. We separate the V\"ais\"al\"a's inequality as a typical application of our main results from Section~\ref{sec:foundations of QR mappings in mms}.

We prove our most important applications of the pullback techniques, Theorem~\ref{thm:thmB}, in Section~\ref{sec:biLipschitz embeddability}. We also use our preceding techniques to construct natural examples of BLD mappings between very nice metric spaces with positive branch set.

The proof of our second application, Theorem~\ref{thm:thmC}, is given in Section~\ref{sec:characterization of BLD mappings}, where we also prove a compactness results for BLD mappings in great generality.

\subsection{An important clarification from the first-named author}
The new tool, namely, the pullback factorization, as well as the main results (in Sections~\ref{sec:foundations of QR mappings in mms} and Section~\ref{sec:biLipschitz embeddability}) in this paper were essentially obtained in December 2014 by the second-named author in two short notes~\cite{w15,w15b}. The first-named author essentially only combined these notes together in a more detailed and readable manner, together with our joint paper~\cite{gw16} and a few natural applications of these techniques. Thus he wants to point out the credits of this paper should be given to the second-named author.

\subsection*{Acknowledgements}  
Some of this research was conducted by the second-named author at the Institute for Pure and Applied Mathematics during the program ``Interactions Between Analysis and Geometry" 2013.

\newpage
\section{Preliminaries}\label{sec:preliminaries}

\subsection{Generalized manifolds}\label{subsec:generalized manifolds}
Let $H_c^*(X)$ denote the Alexander-Spanier cohomology groups of a space $X$ with compact supports and coefficients in $\mathbb{Z}$. 

\begin{definition}\label{def:cohomology manifold}
	A space $X$ is called an $n$-dimensional, $n\geq 2$, \textit{cohomology manifold} (over $\mathbb{Z}$), or a \textit{cohomology $n$-manifold} if
	\begin{description}
		\item[(a)] the cohomological dimension $\dim_{\mathbb{Z}}X$ is at most $n$, and
		\item[(b)] the local cohomology groups of $X$ are equivalent to $\mathbb{Z}$ in degree $n$ and to zero in degree $n-1$.
	\end{description}
\end{definition}
Condition (a) means that $H_c^p(U)=0$ for all open $U\subset X$ and $p\geq n+1$. Condition (b) means that for each point $x\in X$, and for each open neighborhood $U$ of $x$, there is another open neighborhood $V$ of $x$ contained in $U$ such that
\begin{equation*}
H_c^p(V)=
\begin{cases}
\mathbb{Z} & \text{if } p=n \\
0 & \text{if } p=n-1,
\end{cases}
\end{equation*}
and the standard homomorphism
\begin{equation}\label{eq:standard homomorphism}
H_c^n(W)\to H_c^n(V)
\end{equation}
is a surjection whenever $W$ is an open neighborhood of $x$ contained in $V$. As for examples of cohomology $n$-manifolds, we point out all topological $n$-manifolds are cohomology $n$-manifolds. More examples can be found in~\cite{hr02}.

\begin{definition}\label{def:generalized manifold}
	A space $X$ is called a \textit{generalized $n$-manifold}, $n\geq 2$, if it is a finite-dimensional cohomology $n$-manifold.
\end{definition}

If a generalized $n$-manifold $X$ satisfies $H_c^n(X)\simeq\mathbb{Z}$, then $X$ is said to be orientable and a choice of a generator $g_X$ in $H_c^n(X)$ is called an orientation; $X$ together with $g_X$ is an oriented generalized $n$-manifold. If $X$ is oriented, we can simultaneously choose an orientation $g_U$ for all connected open subsets $U$ of $X$ via the isomorphisms
$$H_c^n(U)\rightarrow H_c^n(U).$$

Let $X$ and $Y$ be oriented generalized $n$-manifolds, $\Omega\subset X$ be an oriented domain and let $f\colon \Omega\to Y$ be continuous. For each domain $D\subset\subset\Omega$ and for each component $V$ of $Y\backslash f(\bdary D)$, the map
$$f|_{f^{-1}(V)\cap D}:f^{-1}(V)\cap D\to V$$
is proper. Hence we have a sequence of maps
\begin{equation}\label{eq:hr 2.1}
H_c^n(V)\to H_c^n(f^{-1}(V)\cap D)\to H_c^n(D),
\end{equation}
where the first map is induced by $f$ and the second map is the standard homomorphism. The composition of these two maps sends the generator $g_V$ to an integer multiple of the generator $g_D$; this integer, denoted by $\mu(y,f,D)$, is called the \textit{local degree of $f$ at a point $y\in V$ with respect to $D$}. The local degree is an integer-valued locally constant function
$$y\mapsto \mu(y,f,D)$$
defined in $Y\backslash f(\bdary D)$. If $V\cap f(D)=\emptyset$, then $\mu(y,f,D)=0$ for all $y\in V$.

\begin{definition}\label{def:sense-preserving}
	A continuous map $f\colon X\to Y$ between two oriented generalized $n$-manifolds is said to be \textit{sense-preserving} if
	$$\mu(y,f,D)>0$$
	whenever $D\subset\subset X$ is a domain and $y\in f(D)\backslash f(\bdary D)$.
\end{definition}
%
%
%
%
%

\subsection{Hausdorff measure}
Let $X=(X,d)$ be a metric space. Fix a positive real number $s$. For each $\delta>0$ and $E\subset X$, we set 
\begin{align*}
\mathcal{H}_{s,\delta}=\inf \sum_{i}\big(\diam(E_i)\big)^s,
\end{align*}
where the infimum is taken over all countable covers of $E$ by sets $E_i\subset X$ with diameter no more than $\delta$. When $\delta$ decreases, the value of $\mathcal{H}_{s,\delta}(E)$ for a fixed set $E$ increases, and the \emph{$s$-dimensional Hausdorff measure} of $E$ is defined to be
\begin{align}\label{eq:def of Hausdorff measure}
\mathcal{H}^{s}(E):=\lim_{\delta\to 0}\mathcal{H}_{s,\delta}(E).
\end{align}
The set function $E\mapsto \mathcal{H}^{s}(E)$ determins a Borel regular measure on $X$. In the later part of this paper, we also use the notation $\mathscr{H}^s(E)$ for the $s$-dimensional Hausdorff measure of a set $E$ in a metric space $X$.

The \emph{Hausdorff dimension} of a set $E$ in $X$ is the infimum of the numbers $s>0$ such that $\mathcal{H}^s(E)=0$.

\subsection{Metric measure spaces}\label{subsec:metric measure spaces}

\begin{definition}\label{def:metric measure space}
	A \textit{metric measure space} is defined to be a triple $(X,d,\mu)$, where $(X,d)$ is a \emph{separable} metric space and $\mu$ is a \emph{nontrivial locally finite Borel regular measure} on $X$.
\end{definition}

\begin{definition}\label{def:doubling metric measure space}
	A Borel regular measure $\mu$ on a metric space $(X,d)$ is called a \textit{doubling measure} if every ball in $X$ has positive and finite measure and there exists a constant $C_\mu\geq 1$ such that
	\begin{equation}\label{eq:doubling measure}
	\mu(B(x,2r))\leq C_\mu \mu(B(x,r))
	\end{equation}
	for each $x\in X$ and $r>0$. We call the triple $(X,d,\mu)$ a doubling metric measure space if $\mu$ is a doubling measure on $X$. We call $(X,d,\mu)$ an Ahlfors $Q$-regular space, $1\leq Q<\infty$, if there exists a constant $C\geq 1$ such that
	\begin{equation}\label{eq:Ahlfors regular measure}
	C^{-1}r^Q\leq \mu(B(x,r))\leq Cr^Q
	\end{equation}
	for all balls $B(x,r)\subset X$ of radius $r<\diam X$.
\end{definition}

It is well-known that if $(X,d,\mu)$ is an Ahlfors $Q$-regular space, then
\begin{equation}\label{eq: comparable with Hausdorff measure}
\mu(E)\approx \mathscr{H}^Q(E)
\end{equation}
for all Borel sets $E$ in $X$, see~\cite[Chapter 8]{h01}.

\subsection{Local contractibility and local metric orientation}\label{subsec:local contractibility and metric orientation}
Recall that  a metric space $X$ is said to be \emph{linearly locally contractible} if there exits a constant $c\geq 1$ such that every ball $B(x,r)$ is contractible in $B(x,cr)$.  $X$ is \textit{locally linearly locally contractible} if for each $x\in X$, there exists a neighborhood $U_x$ of $x$ such that $U_x$ is $c_x$-linearly locally contractible.

%
%
%
Assume that $X$ is an oriented generalized $n$-manifold, and assume that $X$ satisfy the following conditions:
\begin{itemize}
	\item X is $n$-rectifiable and has locally finite $\mathscr{H}^n$-measure;
	\item X is Ahlfors $n$-regular;
	\item X is locally bi-Lipschitz embeddable in Euclidean space.
\end{itemize}
Let $U$ be an open connected neighborhood of a point in $X$ that can be embedded bi-Lipschitz in some $\bR^N$ and that $U$ has finite Hausdorff $n$-measure. Because of properties $(A1)$ and $(A2)$, the set $U$ has a tangent $n$-plane $T_xU$ at a.e. point $x\in U$. The collection of these planes is called a measurable tangent bundle of $U$, and it is denoted by $TU$.

Each $n$-plane $T_xU$, whenever it exists, is an $n$-dimensional subspace of $\bR^N$, and a measurable choice of orientations $\xi=(\xi_x)$ on each $T_xU$ is called an orientation of the tangent bundle $TU$. Such orientations always exist. Because $X$ is an oriented generalized $n$-manifold, there is another orientation on $U$, provided $U$ is connected; this is a generator $g_U$ in the group $H_c^n(U)=\mathbb{Z}$ determined by the fixed orientation on $X$.

Fix a point $x\in U$ such that $T_xU$ exists. Then the projection
$$\pi_x:\bR^N\to T_xU+x$$
satisfies $x\notin \pi_x(\bdary D)$ whenever $D$ is a sufficiently small open connected neighborhood of $x$ in $U$. Moreover, for any $\varepsilon>0$,
\begin{equation}\label{eq: good property of projection map}
	\frac{|x-x_0|}{2}\leq \pi_{x_0}(x-x_0)\leq \varepsilon |x-x_0|
\end{equation}
whenever $|x-x_0|$ small enough, see e.g.~\cite{hr02,g14}. Thus, if $V$ is the $x$-component of $(T_xU+x)\backslash \pi_x(\bdary D)$, we have
\begin{equation}\label{eq:canonical isomorphism}
	H_c^n(T_xU)\leftarrow H_c^n(V)\stackrel{\pi_x^*}{\rightarrow}H_c^n(\pi_x^{-1}(V)\cap D)\rightarrow H_c^n(D)\rightarrow H_c^n(U),
\end{equation}
where the unnamed arrows represent a canonical isomorphism. If the orientations $\xi_U$ and $g_U$ correspond to each other under the map in~\eqref{eq:canonical isomorphism}, we say that $T_xU$ and $U$ are \textit{coherently oriented at $x$} by $\xi_x$ and $g_U$; if a measurable coherent orientation $\xi=(\xi_x)$ is chosen at a.e. point, we say that $TU$ is \textit{metrically oriented} by $\xi$ and $g_U$.

%

\subsection{Inverse dilatation}\label{subsec:Inverse dilatation}
Let $f\colon X\rightarrow Y$ be continuous.  For each $x\in X$, denote by $U(x,r)$ the component of $x$ in $f^{-1}(B(f(x),r))$.  

Set 
\begin{align*}
H_f^*(x,s)=\frac{L_f^*(x,s)}{l_f^*(x,s)},
\end{align*}
where
\begin{align*}
L_f^*(x,s)=\sup_{z\in \partial U(x,s)}d(x,z)\quad \text{and}\quad l_f^*(x,s)=\inf_{z\in \partial U(x,s)}d(x,z).
\end{align*}
The \textit{inverse linear dilatation function} of $f$ at $x$ is defined pointwise by
\begin{align*}
H_f^*(x)=\limsup_{s\to 0}H_f^*(x,s).
\end{align*}
Similarly, the \emph{weak inverse linear dilatation function} of $f$ at $x$ is defined pointwise by
\begin{align*}
h_f^*(x)=\liminf_{s\to 0}H_f^*(x,s).
\end{align*}

\subsection{Condition $N$ and $N^{-1}$}\label{subsec:Condition N and N inverse}
A mapping $f\colon (X,\mu)\to (Y,\nu)$ between two measure spaces is said to satisfy \textit{Condition $N$} if $\nu(f(E))=0$ whenever $E$ is a subset of $X$ with $\mu(E)=0$. Similarly, $f$ satisfies \textit{Condition $N^{-1}$} if $\nu(f(E))>0$ whenever $E$ is a subset of $X$ with $\mu(E)>0$.

\newpage
\section{Sobolev spaces on metric measure spaces}\label{sec:Sobolev spaces on metric measure spaces}

A main theme in analysis on metric spaces is that the infinitesimal structure
of a metric space can be understood via the curves that it contains. The
reason behind this is that we can integrate Borel functions along rectifiable
curves and do certain non-smooth calculus akin to the Euclidean spaces. The notion of upper gradients becomes an important tool in understanding these non-smooth calculus, particularly, when the underlying metric spaces have ``many" rectifiable curves.

In this section, we will briefly introduce the theory of Sobolev spaces on metric measure spaces based on the upper gradient approach. For detailed description of this approach, see the monographs~\cite{hkst15,s00}.

\subsection{Modulus of a curve family}\label{subsec:Modulus of a curve family}

Let $(X,d)$ be a metric space. A curve (or path) in $X$ is a continuous map $\gamma\colon I\to X$, where $I\subset \bR$ is an interval. We call $\gamma$ compact, open, or half-open, depending on the type of the interval $I$.

Given a compact curve $\gamma\colon [a,b]\to X$, we define the \textit{variation function} $v_\gamma\colon [a,b]\to [a,b]\to \infty$ by
\begin{align*}
	v_\gamma(s)=\sup_{a\leq a_1\leq b_1\leq \cdots\leq a_n\leq b_n\leq s}\sum_{i=1}^n d(\gamma(b_i),\gamma(a_i)).
\end{align*}
The length of $\gamma$ is defined to be the variation $v_\gamma(b)$ at the end point $b$ of the parametrizing interval $[a,b]$. If $\gamma$ is not compact, its length is defined to be the supremum of the lengths of the compact subcurves of $\gamma$. Thus, every curve has a well defined length in the extended nonnegative reals, and we denote it by length$(\gamma)$ or simply $l(\gamma)$.

A curve is said to be \textit{rectifiable} if its length $l(\gamma)$ is finite, and \textit{locally rectifiable} if each of its compact subcurves is rectifiable. For any rectifiable curve $\gamma$ there are its associated length function $s_\gamma\colon I\to [0,l(\gamma)]$ and a unique 1-Lipschitz map $\gamma_s\colon [0,l(\gamma)]\to X$ such that $\gamma=\gamma_s\circ s_\gamma$. The curve $\gamma_s$ is the \textit{arc length parametrization} of $\gamma$.

When $\gamma$ is rectifiable, and parametrized by arclength on the interval $[a,b]$, the integral of a Borel function $\rho\colon X\to [0,\infty]$ along  $\gamma$ is
\[ \int_\gamma \rho \,ds = \int_0^{l(\gamma)} \rho(\gamma_s(t))\,dt\text{.}\]
Similarly, the line integral of a  Borel function $\rho\colon X\to [0,\infty]$ over a locally rectifiable curve $\gamma$ is defined to be the supremum of the integral of $\rho$ over all compact subcurves of $\gamma$.

A curve $\gamma$ is \textit{absolutely continuous} if $v_\gamma$ is absolutely continuous. Via the chain rule, we then have
\begin{align}\label{eq:formula for ABS curve line integral}
	\int_{\gamma}\rho ds=\int_a^b\rho(\gamma(t))v_\gamma'(t)dt.
\end{align}

Let $X=(X,d,\mu)$ be a metric measure space as defined in Definition~\ref{def:metric measure space}. Let $\Gamma$ a family of curves in $X$.  A Borel function $\rho\colon X\rightarrow [0,\infty]$ is \textit{admissible} for $\Gamma$ if for every locally rectifiable curves $\gamma\in \Gamma$,
\begin{equation}\label{admissibility}
	\int_\gamma \rho\,ds\geq 1\text{.}
\end{equation}
The \textit{$p$-modulus} of $\Gamma$, $p\geq 1$, is defined as
\begin{equation*}
	\modulus_p(\Gamma) = \inf_{\rho} \left\{ \int_X \rho^p\,d\mu:\text{$\rho$ is admissible for $\Gamma$} \right\}.
\end{equation*}
A family of curves is called \emph{$p$-exceptional} if it has $p$-modulus zero. We say that a property of curves holds for \emph{$p$-almost every curve} if the collection of curves for which the property fails to hold is $p$-exceptional.

\subsection{Sobolev spaces based on upper gradients}
Let $X=(X,d,\mu)$ be a metric measure space and $Z=(Z,d_Z)$ be a metric space.

\begin{definition}\label{def:upper gradient}
	A Borel function $g\colon X\rightarrow [0,\infty]$ is called an \textit{upper gradient} for a map $f\colon X\to Z$ if for every rectifiable curve $\gamma\colon [a,b]\to X$, we have the inequality
	\begin{equation}\label{ugdefeq}
		\int_\gamma g\,ds\geq d_Z(f(\gamma(b)),f(\gamma(a)))\text{.}
	\end{equation}
	If inequality \eqref{ugdefeq} merely holds for $p$-almost every compact curve, then $g$ is called a \textit{$p$-weak upper gradient} for $f$.  When the exponent $p$ is clear, we omit it.
\end{definition}

The concept of upper gradient were introduced in~\cite{hk98}. It was initially called \textit{``very weak gradient"}, but the befitting term ``upper gradient" was soon suggested. Functions with $p$-integrable $p$-weak upper gradients were subsequently studied in~\cite{km98}, while the theory of Sobolev spaces based on upper gradient was systematically developed in~\cite{s00} and~\cite{hkst15}.

By \cite[Lemma 5.2.3]{hkst15}, $f$ has a $p$-weak upper gradient in $L^p_{loc}(X)$ if and only if it has an actual upper gradient in $L^p_{loc}(X)$.

A $p$-weak upper gradient $g$ of $f$ is \textit{minimal} if for every $p$-weak upper gradient $\tilde{g}$ of $f$, $\tilde{g}\geq g$ $\mu$-almost everywhere.  If $f$ has an upper gradient in $L^p_\loc(X)$, then $f$ has a unique (up to sets of $\mu$-measure zero) minimal $p$-weak upper gradient by the following result from~\cite[Theorem 5.3.23]{hkst15}.  In this situation, we denote the minimal upper gradient by $g_{f}$.

\begin{proposition}\label{prop:hkst Thm 5.3.23}
	The collection of all $p$-integrable $p$-weak upper gradients of a map $u\colon X\to Z$ is a closed convex lattice inside $L^p(X)$ and, if nonempty, contains a unique element of smallest $L^p$-norm. In particular, if a map has a $p$-integrable $p$-weak upper gradient, then it has a minimal $p$-weak upper gradient.
\end{proposition}

In view of the above result, the minimal $p$-weak upper gradient $g_u$ should be thought of as a substitute for $|\nabla u|$, or the length of a gradient, for functions defined in metric measure spaces.

Fix a Banach space $\mathbb{V}$, and we first define the Sobolev space $N^{1,p}(X,\mathbb{V})$ of $\mathbb{V}$-valued mappings. Let $\tilde{N}^{1,p}(X,\mathbb{V})$ denote the collection of all maps $f\in L^p(X,\mathbb{V})$ that have an upper gradient in $L^p(X)$. We equip it with seminorm
\begin{equation*}
	\|f\|_{\tilde{N}^{1,p}(X,\mathbb{V})}=\|f\|_{L^p(X,\mathbb{V})}+\|g_f\|_{L^p(X)},
\end{equation*}
where $g_f$ is the minimal $p$-weak upper gradient of $f$. We obtain a normed space $N^{1,p}(X,\mathbb{V})$ by passing to equivalence classes of functions in $\tilde{N}^{1,p}(X,\mathbb{V})$ with respect to equivalence relation: $f_1\sim f_2$ if $\|f_1-f_2\|_{\tilde{N}^{1,p}(X,\mathbb{V})}=0$. Thus
\begin{equation}\label{eq:definition of Nowton-Sobolev space}
	N^{1,p}(X,\mathbb{V}):=\tilde{N}^{1,p}(X,\mathbb{V})/\{f\in \tilde{N}^{1,p}(X,\mathbb{V}): \|f\|_{\tilde{N}^{1,p}(X,\mathbb{V})}=0\}.
\end{equation}

Let $\tilde{N}_{loc}^{1,p}(X,\mathbb{V})$ be the vector space of functions $f\colon X\to \mathbb{V}$ with the property that every point $x\in X$ has a neighborhood $U_x$ in $X$ such that $f\in \tilde{N}^{1,p}(U_x,\mathbb{V})$. Two functions $f_1$ and $f_2$ in $\tilde{N}_{loc}^{1,p}(X,\mathbb{V})$ are said to be equivalent if every point $x\in X$ has a neighborhood $U_x$ in $X$ such that the restrictions $f_1|_{U_x}$ and $f_2|_{U_x}$ determine the same element in $\tilde{N}^{1,p}(U_x,\mathbb{V})$. The local Sobolev space $N_{loc}^{1,p}(X,\mathbb{V})$ is the vector space of equivalent classes of functions in $\tilde{N}_{loc}^{1,p}(X,\mathbb{V})$ under the preceding equivalence relation.

To define the Sobolev space $N^{1,p}(X,Y)$ of mappings $f\colon X\to Y$, we first fix an isometric embedding $\varphi$ of $Y$ into some Banach space $\mathbb{V}$. Then the Sobolev space $N^{1,p}(X,Y)$ consists of all mappings $f\colon X\to Y$ with $\varphi\circ f\in N^{1,p}(X,\mathbb{V})$.

%
%

\subsection{Poincar\'e inequalities and spaces of locally bounded geometry}
The following concept of an abstract Poincar\'e inequality was first introduced by Heinonen and koskela~\cite{hk98} and it plays an important role in the study of analysis on metric spaces; see for instance~\cite{hk00,hkst15}.
\begin{definition}\label{def:Poincar\'e inequality}
	We say that a metric measure space $(X,d,\mu)$ admits a \textit{(1,p)-Poincar\'e inequality} if there exist constants $C\geq 1$ and $\tau\geq 1$ such that
	\begin{equation}\label{eq:Poin ineq}
		\dashint_B|u-u_B|d\mu\leq C\diam(B)\Big(\dashint_{\tau B}g^pd\mu\Big)^{1/p}
	\end{equation}
	for all open balls $B$ in $X$, for every function $u:X\to\bR$ that is integrable on balls and for every upper gradient $g$ of $u$ in $X$.
\end{definition}

The $(1,p)$-Poincar\'e inequality can be thought of as a requirement that a
space contains ``many" curves, in terms of the $p$-modulus of curves in the
space. For a complete doubling metric measure space $X=(X,d,\mu)$ supporting a $(1,p)$-Poincar\'e inequality, there are many important consequences~\cite[Section 9]{hkst15}. For example, one has the important Sobolev embedding results as in the Euclidean spaces. We point out the following geometric implications of Poincar\'e inequalities from~\cite[Theorem 8.3.2]{hkst15}.

\begin{lemma}\label{lemma:poincare implies quasiconvex}
	A complete and doubling metric measure space that supports a Poincar\'e inequality is quasiconvex, quantitatively.
\end{lemma}
Recall that a metric space $Z=(Z,d_Z)$ is said to be \textit{$C$-quasiconvex} or simply \textit{quasiconvex}, $C\geq 1$, if each pair of points can be joined by a rectifiable curve in $Z$ such that
\begin{equation}\label{eq:def for quasiconvex}
	l(\gamma)\leq Cd_Z(x,y).
\end{equation}

Recall that a metric space $X$ is said to be  \textit{$\theta$-linearly locally connected} ($\theta$-LLC) if there exists $\theta\geq 1$ such that for $x\in X$ and $0<r\leq \diam(X)$,

($\theta$-LLC-1) every two points $a,b\in B(x,r)$ can be joined in $B(x,\theta r)$, and

($\theta$-LLC-2) every two points $a,b\in X\backslash \closure{B}(x,r)$ can be joined in $X\backslash \closure{B}(x,\theta^{-1}r)$.

Here, by joining $a$ and $b$ in $B$ we mean that there exists a continuum $\gamma\colon [0,1]\to B$ with $\gamma(0)=a,\gamma(1)=b$.

The following result was proved in~\cite[Corollary 5.8]{hk98}, where the quasiconvexity condition is provided by Lemma~\ref{lemma:poincare implies quasiconvex}.
\begin{proposition}\label{prop:Linear locally contra implies LLC}
	Let $(X,d,\mu)$ be a complete Ahlfors $Q$-regular metric measure space that supports a $(1,Q)$-Poincar\'e inequality. Then $X$ is $\theta$-linearly locally connected with $\theta$ depending only on the data associated with $X$.
\end{proposition}

We next introduce an important class of metric spaces, where a large part of the theory of quasiconformal or/and quasiregular mappings can be extended as in the Euclidean spaces. The importance of such spaces was first realized in~\cite{hk98} in their characterization of Poincar\'e inequalities.

\begin{definition}\label{def:Loewner space}
	Let $(X,d,\mu)$ be a metric measure space. We say that $X$ has \textit{$Q$-Loewner property} if there is a function $\phi\colon (0,\infty)\to (0,\infty)$ such that
	\begin{equation*}
		\modulus_Q(\Gamma(E,F,X))\geq \phi(\zeta(E,F))
	\end{equation*}
	for every non-degenerate compact connected sets $E,F\subset X$, where
	\begin{equation*}
		\zeta(E,F)=\frac{\dist(E,F)}{\min\{\diam E,\diam F\}}.
	\end{equation*}
\end{definition}
By~\cite[Theorem 3.6]{hk98}, if $X$ is Ahlfors $Q$-regular, and $Q$-Loewner, then
\begin{equation}\label{eq: Loewner property for n regular space}
	\modulus_Q(\Gamma(E,F,X))\geq C\Big(\log \zeta(E,F)\Big)^{1-Q}
\end{equation}
when $\zeta(E,F)$ is large enough with $C$ depends only on the data of $X$. By~\cite[Corollary 5.13]{hk98}, a complete (or equivalently proper) Ahlfors $Q$-regular metric measure space that supports a $(1,Q)$-Poincar\'e inequality is $Q$-Loewner.

Following~\cite{hkst01}, we introduce the notion of metric spaces of locally bounded geometry.

\begin{definition}\label{def:space of bounded geometry}
	A metric measure space $(X,d,\mu)$ is said to be of \textit{locally $Q$-bounded geometry}, $Q\geq 1$, if $X$ is separable, pathwise connected, locally compact, and if there exist constants $C_0\geq 1$, $0<\lambda\leq 1$, and a decreasing function $\psi:(0,\infty)\to (0,\infty)$ such that each point $x\in X$ has a neighborhood $U$ (with compact closure in $X $) so that
	\begin{itemize}
		\item $\mu(B_R)\leq C_0R^Q$ whenever $B_R\subset U$ is a ball of radius $R>0$;
		\item $\Modd_Q(\Gamma(E,F,B_R))\geq \psi(t)$ whenever $B_R\subset U$ is a ball of radius $R>0$ and $E$ and $F$ are two disjoint, non-degenerate compact connected sets in $B_{\lambda R}$ with
		\begin{align*}
			\dist(E,F)\leq t\cdot\min\{\diam E,\diam F\}.
		\end{align*}
	\end{itemize}
\end{definition}

In other words, a pathwise connected, locally compact space is of locally $Q$-bounded geometry if and only if it is \textit{locally uniformly Ahlfors $Q$-regular} and \textit{locally uniformly $Q$-Loewner}. In terms of Poincar\'e inequality, a pathwise connected, locally compact space is of locally $Q$-bounded geometry if and only if it is locally uniformly Ahlfors $Q$-regular and supports a local uniform $(1,Q)$-Poincar\'e inequality. Here by saying locally uniformly Ahlfors $Q$-regular we mean that there exists a constant $C_0>0$ such that for each $x\in M$, there is a radius $r_x>0$ so that~\eqref{eq:Ahlfors regular measure} holds for all $0<r<r_x$ with the constant $C_0$, and by saying supporting a local uniform $(1,Q)$-Poincar\'e inequality, we mean that there exists a constant $C>0$ such that for each $x$, there exists a ball $B$ centered at $x$ (with radius depending on $x$) such that the Poincar\'e inequality~\eqref{eq:Poin ineq} with exponent $p=Q$ holds with the constant $C$.  

As a particular case, let us point out that every Riemannian $n$-manifold is of locally $n$-bounded geometry. More exotic examples can be found in~\cite[Section 6]{hk98}.

\newpage
\section{The basic pullback studies}\label{sec:basic pullback studies}
\subsection{Some topological facts about discrete open mappings}\label{subsec:some topological fact}
In this section, we establish preliminary facts about a generic discrete open
mapping $\psi\colon X\to Z$. For every subset $A\subset X$, and every point $z\in Z$, the \emph{multiplicity} $N(z,\psi,A)$ of $\psi$ at $z$ with respect to $A$ is defined to be the cardinality of the set $\psi^{-1}(z)\cap A$. The multiplicity of $A$ under $\psi$ is $N(\psi,A)=\sup_{z\in Z}N(z,\psi,A)$, and
the map $\psi$ has \textit{locally bounded multiplicity} if every $x\in X$ has a neighborhood $D\subset X$ such that $N(\psi,D)<\infty$, and \textit{bounded multiplicity} if $N(\psi,X)<\infty$.

A subset $D\subset X$ is \textit{normal} for $\psi$ if $D$ is open, relatively compact and $\psi(\partial D)=\partial \psi(D)$. We record, without
proof, the following elementary way to guarantee that the restriction of $\psi$ to a subset remains open with respect to the appropriate subspace topologies.

\begin{lemma}\label{lemma:discrete open wrt subspace topology}
	If $S\subset X$ and $\psi^{-1}(\psi(S))=S$, then the restriction $\psi|_S\colon S\to \psi(S)$ is discrete and open as well with respect to the subspace topologies of $S$ and $\psi(S)$.
\end{lemma}

Fix a normal subset $D\subset X$. For $1\leq n\leq N(\psi,D)$, let $D_n=\{x\in D:N(\psi(x),\psi,D)=n\}$. For each $z\in \psi(D)$, we define $M_z$ by
\begin{align*}
	M_z=\min\Big\{\inf_{x,x'\in \psi^{-1}(z)\cap D}\frac{d(x,x')}{6},\ d(\psi^{-1}(z),\ X\backslash D)\Big\}.
\end{align*}
For each $r>0$, we define $U(x,\psi,r)\footnote{We will also use $U(x,\psi,r)$ to denote the $x$-component of $\psi^{-1}(B(\psi(x),r))$ in later sections.}=\psi^{-1}(B(\psi(x),r))\cap B(x,M_{\psi(x)})$.
The following proposition allows us to choose radii $R(z)=R(z,\psi,D)$ such that the sets $U(x,\psi,r)$, $x\in \psi^{-1}(\{z\})$, behave reasonably well, provided $r<R(\psi(x))$. It is well-known to experts, but for completeness, we include the proofs here.

\begin{proposition}\label{prop:existence of normal nbhd}
	Let $D\subset X$ be normal. Then there is a function $R\colon \psi(D)\to (0,\infty)$ with the following properties:
	\begin{enumerate}
		\item For each $z\in \psi(D)$ and $r\leq R(z)$, $B(z,R(z))\subset \psi(D)$;
		
		\item For ever $z\in \psi(D)$ and $r\leq R(z)$, 
		\begin{align*}
			\psi^{-1}\big(B(z,r)\big)\cap D=\bigcup_{x\in \psi^{-1}(z)\cap D}U(x,\psi,r);
		\end{align*}
		
		\item For every $z\in \psi(D)$, $r\leq R(z)$, and $z'\in B(z,r)$
		\begin{align*}
			N(z',\psi,D)=\sum_{x\in \psi^{-1}(z)\cap D}N(z',\psi, U(x,\psi,r)).
		\end{align*}
		In particular, $N(z,\psi,U(x,\psi,r))=1$ for each $x\in \psi^{-1}(z)\cap D$;
		
		\item For every $z\in \psi(D)$, $r\leq R(z)$, and each $x\in \psi^{-1}(z)$, $\psi(U(x,\psi,r))=B(z,r)$;
		
		\item For every $z\in \psi(D)$, $r\leq R(z)$, and each $z'\in B(z,r)$, $N(z',\psi,D)\geq N(z,\psi,D)$;
		
		\item For each $z\in \psi(D)$, $r\leq R(z)$, and each $x\in \psi^{-1}(z)$, $\psi$ is injective on $D_n\cap U(x,\psi,r)$;
		
		\item For each $x,x'\in D_n$, each $r\leq R(\psi(x))$, and each $r'\leq R(\psi(x'))$, if $B(\psi(x),r)\subset B(\psi(x'),r')$, then either $\psi$ is injective on
		\begin{align*}
			D_n\cap \big(U(x,\psi,r)\cup U(x',\psi,r')\big)\quad \text{or}\quad U(x,\psi,r)\cap U(x',\psi,r')=\emptyset;
		\end{align*}
		
		\item For each $x,x'\in D$, $r\leq R(\psi(x))$, and each $r'\leq R(\psi(x'))$, if $B(\psi(x),r)\subset B(\psi(x'),r')$, then either
		\begin{align*}
			U(x,\psi,r)\subset U(x',\psi,r')\quad \text{or}\quad U(x,\psi,r)\cap U(x',\psi,r')=\emptyset.
		\end{align*}	
	\end{enumerate}
	Moreover, every other function $\tilde{R}\colon \psi(D)\to (0,\infty)$ such that $\tilde{R}(z)\leq R(z)$ for each $z\in \psi(D)$ satisfies all these properties as well.
\end{proposition}
\begin{proof}
	The very final statement is trivial, so we need only construct the function
	R. Moreover, it is immediately clear that it suffices to prove the properties (1), (2),
	(3), (4), (5), and (6) only for $r = R(z)$, and property (7) only for $r = R(\psi(x))$ and
	$r' = R(\psi(x'))$. Property (8), on the other hand, we must prove for all $r\leq  R(\psi(x))$ and $r'\leq R(\psi(x'))$. 
	
	Since $\psi(D)$ is open, there is for each $z$ a radius $R_0(z) > 0$ such
	that (1) is satisfied for $R(z) = R_0(z)$. 
	
	We next claim that for each $z\in \psi(D)$, there is some radius $R_1(z)\leq R_0(z)$ such (2) is satisfied for $R(z) = R_1(z)$. Indeed, if not,
	then there is a sequence of points $\{x_i\}$ such that $\{\psi(x_i)\}$ converges to $z$ and such that
	\[d(x_i,\psi^{-1}(z)\cap D)\geq M_z.\] 
	Since $D$ is relatively compact, there is a subsequence $x_{i_j}$ converging to some $x'\in D$. Then $\psi(x') = z$, but 
	\[d(x',\psi^{-1}(z)\cap D)\geq M_z,\]
	so $x'\in \psi^{-1}({z})\cap \partial D$, contradicting the normality of $D$. 
	
	Property (3) also holds for $R(z) = R_1(z)$, as an immediate consequence of property (2). 
	
	Now, since $\psi$ is open, the intersection $\bigcap_{x\in \psi^{-1}(z)\cap D}\psi\big(U(x,\psi,R_1(z))\big)$ is open as well, and therefore contains the ball $B(z, R_2(z))$ for some $R_2(z)\leq R_1(z)$. It follows that property (4) is met for $R(z)=R_2(z)$.
	
	Property (5) follows immediately from (4) and (3), and (6) follows from (4) and (3) as well, by the pigeonhole principle.
	
	We claim that property (8) is also satisfied for $R(z) = R_2(z)$. Indeed, suppose by contradiction that there are two points $x,x'\in D$, $r\leq R_2(x)$, $r'\leq R_2(x')$ such that $B(\psi(x), r) \subset B(\psi(x'), r')$, that $U(x,\psi,r)\cap U (x',\psi,r')\neq \emptyset$, and that $U(x,\psi,r)\nsubseteq U(x',\psi,r')$. Let $z = \psi(x)$, $z'=\psi(x')$. By assumption, $U(x,\psi,r)$ intersects both $B(x',M_{z'})$ and $B(x'',M_{z'})$ for some $x''\in \psi^{-1}(z')\backslash \{x'\}$, and so
	\[2M_z \geq \diam (U (x,\psi,r))\geq d(U(x',\psi,r'), U(x'',\psi,r'))\geq 4M_{z'},\]
	so that $M_z\geq 2M_{z'}$. On the other hand, property (4) implies that there are points $x_1\in U(x',\psi,r')$ and $x_2\in U(x'',\psi,r')$ such that $\psi(x_1)=\psi(x_2)=z$. Thus
	\begin{align*}
		6M_z&\leq d(x_1,x_2)\leq \diam(U(x',\psi,r'))+\diam(U(x'',\psi,r''))+d(U(x',\psi,r'),U(x'',\psi,r''))\\
		&\leq \diam(U(x',\psi,r'))+\diam(U(x'',\psi,r''))+\diam(U(x,\psi,r))\leq 4M_{z'}+2M_z,
	\end{align*} 
	so that $M_z\leq M_{z'}$, a contradiction. Thus property (8) is satisfied for $R(z)=R_2(z)$.
	
	Finally, for each $z$, set $R(z)=R_2(z)/3$. Suppose $U(x,\psi,R(\psi(x)))\cap U(x',\psi,R(\psi(x')))\neq \emptyset$, for $x,x'\in D_n$. Let $z=\psi(x)$ and $z'=\psi(x')$. By property (4), it follows that $B(z,R(z))\cap B(z',R(z'))\neq \emptyset$. We may assume without loss of generality that $R(z)\leq R(z')$. Then $B(z,R(z))\subset B(z',R(z'))$. By property (8), we conclude that $$U(x,\psi,R(\psi(x)))\subset U(x',\psi,R_2(\psi(x'))).$$
	 Property (6) applied to the set $U(x,\psi,R_2(\psi(x)))$ gives immediately property (7). 
\end{proof}

By Proposition~\ref{prop:existence of normal nbhd} (5), it immediately follows the multiplicity $N(z,\psi,D)$ is lower semicontinuous as a function of $z$. Since for each $n\in \mathbb{N}$, $\psi^{-1}\big(\psi(D_n)\big)=D_n$, we see that for every open subset $V\subset D$, $\psi(V\cap D_n)=\psi(V)\cap \psi(D_n)$ so that $\psi|_{D_n}$ is a discrete open $n$-to-1 mapping onto its image. By Proposition~\ref{prop:existence of normal nbhd} (6), $\psi|_{D_n}$ is locally bijective, and thus is a local homeomorphism. 

It is helpful on occasion to decompose $D_n$ into sets $D_{n,1},\dots,D_{n,n}$ on which $\psi$ is bijective. The next lemma allows us to do this with no measure-theoretic concerns.

\begin{lemma}\label{lemma:decomposition Dn}
	Let $1\leq n\leq N(\psi,D)$. Then there are pairwise disjoint Borel subsets $D_{n,j}$, $j=1,\dots,n$, such that $D_n=\bigcup_{j=1}^nD_{n,j}$ and such that for each $j$, $\psi|_{D_{n,j}}$ is a bijection onto $\psi(D_{n})$.
\end{lemma}
\begin{proof}
	Let $\{V_i\}$ be a countable cover of $D_n$ by open subsets $V_i\subset D$, such that $\psi|_{D_n\cap V_i}$ is injective. Such a cover exists by the relative compactness, and hence separability of $D$, along with Proposition~\ref{prop:existence of normal nbhd} (6). For each $j=0,\dots,n$, construct $D_{n,j}$ as follows: Let $D_{n,0}=\emptyset$. Having defined $D_{n,j}$ for all $j<k\leq n$, let $D_{n,k}^0=\emptyset$, and for each integer $i\geq 1$, define $D_{n,j}^i$ by
	\begin{align*}
		D_{n,j}^i=D_{n,j}^{i-1}\cup\Big(V_i\backslash \Big(\psi^{-1}\big(\psi(D_{n,j}^{i-1})\big)\cup \bigcup_{j=0}^{k-1}D_{n,j}\Big)\Big).
	\end{align*}
	Finally, let $D_{n,k}=\bigcup_{i=1}^\infty D_{n,j}^i$. By construction, the sets $D_{n,j}$ are disjoint Borel sets on which $\psi$ is injective. It remains to show that 
	$\psi(D_{n,j})=\psi(D_n)$. To see this, fix $1\leq k\leq n$ and $z\in \psi(D_n)$. The union $\bigcup_{j=1}^{k-1}D_{n,j}$ contains at most $(k-1)$ elements of the fiber $\psi^{-1}(z)$, so there is some $x\in D_n\backslash \bigcup_{j=1}^{k-1}D_{n,j}$ such that $\psi(x)=z$. Since the sets $V_i$ are an open cover of $D_n$, there is some $V_i\ni x$. Then either $z\in \psi(D_{n,k}^{i-1})$ or $z\notin \psi(D_{n,k}^{i-1})$. In the latter case, $x\notin \psi^{-1}(\psi(D_{n,k}^{i-1}))$. By definition, we have $x\in D_{n,k}^i$, and so $z=\psi(x)\in \psi(D_{n,k})$.
\end{proof}

\subsection{The pullback measure}
Let $\rho\colon X\to [0,\infty)$ be a function. We define two functions $\sup(\rho,y,\psi,A)$ and $\sum(\rho,y,\psi,A)$ with respect to a branched covering $\psi\colon X\to Y$ by
\begin{align*}
	\sup(\rho,y,\psi,A)=\sup_{x\in \psi^{-1}(y)\cap A}\rho(x)
\end{align*}
and 
\begin{align*}
	\sum(\rho,y,\psi,A)=\sum_{x\in \psi^{-1}(y)\cap A}\rho(x).
\end{align*}

Suppose, as we will do from now on, that $Y$ is equipped with a (locally finite, Borel-regular) measure $\nu$. We define the ``pullback measure" $\psi^*\nu$ on $X$ via the formula
\begin{align*}
	\psi^*\nu(A)=\int_YN(y,f,A)d\nu(y).
\end{align*}
By the subadditivity of the integral, $\psi^*\nu$ is an outer measure. Note that by definition, for every subset $A\subset X$, $\psi^*\nu(A)=0$ if and only if $\nu(\psi(A))=0$.

Since the restrictions $\psi|_{D_n}$ are local homeomorphisms, it follows that $\psi$ maps Borel sets to Borel sets, and that $\psi^*\nu$ is a locally finite Borel regular outer measure on $X$. Let $A\subset X$. Since $\psi$ preserves Borel sets, and $\nu(\psi(A))=0$ if and only if $\psi^*\nu(A)=0$, it follows from the Borel regularity of each measure that $A$ is $\psi^*\nu$-measurable if and only if $\psi(A)$ is $\nu$-measurable. Thus, if $\rho\colon X\to \R$ is a Borel (resp. $\psi^*\nu$-measurable) function on $X$ and $A\subset X$ a Borel set, then $\sup(\rho,\cdot,\psi,A)$ and $\sum(\rho,\cdot,\psi,A)$ are Borel (resp. $\nu$-measurable) functions on $Y$. Moreover, by standard approximation arguments, we have
\begin{align}\label{eq:change of variable general}
	\int_X \rho(x)d\psi^*\nu(x)=\int_{Y}\sum(\rho,y,\psi,X)d\nu(y),
\end{align}
for every Borel function $\rho\colon X\to \R$.

We define the (volume) Jacobian of $\psi\colon X\to Y$ by
\begin{align*}
	J_\psi:=\frac{d\psi^*\nu}{d\mu},
\end{align*}
and the inverse (volume) Jacobian of $\psi$ by
\begin{align*}
	J_\psi^{-1}:=\frac{d\mu}{d\psi^*\nu}.
\end{align*}
We say that $\psi$ satisfies Condition $N$ if $\nu(\psi(A))=0$ whenever $\mu(A)=0$, and Condition $N^{-1}$ if $\mu(A)=0$ whenever $\nu(\psi(A))=0$. Since, by the preceding discussion, $\psi^*\nu(A)=0$ if and only if $\nu(\psi(A))=0$, Condition $N$ is equivalentto the condition that $\psi^*\nu\ll \mu$, and Condition $N^{-1}$ is equivalent to the condition that $\nu\ll \psi^*\nu$. From this, it immediately follows that
\begin{align}\label{eq:area inequality general 1}
	\int_X\rho(x) J_\psi(x) d\mu(x)\leq \int_X \rho(x)d\psi^*\nu(x)=\int_Y\sum(\rho,y,\psi,X)d\nu(y),
\end{align}
for every Borel function $\rho\colon X\to \R$, and the two sides are equal for every Borel function $\rho$ if and only if $\psi$ satisfies Condition $N$. Similarly,
\begin{align}\label{eq:area inequality general 2}
	\int_Y\sum(\rho J_\psi^{-1},y,\psi,X)d\nu(y)= \int_X\rho(x)J_\psi^{-1}(x)d\psi^*\nu(x)\leq \int_X\rho(x)d\nu(x),
\end{align}
with equality for all $\rho$ if and only if $\psi$ satisfies Condition $N^{-1}$.

\subsection{Integration theory in pullback geometry}
As already mentioned in the introduction, we will later use certain pullback normal neighborhoods, instead of balls, to run the covering arguments from~\cite{bkr07} to relate metrically quasiregularity and analytic quasiregularity. From the technical point of view, it is important that our special family of normal neighborhoods will have many nice behaviors as the family of balls. For this reason, we develop a variant of the Lebesgue-Radon-Nikodym theory in this section. 

\begin{lemma}\label{lemma:covering lemma}
	Let $\mathcal{U}$ be a family of sets of the form $U=U(x,\psi,r)$, with $x\in D$ and $5r_x<R(\psi(x))$, and suppose $\nu$ is locally doubling. Then there is a countable, pairwise disjoint subfamily $\mathcal{U'}\subset \mathcal{U}$ such that
	\begin{align*}
		\bigcup_{U\in \mathcal{U}}U\subset \bigcup_{U'\in \mathcal{U}'}5U'.
	\end{align*}
\end{lemma}
\begin{proof}
	If $\mathcal{U}$ is empty the theorem is trivial. Otherwise, we construct $\mathcal{U}'=\{U_i\}$, where each $U_i=U(x_i,\psi,r_i)\in \mathcal{U}$, in the following manner. For ease of notation, we index the family $\mathcal{U}$ by some index set $\mathcal{A}$, so that $\mathcal{U}=\bigcup_{\alpha\in \mathcal{A}}U_\alpha$, where $U_\alpha=U(x_\alpha,\psi,r_\alpha)$. We construct $\mathcal{U}'$ inductively. We first choose $U_0\in \mathcal{U}$ such that
	\begin{align*}
		r_0\geq \frac{1}{2}\sup_{\alpha\in \mathcal{A}}r_\alpha.
	\end{align*}
	Having chosen $U_k$ for all $k<n$, let 
	\begin{align*}
		\mathcal{A}_n=\big\{\alpha\in \mathcal{A}:U_\alpha\nsubseteq \bigcup_{k=1}^{n-1}5U_k\big\}.
	\end{align*}
	If $\mathcal{A}_n$ is nonempty, choose $U_n$ such that 
	\begin{align*}
		r_n\geq \frac{1}{2}\sup_{\alpha\in \mathcal{A}_n}r_\alpha. 
	\end{align*}
	We proceed inductively, choosing $U_n$ for each $n\in \mathbb{N}$, unless $\mathcal{A}_n$ is empty for some $n$, in which case the sequence terminates at $n-1$.
	
	For each $i$, let $B_i=B(\psi(x_i),r_i)=\psi(U_i)$. We first note that by construction, for $m<n$, $U_n\nsubseteq 5U_m$, and $r_n\leq 2r_m$, so that $\diam(B_n)\leq 4r_m$. Now, $B_n\nsubseteq 5B_m$, and so $B_n\cap B_m=\emptyset$, from which it follows that $U_n\cap U_m=\emptyset$. It remains to show the inclusion.
	
	If the sequence terminates, this is immediate. Otherwise, we need only show that $\bigcap_{n=1}^\infty \mathcal{A}_n=\emptyset$. Suppose by contradiction that $\alpha\in \bigcap_{n=1}^\infty \mathcal{A}_n$. By our construction, $0<r_\alpha\leq 2r_n$ for each $n\in \mathbb{N}$. By the doubling property of $\nu$ and the compactness of $D$, there is a constant $C>0$ such that $\psi^*\nu(U_n)\geq \nu(B_n)\geq C$ for each $n$. Since the sets $U_n$ are disjoint, this implies that 
	\begin{align*}
		\psi^*\nu(D)\geq \sum_{n=1}^\infty \psi^*\nu(U_n)=\infty, 
	\end{align*}
	contradicting the local finiteness of $\psi^*\nu$.
\end{proof}
\begin{remark}\label{rmk:on covering I}
	The assumption that $\nu$ is locally doubling is not necessary; we have assumed  it since it is harmless and makes the proof somewhat more straightforward. In fact, if $D$ has finite topological dimension $N$, then by the Whitney embedding theorem, $D$ is homeomorphic to a subset $D'\subset \R^{2N+1}$. Thus we may assume that $D$ is a subset in $\R^{2N+1}$. By~\cite{h01}, $D$ may be equipped with a doubling measure. Moreover, even in the case when the topological dimension of $D$ is infinite, the theorem can be proved with an argument along the lines of \cite[Proof of Theorem 1.2]{h01}, with the countability of $\mathcal{U}'$ following from the compactness of $D$. 
\end{remark}

Inspired by the covering lemmas from~\cite{bkr07}, we introduce the following concept. 
\begin{definition}[Admissible pointed neighborhoods]\label{def:admissible pointed nbhd}
	A sequence $\mathcal{A}=\{A_i:x_i\in A_i\}_{i\in I}$ of pointed neighborhoods is \textit{admissible} if for each $i\neq j$ form $I$, the following two conditions are satisfied:
	\begin{enumerate}
		\item Either $x_i\notin A_j$ or $x_j\notin A_i$.
		\item $A_i\nsubseteq A_j\nsubseteq A_i$.
	\end{enumerate}
\end{definition}

We have the following simple fact for admissible pointed sets, which slightly generalizes~\cite[Lemma 2.3]{bkr07}.
\begin{lemma}\label{lemma:bkr covering 1}
	Let $\mathcal{A}=\{A_i:x_i\in A_i\}_{i\in I}$ be an admissible sequence of pointed sets such that for each $k\in I$, 
	\begin{align*}
		B(x_k,r_k)\subset A_k\subset B(x_k,Hr_k).
	\end{align*}
	Then for each $i\neq j$ in $I$, $B(x_i,r_i/5H)\cap B(x_j,r_j/5H)=\emptyset$.
\end{lemma}
\begin{proof}
	Since $\mathcal{A}$ is admissible, we may assume with no loss of generality that $x_i\not\in A_j$, which in particular implies that $x_i\notin B(x_j,r_j)$, and hence $d(x_i,x_j)\geq r_j$. Seeking for a contradiction, let us assume that $B(x_i,r_i/5H)\cap B(x_j,r_j/5H)\neq\emptyset$. Then $d(x_i,x_j)< \frac{r_i+r_j}{5H}$, from which it follows that
    \begin{equation*}
    	r_j\leq d(x_i,x_j)< \frac{r_i+r_j}{5H}
    \end{equation*}
    and so $r_j\leq \frac{1}{5H-1}r_i$. 
    
    On the other hand, for each $y\in B(x_j,Hr_j)$, we have
    \begin{align*}
    	d(y,x_i)&\leq d(y,x_j)+d(x_j,x_i)< Hr_j+\frac{r_i+r_j}{5H}\\
    	&\leq \big(
    	\frac{1}{5H}+\frac{5H^2+1}{5H(5H-1)}\big)r_i<r_i,
    \end{align*}
    since $H\geq 1$. This implies that
    \begin{align*}
    	A_j\subset B(x_j,Hr_j)\subset B(x_i,r_i)\subset A_i,
    \end{align*}
    contradicting with our assumption that $\mathcal{A}$ is admissible.
\end{proof}

Unlike the situation in~\cite{bkr07}, we will not always begin our arguments with a family of balls. Thus we require a slight variation on~\cite{bkr07}.
\begin{lemma}\label{lemma:bkr variant 2}
	Let $X$ be a metric space, $S\subset X$ relatively compact, and for each $x\in S$, let $A_x\subset X$ and $r_x>0$ satisfy
	\begin{align}\label{eq:bkr variant 2}
		B(x,r_x)\subset A_x\subset B(x,Hr_x).
	\end{align}
	Then there is a sequence of points $x_i\in S$ such that the corresponding sequence $A_i=A_{x_i}$ is admissible, and 
	\begin{align}\label{eq:bkr variant 2 2}
		S\subset \bigcup_{i=1}^\infty A_i.
	\end{align}
\end{lemma}
\begin{proof}
	We first construct a sequence $x_i$ satisfying condition (1) of Definition~\ref{def:admissible pointed nbhd} and the inclusion~\eqref{eq:bkr variant 2 2}. We will then pass to a subsequence to complete the proof. Upon choosing $x_i$, we let $r_i=r_{x_i}$ and $A_i=A_{x_i}$.
	
	If $M_0=\sup_{x\in S}r_x=\infty$, then we are finished by the relative compactness of $S$. Otherwise, we proceed inductively. Let $S_0=S$, and choose $x_0\in S_0$ such that $r_0>M_0/2$. Once $x_l$ (and hence also $A_l$) has been chosen for all $l<k$, we define $S_k=S\backslash \bigcup_{l<k}A_l$, and $M_k=\sup_{x\in S_k}r_x$, and then choose $x_k\in S_k$ such that
	$r_k>M_k/2$. We continue this process for all natural numbers $k$, unless for some $n$, $S_{n+1}=\emptyset$, in which case the sequence terminates with $A_n$. We observe that for $i<j$, $A_i$ does not contain $x_j$, by construction. Thus $\{A_i\}$ satisfies condition (1) of Definition~\ref{def:admissible pointed nbhd}.
	
	If $\{A_i\}$ is finite, then the inclusion~\eqref{eq:bkr variant 2 2}  follows immediately from the definition of the sets $S_i$. Assume then that $\{A_i\}$ is infinite. We first claim that $M=\lim_{k\to \infty}M_k=0$. Indeed, if $M>0$, then by the left-hand inequality in~\eqref{eq:bkr variant 2}, $d(x_i,x_j)\geq M$ for each $i=j$, and so $S$ contains an infinite $M$-separated set, contradicting its relative compactness. From the claim it follows that $\bigcap_{k=1}^\infty S_k=\emptyset$, and so the inclusion~\eqref{eq:bkr variant 2 2} holds.
	
	The only remaining obstruction to $\{A_i\}$ being admissible is that it may fail to satisfy condition (2) of Definition~\ref{def:admissible pointed nbhd}. To rectify this, we simply remove any member $A_i$ that is
	contained in another set $A_j$. It is immediate that the resulting subsequence $\{A_{i_j}\}$ satisfies conditions (1) and (2) of Definition~\ref{def:admissible pointed nbhd}. To show that the inclusion~\eqref{eq:bkr variant 2 2} holds, we observe
	that by the right-hand inequality in~\eqref{eq:bkr variant 2},  
	\begin{align*}
		\diam(A_k)\leq Hr_k\leq HM_k/2.
	\end{align*}
	Since $M_k$ converges to 0, so does $\diam(A_k)$, and thus every nested chain $A_{k_1}\subset \cdots\subset A_{k_j}\cdots$ must be finite. The final member $A_{k_n}$ of such a chain is therefore in the subsequence
	$\{A_{i_j}\}$. Thus for each $k$, $A_k\subset A_{i_{j_k}}$ for some $j_k$, so that 
	\begin{align*}
		\bigcup_{j=0}^\infty A_{i_j}=\bigcup_{i=0}^\infty A_i\supset S.
	\end{align*}
\end{proof}

Let $\psi\colon (X,d_X,\mu)\to (Y,d_Y,\nu)$ be a branched covering. If the measure $\nu$ is doubling, we can generalize a number of important results from integration theory on doubling measures to $\psi^*\nu$.
\begin{lemma}\label{lemma:Lebesgue differentiation theorem pb}
	Let $D\subset X$ be a normal domain, $A\subset D$ a Borel set, and $\mathcal{U}$ a family of sets of the form $U=U(x,\psi,r)$, such that for each $x\in A$, there is a sequence $r_i$ converging to 0 such that $U(x,\psi,r_i)\in \mathcal{U}$. Then for each $\varepsilon>0$ there is a
	pairwise disjoint, countable family of sets $U_i=U(x_i,\psi,r_i)$ such that
	\begin{align*}
		\psi^*\nu\Big(A\backslash \bigcup_{i=1}^\infty U_i\Big)<\varepsilon.
	\end{align*}
\end{lemma}
\begin{proof}
	We will follow closely the standard argument as in the ball case. By Lemma~\ref{lemma:covering lemma} and our assumption on $U$, we can select a pairwise disjoint, countable family $\mathcal{U}_0$ of sets $U_i=U(x_i,\psi,r_i)$ such that $A$ is contained in $\bigcup_{i=1}^\infty 5U_i$ and that $\sum_{i=1}^\infty \nu(B_i)<\infty$, where $B_i=\psi(U_i)=B(\psi(x_i),r_i)$. 
	Note that by the compactness of $D$ and doubling property of $\nu$, we have
	\begin{align*}
		\sum_{i\geq 1}\psi^*\nu(5U_i)&\leq N(\psi,D)\sum_{i=1}^\infty\nu(\psi(5U_i))=N(\psi,D)\sum_{i=1}^\infty\nu(5B_i)\\
		&\leq cN(\psi,D)\sum_{i=1}^\infty\nu(B_i)<\infty,
	\end{align*}
	which implies that
	\begin{align*}
		\sum_{i>N}\psi^*\nu(5U_i)\to 0 \quad \text{ as } N\to \infty.
	\end{align*}
	It suffices to show that 
	\begin{align*}
		A\backslash \bigcup_{i=1}^N U_i\subset \bigcup_{i>N}5U_i\quad \text{ or }\quad \psi(A)\backslash \bigcup_{i=1}^N B_i\subset \bigcup_{i>N} 5B_i.
	\end{align*}
	To this end, take $a\in A\backslash \bigcup_{i=1}^N U_i$ and select a small $U_\alpha\in \mathcal{U}$ that does not meet any of the sets $U_i$ for $i\leq N$ (probably needs to assume apply everything to $\closure{U}$). On the other hand, by Lemma~\ref{lemma:covering lemma}, the family $\mathcal{U}_0$ can be chosen so that some set $U_j$ from $\mathcal{U}_0$ with radius at least $r_j\geq r_\alpha/2$. Thus $j>N$ and $U_\alpha\subset 5U_j$, as required.
\end{proof}

If $\rho\in L^p_{loc}(\psi^*\nu)$ and
\begin{align}\label{eq:def for Lebesgue point}
	\limsup_{r\to 0}\Big(\dashint_{U(x,\psi,r)}|\rho-\rho(x)|^pd\psi^*\nu\Big)^{1/p}=0,
\end{align}
then $x$ is called a \textit{$(p,\psi)$-Lebesgue point} of $\rho$. For $p=1$, we simply say $x$ is a $\psi$-\textit{Lebesgue point} of $\rho$. 

As in the classical case, a standard application of the Vitali covering theorem, Lemma~\ref{lemma:Lebesgue differentiation theorem pb}, one obtains the following version of the Lebesgue differentiation theorem.
\begin{corollary}\label{coro:Lebesgue point pb}
	If $\rho\in L^p_{loc}(\psi^*\nu)$ with a doubling measure $\nu$. Then $\psi^*\nu$-a.e. $x\in X$ is a $(p,\psi)$-Lebesgue point of $\rho$.
\end{corollary}
\begin{proof}
	The argument for this is very similar to the standard case. For completeness, we still outline it here. As in the classical case, we do some simple reduction. First, it suffices to show the case $p=1$. Indeed, having the case $p=1$, for general $p\in [1,\infty)$, we know
	\begin{align*}
		\lim_{r\to 0}\dashint_{U(x,\psi,r)}|\rho(y)-r_i|^p d\psi^*\nu(y)=|\rho(x)-r_i|^p
	\end{align*}
	for $\psi^*\nu$-a.e. $x$ and each countable dense subset $\{r_i\}_{i\in \mathbb{N}}$ of $\R$. In particular, for $\psi^*\nu$-a.e. $x$, the above equation holds for all $i$. Fix such a point and $\varepsilon>0$. We may choose $r_i$ such that $|\rho(x)-r_i|^p<\frac{\varepsilon}{2^p}$. Then
	\begin{align*}
		\limsup_{r\to 0}&\dashint_{U(x,\psi,r)}|\rho(y)-\rho(x)|^p d\psi^*\nu(y)\\
		&\leq 2^{p-1}\Big(\dashint_{U(x,\psi,r)}|\rho(y)-r_i|^p d\psi^*\nu(y)+\dashint_{U(x,\psi,r)}|r_i-\rho(x)|^p d\psi^*\nu(y)\Big)\\
		&\leq 2^{p-1}\big(|\rho(x)-r_i|^p+|\rho(x)-r_i|^p\big)<\varepsilon. 
	\end{align*}
	
	Secondly, to prove~\eqref{eq:def for Lebesgue point} for the case $p=1$, it suffices to show that
	\begin{align}\label{eq:Lebesgue point variant}
		\lim_{r\to 0}\dashint_{U(x,\psi,r)}\rho(y) d\psi^*\nu(y)=\rho(x) 
	\end{align}
	holds for $\psi^*\nu$-a.e. for all locally integrable non-negative function $\rho$, since the general case follows routinely by splitting $|\rho-\rho(x)|$ into the positive part and negative part and then applying the result separately.  
	
	Finally, we turn to the proof of~\eqref{eq:Lebesgue point variant}. Let $E$ be the set of points in $X$ where the equation~\eqref{eq:Lebesgue point variant} fails and cover $E$ by sets of the form $U(x,\psi,r)$ with $r=r_x$ sufficiently small. By Lemma~\ref{lemma:Lebesgue differentiation theorem pb}, there is a countable union of sets of this kind containing $\psi^*\nu$-a.e. points in $E$. Thus, it suffices to show that $E$ has $\psi^*\nu$-measure zero in a fixed set $U$ where $\rho$ is integrable.
	
	We first claim that if $t>0$ and if
	\begin{align*}
		\liminf_{r\to 0}\dashint_{U(x,\psi,r)}\rho(y)d\psi^*\nu(y)\leq t
	\end{align*}
	for each $x$ in a subset $A$ of $U_0$, then 
	\begin{align*}
		\int_A \rho(y)d\psi^*\nu(y)\leq t\psi^*\nu(A).
	\end{align*}
	Indeed, fix $\varepsilon>0$ and choose an open set $U\supset A$ such that $\psi^*\nu(U)\leq \psi^*\nu(A)+\varepsilon$. Then each point in $A$ has arbitrarily small neighborhoods in $U$ of the form $U(y,\psi,r)$ where the mean value of $\rho$ is less than $t+\varepsilon$. Lemma~\ref{lemma:Lebesgue differentiation theorem pb} implies that we can pick a countable disjoint collection of such a collection covering $\psi^*\nu$-a.e. of $A$, from which
	\begin{align*}
		\int_A \rho(y)d\psi^*\nu(y)\leq (t+\varepsilon)\psi^*\nu(U)+\varepsilon\leq (t+\varepsilon)(\psi^*\nu(A)+\varepsilon),
	\end{align*}
	and the desired claim follows upon letting $\varepsilon\to 0$. A similar argument shows that if $t>0$ and if 
		\begin{align*}
		\limsup_{r\to 0}\dashint_{U(x,\psi,r)}\rho(y)d\psi^*\nu(y)\geq t
		\end{align*}
		for all $x\in A\subset U_0$, then
		\begin{align*}
			\int_A \rho(y)d\psi^*\nu(y)\geq t\psi^*\nu(A).
		\end{align*}
	
	On the other hand, if $A_{s,t}$ is the set of points $y$ in $U_0$ for which
	\begin{align*}
		\liminf_{r\to 0}\dashint_{U(x,\psi,r)}\rho(y)d\psi^*\nu(y)\leq s<t\leq \limsup_{r\to 0}\dashint_{U(x,\psi,r)}\rho(y)d\psi^*\nu(y),
	\end{align*}	
	then $\psi^*\nu(A_{s,t})=0$, since our preceding claims imply 
	\begin{align*}
		t\psi^*\nu(A_{s,t})\leq \int_{A_{s,t}}\rho(y)d\psi^*\nu(y)\leq s\psi^*\nu(A_{s,t}).
	\end{align*}
	Thus the limit on the left hand side of~\eqref{eq:Lebesgue point variant} exists and is finite $\psi^*\nu$-a.e. in $U_0$. Denote this limit by $g(x)$ whenever it exists. It remains to show that $g(x)=\rho(x)$ $\psi^*\nu$-a.e. in $U_0$.
	
	Fix both a Borel set $F\subset U_0$ and $\varepsilon>0$; for each $n\in \mathbb{N}$, denote
	\begin{align*}
		A_n=\big\{x\in F:(1+\varepsilon)^n\leq g(x)<(1+\varepsilon)^n\big\}.
	\end{align*}
	Then by the second claim
	\begin{align*}
		\int_F g(y)d\psi^*\nu(y)&=\sum_n\int_{A_n}g(y)d\psi^*\nu(y)\leq \sum_n(1+\varepsilon)^{n+1}\psi^*\nu(A_n)\\
		&\leq (1+\varepsilon)\sum_n\int_{A_n} \rho(y)d\psi^*\nu(y)=(1+\varepsilon)\int_F\rho(y)d\psi^*\nu(y),
	\end{align*}
	and similarly, by our first claim,
	\begin{align*}
			\int_F g(y)d\psi^*\nu(y)&=\sum_n\int_{A_n}g(y)d\psi^*\nu(y)\geq \sum_n(1+\varepsilon)^n\psi^*\nu(A_n)\\
			&\geq (1+\varepsilon)^{-1}\sum_n\int_{A_n} \rho(y)d\psi^*\nu(y)=(1+\varepsilon)^{-1}\int_F\rho(y)d\psi^*\nu(y).
	\end{align*}
	By letting $\varepsilon\to 0$, we infer that
	\begin{align*}
		\int_F g(y)d\psi^*\nu(y)=\int_F\rho(y)d\psi^*\nu(y)
	\end{align*}
	and hence $g=\rho$ $\psi^*\nu$-a.e. in $U_0$. The proof is complete.
	
\end{proof}

As a corollary of Lemma~\ref{lemma:Lebesgue differentiation theorem pb} and Corollary~\ref{coro:Lebesgue point pb}, we have the following version of the Lebesgue-Radon-Nikodym theorem.
\begin{corollary}\label{coro:Radon-Nikodym derivative}
	Let $\mu$ be a Radon measure on $X$, and let $\nu$ be doubling. Then the Radon-Nikodym derivative of $\mu$ with respect to $\psi^*\nu$ is given, at $\psi^*\nu$-a.e. $x\in X$, by
	\begin{align*}
		\frac{d\mu}{d\psi^*\nu}(x)=\lim_{r\to 0}\frac{\mu(U(x,\psi,r))}{\psi^*\nu(U(x,\psi,r))}=\lim_{r\to 0}\frac{\mu(U(x,\psi,r))}{\nu(B(\psi(x),r))}.
	\end{align*}
\end{corollary}
\begin{proof}
	We only need to show the first equality, since the second one follows from the fact that 
	\begin{align*}
		\lim_{r\to 0}\frac{\psi^*\nu(U(x,\psi,r))}{\nu(B(\psi(x),r))}=1
	\end{align*}
	at $\psi^*\nu$-a.e. $x\in X$, which is a simple consequence of Corollary~\ref{coro:Lebesgue point pb}.
\end{proof}

Suppose $\mu$ is doubling, $D\subset X$ is normal, and $\nu_n$ is a measure on $Y$ concentrated on $\psi(D_n)$. Then the pullback $\psi^*\nu_n$ is concentrated on $D_n$, so that
\begin{align}\label{eq:19}
	\frac{d\psi^*\nu_n}{d\mu}(x)=0
\end{align}
for $\mu$-a.e. $x\in D\backslash D_n$. Moreover, since $\psi|_{D_n}$ is a local homeomorphism, the Lebesgue-Radon-Nikodym Theorem implies that
\begin{align}\label{eq:20}
	\frac{d\psi^*\nu_n}{d\mu}(x)=\lim_{r\to 0}\frac{\psi^*\nu_n(B(x,r))}{\mu(B(x,r))}=\lim_{r\to 0}\frac{\nu_n(\psi(B(x,r)))}{\mu(B(x,r))}
\end{align}
at $\mu$-a.e. $x\in D_n$. In particular, the preceding argument, applied to the measures $\nu_n=\nu|_{\psi(D_n)}$ for $n=1,\dots,N(\psi,D)$ gives us
\begin{align}\label{eq:Jacobian presentation}
	J_\psi(x)=\lim_{r\to 0}\frac{\nu(\psi(B(x,r)))}{\mu(B(x,r))}
\end{align}
for $\mu$-a.e. $x\in X$. Arguing similarly as above, replacing the Lebesgue-Radon-Nikodym Theorem with Corollary~\ref{coro:Radon-Nikodym derivative}, we infer that
\begin{align}\label{eq:inverse Jacobian presentation}
	J_\psi^{-1}(x)=\lim_{r\to 0}\frac{\mu(U(x,\psi,r))}{\nu(B(\psi(x),r))}
\end{align}
for $\psi^*\nu$-a.e. $x\in X$.


\subsection{The local index and essential index}
The \textit{local index} $i(x,\psi)$ of a continuous mapping $\psi$ at $x$ is traditionally determined by
the homomorphism of local homology or cohomology groups induced by $\psi$ at $x$. In
order to retain the ability to study spaces without a good topological degree theory, we instead define the index to be the local multiplicity of $\psi$ at $x$, i.e.,
\begin{align*}
	i(x,\psi)=\inf\big\{N(f,D):D\subset X \text{ is an open neighborhood of } x\big\}.
\end{align*}
This definition coincides with the usual one when $\psi$ is a discrete, open, and sense-preserving mapping between homology or cohomology manifolds (see for instance~\cite{hr02}). It is immediate from the definition that $\psi$ is a
local homeomorphism at $x$ if and only if $i(x,\psi)=1$. The \textit{branch set} $\mathcal{B}_\psi$ of $\psi$ is the set of points where $\psi$ is not locally homeomorphic, i.e.,
\begin{align*}
	\mathcal{B}_\psi=\big\{x\in X:i(x,\psi)>1 \big\}.
\end{align*}

In the classical setting, as well as in a number of generalizations, the branch set of a quasiregular mapping has measure zero, as does its image. We will see that in general, this need not be the case, and in some situations, it is still unknown, and so we need a substitute. If $\nu$ is doubling, then we define the \textit{essential local index} $i_{\ess}(x,\psi)$ of $\psi$ at $x$ to be the quantity
\begin{align*}
i_{\ess}(x,\psi)=\limsup_{r\to 0}\frac{\psi^*\nu(U(x,\psi,r))}{\nu(B(\psi(x),r))}.
\end{align*}

\begin{lemma}\label{lemma:on essential index}
	It always holds that $1\leq i_{\ess}(x,\psi)\leq i(x,\psi)$. Moreover, if 
	$x$ is a $\psi$-Lebesgue point for $D_n$, where $D_n=\{z\in D:N(\psi(z),f,D)=n\}$, then $i_{\ess}(x,\psi)=1$. In particular, if $\nu$ is doubling, then $i_{\ess}(x,\psi)=1$ $\psi^*\nu$-a.e. in $X$.
\end{lemma}
\begin{proof}
 The inequalities $1\leq i_{\ess}(x,\psi)\leq i(x,\psi)$ follows trivially from the definition and the fact that $1\leq N(y,\psi,U(x,\psi,r))\leq i(x,\psi)$ for all $r$ sufficiently small and all $y\in B(\psi(x),r)$. If $x$ is a $\psi$-Lebesgue point for $D_n$, then 
 \begin{align*}
 	\lim_{r\to 0}\frac{\psi^*\nu(U(x,\psi,r))}{\psi^*\nu(U(x,\psi,r)\cap D_n)}=1,
 \end{align*}
 which implies, by definition and Proposition~\ref{prop:existence of normal nbhd} (6), that $i_{\ess}(x,\psi)=1$. The last assertion follows immediately from Corollary~\ref{coro:Lebesgue point pb}. 
\end{proof}

\newpage
\section{The pullback factorization}\label{sec:the pullback factorization}
In this section, we explore some elementary properties of branched coverings between topological and metric spaces. Our main tool is the \textit{``pullback metric"}, which endows the domain and mapping with a number of nice properties.

In particular, the pullback metric causes a mapping to have \textit{bounded diameter/length distortion (BDD/BLD)}, and therefore provides a rich source of examples and counterexamples in the study of BLD mappings. 

Though all the objects under consideration will be metric spaces for most of the
paper, for the moment we are uninterested in the metric in the source space, and
we prefer to ignore it in order to emphasize the general nature of our construction.

\subsection{The pullback metric}\label{subsec:The pullback metric}
Let $X$ be a connected topological space and $Y=(Y,d_Y)$ a metric space. Suppose $f\colon X\to Y$ is a continuous mapping. We define the ``\textit{pullback metric}" $f^*d_Y\colon X\times X\to [0,\infty)$ as follows:
\begin{align}\label{eq:def for pullback metric}
	f^*d_Y(x_1,x_2)=\inf_{x_1,x_2\in \alpha}\diam\big(f(\alpha)\big),
\end{align}
where the infimum is taken over all continua $\alpha$ joining $x_1$ and $x_2$ in $X$.

It is immediate from the definition that $f^*d_Y$ satisfies the triangle inequality. Moreover, the connectivity assumption on $X$ guarantees that it is finite. Thus even with no further assumptions on $X$, $Y$ or $f$, $f^*d_Y$ is a pseudo-metric on $X$. If we assume further that $f$ is discrete (or, for that matter, even \textit{light}, which means that the preimage of each point in $Y$ is totally disconnected in $X$), then $f^*d_Y$ becomes a genuine metric. We denote by $X^f$ or $f^*Y$ the metric space $(X,f^*d_Y)$.

\begin{remark}\label{rmk:on pullback via length}
Another natural candidate to serve as a pullback metric will be the following
\begin{align}\label{eq:def for pullback metric via length}
	f^*d_Y(x_1,x_2)=\inf_{x_1,x_2\in \alpha}l(f(\alpha)),
\end{align}	
where the infimum is taken over all curves $\alpha$ joining $x_1$ and $x_2$ in $X$. A lot of properties that we are going to prove below can be alternatively proved using this metric. However, the typical assumption for the metric space $X$ will be \textit{quasiconvexity}, which is a priori too strong.
\end{remark}

\subsection{Canonical factorization}\label{subsec:Canonical factorization}
From here on out we will always assume that $X$ and $Y$ are locally compact, connected, locally connected metric spaces, and that $f\colon X\to Y$ is \textit{discrete}, \textit{open}, and \textit{of locally bounded multiplicity}. We refer to these as the standing assumptions for this section.

In what follows, let $g\colon X\to X^f$ be the identity map, and let $\pi\colon X^f\to Y$ satisfy $\pi\circ g=f$, so that on the level of sets, $f=\pi$. We refer this canonical factorization as the \textit{pullback factorization for} $f$. 

\begin{figure}[h]
\[
		\xymatrix{
			X \ar[r]^{g} \ar[dr]_{f} & X^f \ar[d]^{\pi} \\
			& Y
		}
\]
\caption{The canonical pullback factorization between metric spaces}\label{Fig:pullback factorization}
\end{figure}

Under appropriate conditions, the metric space $X^f$ and the branched covering $\pi$ are well behaved. In order to make precise these nice behaviors, we need to recall some of the nice mapping classes between metric spaces.

\begin{definition}\label{def:BLD mapping}
	A branched covering $f\colon X\to Y$ between two metric spaces is said to be an \textit{$L$-BLD}, or a \textit{mapping of $L$-bounded length distortion}, $L\geq 1$, if 
	\begin{align*}
	L^{-1}l(\alpha)\leq l(f\circ \alpha)\leq Ll(\alpha)
	\end{align*}
	for all non-constant curves $\alpha$ in $X$, where $l(\gamma)$ denotes the length of a curve $\gamma$ in a metric space.
\end{definition}
The definition of BLD mappings is clearly only interesting if the metric spaces $X$ and $Y$ have a reasonable supply of rectifiable curves and so the most natural setting in which we study such mappings is that of quasiconvex metric spaces; we recall that a metric space $X$ is \textit{$c$-quasiconvex} if every
pair of points $x_1,x_2\in X$ may be joined with a curve $\gamma$ of length $l(\gamma)\leq cd(x_1,x_2)$. When the constant $c$ is unimportant, we omit it.

Since our construction of pullback metric generally only requires the control of diameters of sets, a more appropriate setting for our results will be the more general class of metric spaces of bounded turning. Recall that $X$ has \textit{$c$-bounded turning} if every pair of points $x_1,x_2\in X$ can be joined by a continuum $E\subset X$ such that $\diam E\leq cd(x_1,x_2)$. Note that by the local connectivity, we also have local path connectivity, and we may use curves instead of general continuua in the definition of bounded turning. It is elementary to verify, as in the case of quasiconvex spaces, that the infimum of the diameters of continuua or curves joining two points is realized, provided $X$ is assumed to be complete as well as locally compact.

When $X$ and $Y$ have bounded turning, the natural branched analog of a bi-Lipschitz homeomorphism is what we call a \textit{mapping of bounded diameter distortion}, which is defined in analogy with BLD mappings:
\begin{definition}\label{def:BDD mapping}
	A branched covering $f\colon X\to Y$ between two metric spaces is said to be an \textit{$L$-BDD}, or a \textit{mapping of $L$-bounded diameter distortion}, $L\geq 1$, if 
	\begin{align*}
	L^{-1}\diam(\alpha)\leq \diam(f\circ \alpha)\leq L\diam(\alpha)
	\end{align*}
	for all non-constant curves $\alpha$ in $X$.
\end{definition}

It follows directly from the definition of arc-length that an $L$-BDD mapping is $L$-BLD as well, regardless of any connectivity assumptions on either $X$ or $Y$. We will see later that if $Y$ is $c$-quasiconvex and $f$ is $L$-Lipschitz and $L$-BLD with $N=N(f,X)<\infty$, then $f$ is $LcN$-BDD. In particular, every $L$-BLD mapping $f\colon X\to Y$ between length spaces is $LN$-BDD.

As we will see in a moment, the metric space $X^f$ retains the original topology of $X$ and is often rather well-behaved: having 1-bounded turning and inheriting many metric and geometric properties from $Y$. Moreover, the branched covering $\pi$ from the pullback factorization is easily seen to be 1-Lipschitz and 1-BDD (and a fortiori 1-BLD).

\subsection{Fine properties of the pullback metric}
In this section, we fix a branched covering $\pi\colon X\to Y$ between two metric spaces. However, we will typically consider the space $X^\pi$ (or $\pi^*Y$) by endowing the set $X$ with the pullback metric $\pi^*d_Y$.

Recall also that for each metric $d$ on a topological space $X$, the length metric $l_d(z_1,z_2)$ is given by infimizing the lengths of all curves joining $z_1$ and $z_2$. For a metric space $X=(X,d_X)$, we denote by $X^l$ the length space $(X,l_{d_X})$.

We caution the reader that there are subtleties involved depending on the order in which one pulls back the metric, restricts the mapping to a neighborhood, or passes to the length metric. That is, in general, $l_{\pi^*d_Y}$ need not coincide with $\pi^*l_{d_Y} $, nor must $(\pi|_U)^*d_Y$ necessarily coincide
with $(\pi^*d_Y)|_U$. Nevertheless, we will see in a moment that these distinctions are rather minor.

We need the following well-known path-lifting result of Floyd~\cite[Theorem 2]{f50}.
\begin{lemma}\label{lemma:lift Floyd}
Let $\pi\colon X\to Y$ be a proper branched covering and let $\gamma\colon [a,b]\to Y$ be a path. Then for each $x\in \pi^{-1}\big(\gamma(a)\big)$, there is a curve $\tilde{\gamma}\colon [a,b]\to X$ such that $\tilde{\gamma}(a)=x$ and $\gamma=\pi\circ \tilde{\gamma}$.
\end{lemma}

Another important fact (see e.g.~Proposition~\ref{prop:existence of normal nbhd}) which we will use repeatedly is that for each $x\in X$, the sets $U(x,\pi,r)$, $r>0$, form a neighborhood basis of $z$ in the topology of $X$, where from now on, $U(x,\pi,r)$ denotes the $x$-component of $\pi^{-1}(B(\pi(x),r))$.

\begin{remark}\label{rmk:on light open map}
	i). For simplicity, we are not working in quite as much generality as we could. In fact, it would often suffice for our purposes to work in the generality of \emph{light open} mappings.	However, locally finite (and indeed, locally bounded) multiplicity is necessary for a number of our strongest results. Moreover, the theory of BLD mappings (and more generally, quasiregular mappings) between manifolds, or even generalized manifolds, typically requires $\pi$ to be light, open and sense-preserving (or sense-reversing), which in that setting is equivalent to the condition that $\pi$ is discrete, open, and of locally bounded multiplicity~\cite{hs02}. 
	
	ii). It should be noted, however, that in general, local compactness of $Z$ and discreteness of $\pi$ imply locally finite multiplicity (i.e., each $x\in X$ has a neighborhood $U$ such that for all $y\in Y$, $N(y,\pi,U)<\infty$), but \emph{not} locally bounded multiplicity. For example, let
	\begin{align*}
		X=Y=\bigcup_{i=1}^\infty\Big\{(t,it):0\leq t\leq \frac{1}{i^2}\Big\}\subset \mathbb{R}^2,
	\end{align*}
	with $X$ equipped with the subspace topology and $Y$ inheriting the Euclidean metric from $\R^2$. Let $\pi\colon X\to Y$ be given by $\pi(t,it)=(\frac{i^2t}{k^2},\frac{i^2t}{k})$ whenever $2^k\leq i<2^{k+1}$. Then $X$ is compact, $\pi$ is discrete and open, and for every
	neighborhood $U\subset X$ of the origin, we have $N(\pi,U)=\infty$, even though for each $y\in Y$, $N(y,\pi,U)<\infty$.
\end{remark}

\begin{remark}\label{rmk:on continuum and path}
	It is an elementary topological fact that a locally compact, connected, locally connected metric space is path connected, and thus every connected open subset is path connected as well. It also follows from a simple diagonalization argument that locally, the infimum in inequality~\eqref{eq:def for pullback metric} is in fact attained, i.e., at every $x\in X$, there is a
	neighborhood $U$ of $x$ such that for every $x'\in U$, there is a continuum $\alpha$ joining $x$ and $x'$ with
	\begin{align*}
		f^*d_Y(x,x')=\diam(f(\alpha)).
	\end{align*}
	It is also perhaps of interest that by the continuity and openness
	of $f$, every continuum $\alpha$ has a connected open (and a fortiori path connected) neighborhood $\alpha_\varepsilon$ with 
	\begin{align*}
		\diam(f(\alpha_\varepsilon))\leq \diam(f(\alpha))+\varepsilon
	\end{align*}
	for each $\varepsilon>0$. Thus, in the definition of pullback metric,
	we could just as well have required $\alpha$ to be (the image of) a curve, or a connected open set, though in the latter case the infimum in~\eqref{eq:def for pullback metric} need not necessarily be realized.
\end{remark}


For simplicity, we will formulate many of our basic results for the case that $\pi$ is \textit{proper} (i.e.~the preimage of each compact set is compact) and surjective, and $N(\pi,X)<\infty$. We lose very little generality with this reduction, in light of the following considerations.

First, if $\pi$ is an arbitrary branched covering, then at every $x\in X$, there is a radius $r>0$ for which $U(x,\pi,r)$ is relatively compact. From that and the local boundedness of the multiplicity
of $\pi$, we have
\begin{align*}
	N(\pi|_{U(x,\pi,r)},U(x,\pi,r))=N(\pi, U(x,\pi,r))<\infty.
\end{align*}
Since $\pi$ is open, $\pi(\partial U(x,\pi,r))=\partial \psi(U(x,\pi,r))$, so that 
$$\pi|_{U(x,\pi,r)}\colon U(x,\pi,r)\to \pi(U(x,\pi,r))$$
is proper and surjective as well, and so local restrictions of $\pi$ satisfy the more restrictive conditions.

Secondly, it follows easily from the local connectivity of $X$ that for any two open subsets $U,V\subset X$, the pullback metrics $(\pi|_U)^*d_Y$ and $(\pi|_V)^*d_Y$ are locally isometric on $U\cap V$ (see e.g. Lemma~\ref{lemma:pullback property 4} below).

As a result of the above considerations, when analyzing the metric space $X^\pi$, it is often enough to apply our followingi results to the metric spaces $(\pi|_{U(x,\pi,r)})^*Y$, and then invoke the local
isometries between these spaces and $X^\pi$.

Up to the end of this section, we will assume the metric space $Y$ is \textit{proper}, i.e. closed bounded balls are compact. 

\begin{lemma}\label{lemma:pullback property 1}
	The metric space $X^\pi$ is a proper metric space, homeomorphic to $X$ via the indentity mapping $g$. Open and closed balls in $X^\pi$ are connected and $X^\pi$ have 1-bounded turning. The projection mapping $\pi\colon X^\pi\to Y$ is 1-Lipschitz, 1-BDD and for each $z\in X^\pi$,
	\begin{align}\label{eq:pullback property 1}
		B(z,r)\subset U(z,\pi,r)\subset B(z,2r).
	\end{align}
\end{lemma}
\begin{proof}
	It follows immediately from the definition that for each $z\in X^\pi$ and $r>0$,
	\begin{align*}
		B(z,r)\subset U(z,\pi,r)\subset B(z,2r).
	\end{align*} 
	Since the sets $U(z,\pi,r)$ form a base for the topology on $X^\pi$, it follows that so do the balls $B(z,r)$, whereby the metric $\pi^*d_Y$ induces the original topology. Moreover,
	by properness of $Y$ and of $\pi$, each set $U(z,\pi,r)$ is relatively compact, and so $X^\pi$ is proper.
	
	Since for every $z_1,z_2\in X^\pi$ and each continuum $\alpha\subset X^\pi$ joining $z_1$ and $z_2$, we have
	\begin{align*}
		d_Y(\pi(z_1),\pi(z_2))\leq \diam(\pi\circ \alpha).
	\end{align*}
	Taking the infimum over all such continuua gives that $\pi$ is 1-Lipschitz, and so we have $\diam(\pi(\alpha))\leq \diam(\alpha)$ for each continuum $\alpha\subset X^\pi$. The reverse inequality follows immediately from the definition, and so we obtain $\diam(\pi(\alpha))=\diam(\alpha)$. Thus $\pi$ is 1-BDD, and a fortiori 1-BLD. That $X^\pi$ has 1-bounded turning, and that open and closed balls are connected, follows immediately from the definition and the fact that the infimum in~\eqref{eq:def for pullback metric} is realized per Remark~\ref{rmk:on continuum and path}.
\end{proof}

For the next result, recall that a continuous mapping $f\colon X\to Y$ between two metric spaces is said to be \textit{$c$-co-Lipschitz} if for all $x\in X$ and $r>0$, 
\begin{align*}
	B(f(x),r)\subset f(B(x,cr)).
\end{align*}
Recall also that we have assumed that $N:=N(\pi,X)<\infty$.

\begin{lemma}\label{lemma:pullback property 2}
	If $Y$ has $c$-bounded turning, then $\pi$ is $c$-co-Lipschitz and is locally $c$-bi-Lipschitz on each set $X_k:=\{z\in X^\pi:N(z,\pi,X)=k\}$. If, additionally, $Y$ is Ahlfors $Q$-regular with constant $c_2$, then $X^\pi$ is Ahlfors $Q$-regular with constant $c^Qc_2N$. Moreover, for each $k=1,\dots,N$, and at each Lebesgue point $z$ of $X_k$, we may take the pointwise constant of $Q$-regularity to be $c^Qc_2$.
\end{lemma}
\begin{proof}
	If $Y$ has $c$-bounded turning, then for each $z_0\in X^\pi$ and $r>0$, every point $y\in B(\pi(z_0),r)$ may be joined to $\pi(z_0)$ with a continuum $\alpha$ of diameter at most $cr$. We have by Lemma~\ref{lemma:lift Floyd} that there is a continuum $\tilde{\alpha}$ containing $z_0$ such that $\pi(\tilde{\alpha})=\alpha\ni y$. Thus there
	is some point $z\in \title{\alpha}$ with $\pi(z)=y$, so that 
	\begin{align*}
		\pi^*d_Y(z_0,z)\leq \diam(\tilde{\alpha})=\diam(\alpha)\leq cr.
	\end{align*}
	Thus $y\in \pi(B(z_0,cr))$ and so $\pi$ is $c$-co-Lipschitz. It follows easily from local compactness and Proposition~\ref{prop:existence of normal nbhd} (5) that the multiplicity function $N(y,\pi,X^\pi)$ is lower-semicontinuuous in $Y$, which in turn implies that $\pi$ is locally bijective on each of the sets $X_k$. From this and the fact that $\pi$ is 1-Lipschitz and $c$-co-Lipschitz we obtain that $\pi$ is locally $c$-bi-Lipschitz on each $X_k$. 
	
	We next turn to the second claim. Note that $\pi$ is 1-Lipschitz, $c$-co-Lipschitz and that $Y$ is Ahlfors $Q$-regular with constant $c_2$, so we have
	\begin{align*}
		\mathcal{H}^Q\big(B(x,r)\big)\geq \mathcal{H}^Q\big(\pi(B(x,r))\big)\geq \mathcal{H}^Q\big(B(\pi(x),r/c)\big)\geq \frac{1}{c_2c^Q}r^Q. 
	\end{align*}
	For the reverse direction, we first assume that $\pi$ is $c$-bi-Lipschitz on $X_k$ and write $B(x,r)=\bigcup_{k=1}^NB_k(x,r)$, where $B_k(x,r)=B(x,r)\cap X_k$. Then it follows from the subadditivity of the Hausdorff measure that
	\begin{align*}
		\mathcal{H}^Q(B(x,r))\leq \sum_{k=1}^{N}\mathcal{H}^Q(B_k(x,r))\leq N\max_{1\leq k\leq N}\mathcal{H}^Q(B_k(x,r)).
	\end{align*}
	For each $k$, note that $\pi^{-1}\colon \pi(B_k(x,r))\to B_k(x,r)$ is $c$-Lipschitz and that $\pi$ is 1-Lipschitz, and we conclude 
	\begin{align*}
	\mathcal{H}^Q(B_k(x,r))=\mathcal{H}^Q(\pi^{-1}\circ\pi(B_k(x,r)))\leq c^Q\mathcal{H}^Q(\pi(B(x,r)))\leq c^Q\mathcal{H}^Q(B(\pi(x),r)).
	\end{align*}
	In the general case, we may use Lemma~\ref{lemma:decomposition Dn} to decompose each $B_k(x,r)$ as a disjoint union of Borel sets on which $\pi$ is $c$-bi-Lipschitz. Then a similar computation leads to the estimate $\mathcal{H}^Q(B(x,r))\leq Nc_2c^Qr^Q$.  
	In conclusion, we have shown
	\begin{align*}
		\frac{1}{c_2c^Q}r^Q\leq \mathcal{H}^Q(B(x,r))\leq Nc_2c^Qr^Q.
	\end{align*}
	At each Lebesgue point $z$ of $X_k$, by Proposition~\ref{prop:existence of normal nbhd}, $\pi$ is injective in a sufficiently small neighborhood of $z$ and so we may repeat the above argument with $N=1$, which gives the desired  pointwise constant $c^Qc_2$. 
\end{proof}

For the next result, recall that a point $x$ in a metric space $X$ is called a \textit{local cut point} if $U\backslash \{x\}$ is disconnected for some neighborhood $U$ of $x$. 

\begin{lemma}\label{lemma:pullback property 3}
   If $Y$ is $c$-LLC-2 and $X$ has no local cut points, then $X^\pi$ is locally $2c$-LLC.  
\end{lemma}
\begin{proof}
By Lemma~\ref{lemma:pullback property 1}, we only need to check the LLC-2 property of $X^\pi$. Since $X$ has no local cut points, so is $X^\pi$. It follows that for each $z_0\in X^\pi$ and $r_0>0$ such that $\pi^{-1}(\pi(z_0))\cap \overline{U(z_0,\pi,r_0)}=\{z_0\}$, there is some $R<\frac{r_0}{2}$ such that for each $z\in U(z_0,\pi,\frac{r_0}{2})$ and $r<R$, every pair of points $z_1',z_2'\in \partial U(z_0,\pi,r_0)$ can be joined by a continuum $\alpha$ such that $\alpha$ is disjoint from $U(z,\pi,r)$. Therefore, let $z\in U(z_0,\pi,\frac{r_0}{2})$ and $r<R$, and suppose $z_1,z_2\in X^\pi\backslash U(z,\pi,r)$. 
	
	Let $y_0=\pi(z_0)$, $y=\pi(z)$, and $y_i=\pi(z_i)$ for $i=1,2$.
	
	 First suppose additionally that $y_1,y_2\in Y\backslash B(y,r)$. Then by the LLC-2 property, we may join $y_i$ to $\partial B(y_0,r_0)$ with a continuum
	$\alpha_i\subset Y\backslash B(y,r/c)$, for $i=1, 2$. By Lemma~\ref{lemma:lift Floyd}, we may lift each continuum $\alpha_i$ to a continuum $\tilde{\alpha}_i$ containing $z_i$, which clearly intersects $\partial U(z_0,\pi,r_0)$, say at $z_i'$. Since $r<R$,
	there is a continuum $\alpha'\subset X^\pi\backslash U(z,\pi,r)$ joining 
	$z_1'$ and $z_2'$, and so 
	\begin{align*}
		\tilde{\alpha}_1\cup \alpha\cup \tilde{\alpha}_2\subset X^\pi\backslash U(z,r/c)
	\end{align*}
	is a continuum joining $z_1$ and $z_2$. Since $B(z,r/c)\subset U(z,\pi,r/c)$ and $U(z,\pi,r)\subset B(z,2r)$, we conclude that $X^\pi$ is locally $2c$-LLC.
	
	Next, suppose $y_1\in B(y,r)$ and $y_2\in Y\backslash B(y,r)$. Then $U(z_1,\pi,r)\cap U(z,\pi,r)=\emptyset$. Let $\alpha_2$, $\alpha_2'$ and $z_2'$ be as in the preceding proof. It is clear that we only need to construct a continuum $\alpha_1$ containing $y_1$ so that its lift $\alpha_1'$ is disjoint from $U(z,\pi,r)$ and intersects $\partial U(z_0,\pi,r_0)$ at $\tilde{z}_1'$. We first select a continuum $\beta_1$ that joins $y_1$ to some point $y_1'\in\partial B(y,r)$ so that its lift $\beta_1'$ is disjoint from $U(z,\pi,r)$ and joins $z_1$ to some point $z_1'\in\partial U(z_1,\pi,r)$. Then since $y_1'\in Y\backslash B(y,r)$, we may join $y_1'$ to some point $\tilde{z}_1\in \partial B(y_0,r_0)$ with a continuum $\gamma_1\subset Y\backslash B(y,r/c)$. As before, we may lift $\gamma_1$ to a continuum $\gamma_1'$ containing $z_1'$, which clearly intersects $\partial U(z_0,\pi,r)$, say at $\tilde{z}_1'$. Then $\alpha_1':=\beta_1'\cup \gamma_1'$ will be a continuum joining $z_1$ to $\partial U(z_0,\pi,r_0)$ that is disjoint from $U(z,\pi,r/c)$.
	
	Finally, suppose $y_1,y_2\in B(y,r)$. We may repeat the preceding arguments for both $y_1$ and $y_2$ to obtain two continua $\alpha_1'$ and $\alpha_2'$ that are disjoint from $U(z,\pi,r/c)$. Moreover, for each $i=1,2$, $\alpha_i'$ joins $z_i$ to $\partial U(z_0,\pi,r_0)$. We may repeat the arguments in the first case to complete the proof.
	
\end{proof}

\begin{lemma}\label{lemma:pullback property 4}
   If $l_{d_Y}$ and $d_Y$ induces the same topology on $Y$, then the metrics $\pi^*d_Y$, $l_{\pi^*d_Y}$, $\pi^*l_{d_Y}$, and $l_{\pi^*l_{d_Y}}$ induces the same topology on $X$. The length of a curve in $X^\pi$ is the same with respect to any of these four metrics, and moreover,
   \begin{align*}
   	\pi^*d_Y\leq \pi^*l_{d_Y}\leq l_{\pi^*d_Y}=l_{\pi^*l_{d_Y}}\leq (2N-1)\pi^*l_{d_Y}.
   \end{align*}
   In particular, the metric space $\pi^*(Y^l)$ is $(2N-1)$-quasiconvex and if $Y$ is $c$-quasiconvex, then $X^\pi$ is $(2N-1)c$-quasiconvex. 
\end{lemma}
\begin{proof}
	Since $Y$ and $Y^l$ are homeomorphic by the identity, the pullback metric $\pi^*l_{d_Y}$ induces the same topology on $X$ as $\pi^*d_Y$ (namely, the original one). The inequalities
	\begin{align*}
		\pi^*d_Y\leq \pi^*l_{d_Y}\leq l_{\pi^*l_{d_Y}}
	\end{align*}
	are trivial since $d\leq l_d$. Note that for an arbitrary metric $d$, $\gamma$ is a rectifiable curve with respect to $d$ if and only if it is with respect to $l_d$, and in that case, the lengths coincide. It follows from this and the fact that $\pi$ is 1-BLD, for any curve $\gamma\subset X^\psi$,
	\begin{align*}
		l_{\pi^*d_Y}(\gamma)=l_{d_Y}(\pi(\gamma))=l_{l_{d_Y}}(\pi(\gamma))=l_{\pi^*l_{d_Y}}(\gamma).
	\end{align*}
	This means that $l_{\pi^*d_Y}=l_{\pi^*l_{d_Y}}$. 
	
	Now, note that since $\pi$ is discrete, by Lemma~\ref{lemma:lift Floyd} and the facts that $\pi$ is 1-BLD and $Y^l$ is a length
	space, we have that every $z\in X^\pi$ has a radius $r_z>0$ for which 
	\begin{align*}
		B_{\pi^*l_{d_Y}}(z,r)=B_{l_{\pi^*l_{d_Y}}}(z,r)
	\end{align*}
	for $r<r_z$, and so $l_{\pi^*l_{d_Y}}$ induces the same topology as $\pi^*l_{d_Y}$. Since the former is equal to $l_{\pi^*d_Y}$, and the latter was already remarked to induce the same topology as $\pi^*d_Y$, we obtain
	the desired topological equivalence. 
	
	Finally, let $\alpha\subset \pi^*(Y^l)$ be a continuum in $X$, with
	$\diam_{\pi^*l_{d_Y}}(\alpha)\leq r$. Then by the 1-BDD property for pullback metrics, $\diam_{l_{d_Y}}(\pi(\alpha))\leq r$. Thus $\pi(\alpha)$ is contained in a closed $l_{d_Y}$-ball of radius $r$, and so by Lemma~\ref{lemma:lift Floyd}, $\alpha$ is contained in a union of $N$ closed $l_{\pi^*l_{d_Y}}$-balls of radius $r$; since $\alpha$ is also connected in the topology induced by $l_{\pi^*l_{d_Y}}$, it follows that 
	\begin{align*}
		\diam_{l_{\pi^*l_{d_Y}}}(\alpha)\leq (2N-1)r,
	\end{align*}
	and so we obtain the inequality $l_{\pi^*l_{d_Y}}\leq (2N-1)\pi^*l_{d_Y}$.
\end{proof}
\begin{remark}\label{rmk:on pullback=orignial}
	In the special case where $X=Y$ and $\pi$ is the identity, the pullback metric $\pi^*d_Y$ coincides with $d_Y$ if and only if $Y$ has 1-bounded turning, and more generally the two are
	$c$-bi-Lipschitz equivalent if and only if $Y$ has $c$-bounded turning. In particular, they coincide when $Y$ is a length space, and are $c$-bi-Lipschitz equivalent when $Y$ is $c$-quasiconvex.
\end{remark}

\subsection{Fine properties of the pullback factorization}\label{subsec:Fine properties of the pullback factorization}
Recall that for any branched covering $f\colon X\to Y$ between two metric spaces, we have the pullback factorization $f=\pi\circ g$, where
$g\colon X\to X^f$ is the identity mapping and $\pi\colon X^f\to Y$ is the branched covering given by $\pi=f$. 

By Lemma~\ref{lemma:pullback property 1}, $\pi$ is a 1-BDD mapping and thus we have factorized $f$ into a composition of a homeomorphism and a 1-BDD ``projection". Note that while on the level of sets, we are factoring out the identity, the mapping $g$ will typically not be an isometry. But we have seen already, the projection mapping $\pi$ can be thought of as being as close to an isometry as possible. Thus, philosophically, we have factored $f$ into a geometric equivalence composed with a topological one.

On the other hand, as a result of the fact that $\pi$ is 1-BDD, $f$ and $g$ share many geometric properties.
\begin{proposition}\label{prop:nice properties of pullback factorization I}
$f$ is $L$-BDD if and only if $g$ is $L$-BDD. 
\end{proposition}
\begin{proof}
	If $g$ is $L$-BDD, then it follows immediately from the fact $\pi$ is 1-BDD that $f$ is $L$-BDD as well. For the reverse direction, simply notice that if $g$ is not $L$-BDD, then there exists a non-constant curve $\gamma\subset X$ such that 
	$$\text{either } \diam(g(\gamma))>L\diam(\gamma) \quad\text{ or }\quad \diam(g(\gamma))<\frac{1}{L}\diam(\gamma).$$
	Taking into account the fact that $\pi$ is 1-BDD, it means 
	$$\text{either } \diam(\pi\circ g(\gamma))>L\diam(\gamma) \quad\text{ or }\quad \diam(\pi\circ g(\gamma))<\frac{1}{L}\diam(\gamma),$$
	which contradicts with our assumption $f$ being $L$-BDD.
	
\end{proof}

Since $X^f$ has 1-bounded turning, the latter implies that $g^{-1}\colon X^f\to X$ is $L$-Lipschitz. Indeed, for any $z_1,z_2\in X^f$, we may select a curve $\alpha\subset X^f$, joining $z_1$ and $z_2$, so that $f^*d_Y(z_1,z_2)=\diam (\alpha)$. Then $\tilde{\alpha}=g^{-1}(\alpha)$ will be a curve in $X$ that connects $g^{-1}(z_1)$ and $g^{-1}(z_2)$. Thus the $L$-BDD property of $f$ and 1-BDD property of $\pi$ give us
\begin{align*}
	d_X(g^{-1}(z_1),g^{-1}(z_2))&\leq \diam(\tilde{\alpha})\leq L\diam(f(\tilde{\alpha}))=L\diam(\pi(\alpha))\\
	&=L\diam(\alpha)=Lf^*d_Y(z_1,z_2).
\end{align*}
Thus when $g$ is $L$-Lipschitz and $L$-BDD, $g$ is an $L$-bi-Lipschitz equivalence between $X$ and $X^f$, whereby $X$ has $L$-bounded turning. Conversely, if $f$ is $L$-BDD, and $X$ has $c$-bounded turning, then $g$ is $Lc$-Lipschitz.

Similar to the BDD case (with indeed the same proof, but using only the 1-BLD property of $\pi$), we have the following useful conclusion.
\begin{proposition}\label{prop:nice properties of pullback factorization II}
	$f$ is $L$-BLD if and only if $g$ is $L$-BLD. 
\end{proposition}

As in the BDD case, if $g$ is $L$-BLD, then $(g^l)^{-1}$ is $L$-Lipschitz, where $g^l\colon X\to (X^f)^l$ is the identity, and so when $f$ is $L$-Lipschitz and $L$-BLD, $g^l$ gives a bi-Lipschitz equivalence between $X$ and $(X^f)^l$.

Thus much of the theory of BLD and BDD mappings reduces to the study of the pullback metric. Similarly, it turns out that under various definitions and for many different levels of generality, $g$ is quasiconformal if and only if $f$ is quasiregular. Thus one obtains a canonical factorization of a quasiregular mapping into a composition of a quasiconformal mapping
with a 1-BDD mapping. This is particularly useful in extending the theory of quasiregular mappings to the metric setting, as the quasiconformal theory has at present advanced much further than its branched counterpart in this generality. In fact, one of the motivations for our exploration of the pullback metric is to establish the equivalence of various geometric, metric, and analytic characterizations of quasiregularity, which we explore in Section~\ref{sec:foundations of QR mappings in mms}; this equivalence
has been established already in various contexts for definitions involving ``outer" dilatation, (e.g., the geometric $K_O$-inequality) which tend to mimic the proofs in the quasiconformal case, but the problem is substantially more delicate, even in the classical case, when inner dilatation is concerned.

\begin{proposition}\label{prop:BLD to BDD}
	Suppose that $Y$ is $c$-quasiconvex, and that $f$ is $L$-Lipschitz and $L$-BLD with $N=N(f,X)<\infty$. Then $f$ is $cNL$-BDD. 
\end{proposition}
\begin{proof}
Since $f$ is $L$-Lipschitz, $\diam(f(\alpha))\leq L\diam(\alpha)$ for each continuum $\alpha\subset X$. On the other hand, by the above discussion, $g$ yields an $L$-bi-Lipschitz equivalence between the length space $X^l$ and $(\pi^*Y)^l$, and so by Lemma~\ref{lemma:pullback property 4}, we have
\begin{align*}
	\diam(\alpha)\leq \diam_{l_{d_X}}(\alpha)\leq L\diam_{l_{\pi^*Y}}(g(\alpha))\leq cNL\diam(g(\alpha))=cNL\diam(f(\alpha)),
\end{align*}
whereby $f$ is $cNL$-BDD.
\end{proof}

\newpage
\section{Foundations of QR mappings in metric measure spaces}\label{sec:foundations of QR mappings in mms}
Throughout the entire section, we will always assume that $X=(X,d_X,\mu)$ and $Y=(Y,d_Y,\nu)$ are locally compact complete metric measure spaces, that $\tilde{\Omega}\subset X$ is a domain, and that $f\colon \tilde{\Omega}\to \Omega\subset Y$ is an \textit{onto} branched covering. For notational simplicity, we will typically drop the subscript on the distance.

In what follows, we fix the pullback factorization $f=\pi\circ g$ as in Section~\ref{sec:the pullback factorization}, where $g\colon \tilde{\Omega}\to \tilde{\Omega}^f$ is the identity mapping and $\pi\colon \tilde{\Omega}^f\to \Omega$ is the 1-BDD projection. We equip $\tilde{\Omega}^f$ with the Borel regular measure $\lambda=\pi^*\nu=g_*f^*\nu$. As always, we simply write $\tilde{\Omega}^f$ for the metric measure space $(\tilde{\Omega}^f,f^*d_Y,\lambda)$.

\begin{figure}[h]
	\[
	\xymatrix{
		(\tilde{\Omega},\mu) \ar[r]^{g} \ar[dr]_{f} & (\tilde{\Omega}^f,\lambda) \ar[d]^{\pi} \\
		& (\Omega,\nu)
	}
	\]
	\caption{The pullback factorization between metric measure spaces}
\end{figure}

Notice that in the special case that $\nu=\mathcal{H}^s$, and $\Omega$ has $C$-bounded turning, it follows immediately from the decomposition Lemma~\ref{lemma:pullback property 2} that $\mathcal{H}^s\leq \lambda\leq C^s\mathcal{H}^s$. Since all Ahlfors $Q$-regular mesures
are comparable to $\mathcal{H}^Q$, it follows again from Lemma~\ref{lemma:pullback property 2} that if $\nu$ is pointwise Ahlfors $Q$-regular and $\Omega$ has locally bounded turning, then $\lambda$ is pointwise Ahlfors $Q$-regular as well,
and more precisely, for each $z\in \tilde{\Omega}^f$, and small enough $r$ (depending on $z$), we have 
\begin{align}\label{eq:estimate for lambda}
	C^{-1}r^Q\leq \lambda(B(z,r))\leq Ci_{\ess}(z,\pi)r^Q,
\end{align}
where $C\geq 1$ depends only on the bounded turning constant and the Ahlfors regularity constant of $\Omega$. Since $f^*\nu(U(x,f,r))=\pi^*\nu(U(g(x),\pi,r))$, we have (by definition)  $$i_{\ess}(x,f)=i_{\ess}(g(x),\pi).$$

To avoid notational confusion, we will denote by $|\nabla f|$ the minimal $p$-weak upper gradient of a Sobolev mapping $f\in N^{1,p}_{loc}(X,Y)$. Strictly speaking, we should also indicate the dependence of $|\nabla f|$ on $p$. But for our later applications, the context is often clear and so we drop it for simplicity.

\subsection{Definitions of quasireguarity in general metric measure spaces}\label{subsec:Definitions of quasireguarity in general metric measure spaces}
In this section, we introduce the different definitions of quasiregularity in general metric measure spaces. Without further notices, the script $Q$ appeared in this paper will always be assumed to be strictly greater than one.

The first one is the so-called weak metrically quasiregular mappings as mentioned in the introduction. Recall that $h_f(x)=\liminf_{r\to 0}H_f(x,r)$.
\begin{definition}[Weak metrically quasiregular mappings]\label{def:weak metric qr}
	A branched covering $f\colon \tilde{\Omega}\to \Omega$ is said to be \textit{weakly metrically $H$-quasiregular} if it satisfies 
	\begin{itemize}
		\item[i).] $h_f(x)<\infty$ for all $x\in \tilde{\Omega}$;
		\item[ii).] $h_f(x)\leq H$ for $\mu$-almost every $x\in \tilde{\Omega}$. 
	\end{itemize}
\end{definition}

The second definition is very commonly used in literature as it is defined pointwisely and is easy to deduce analytic properties of quasiregular mappings.
\begin{definition}[Analytically quasiregular mappings]\label{def:analytic def for qr}
	A branched covering $f\colon \tilde{\Omega}\to \Omega$ is said to be \textit{analytically $K$-quasiregular with exponent $Q$}  if $f\in N^{1,Q}_{loc}(\tilde{\Omega},\Omega)$ and 
	\begin{align*}
	|\nabla f|(x)^Q\leq KJ_f(x)
	\end{align*}
	for $\mu$-a.e. $x\in \tilde{\Omega}$. 
\end{definition}

The geometric definition requires some modulus inequalities between curve families.

\begin{definition}[Geometrically quasiregular mappings]\label{def:geometric def for qr}
	A branch covering $f\colon \tilde{\Omega}\to \Omega$ is said to be \textit{geometrically $K$-quasiregular with exponent $Q$} if it satisfies the \textit{$K_O$-inequailty with exponent $Q$}, \ie, for each open set $\tilde{\Omega}_0\subset \tilde{\Omega}$ and  each path family $\Gamma$ in $\tilde{\Omega}_0\subset \tilde{\Omega}$, if $\rho$ is a test function for $f(\Gamma)$, then
	\begin{align*}
	\Modd_Q(\Gamma)\leq K\int_{\Omega}N(y,f,\tilde{\Omega}_0)\rho^Q(y)d\nu(y).
	\end{align*}
\end{definition}

As pointed out in the introduction, we will refer to the metric definition $(M)$, the weak metric definition $(m)$, the analytic definition $(A)$, and the geometric definition $(G)$ as elements of the \textit{forward definitions}.

Next, we introduce the elements from the \textit{inverse definitions}: the inverse metric definition $(M^*)$, the inverse weak metric definition $(m^*)$, the inverse analytic definition $(A^*)$, and the inverse geometric definition $(G^*)$.

\begin{definition}[Inverse metrically quasiregular mappings]\label{def:inverse metric quasiregular map}
	A branched covering $f\colon \tilde{\Omega}\to \Omega$ between two metric measure spaces is termed \textit{inverse metrically $H$-quasiregular} if the inverse linear dilatation function $H_f^*$ is finite everywhere and essentially bounded from above by $H$. 
\end{definition}

\begin{definition}[Inverse weak metrically quasiregular mappings]\label{def:inverse weak metric quasiregular map}
	A branched covering $f\colon \tilde{\Omega}\to \Omega$ between two metric measure spaces is termed \textit{inverse weak metrically $H$-quasiregular} if it satisfies 
	\begin{itemize}
		\item[i).] $h_f^*(x)<\infty$ for all $x\in \tilde{\Omega}$;
		\item[ii).] $h_f^*(x)\leq H$ for $\mu$-almost every $x\in \tilde{\Omega}$. 
	\end{itemize}
\end{definition}

We need the pullback factorization to define the inverse analytically quasiregular mappings and the name of this terminology will become clear soon.

\begin{definition}[Inverse analytically quasiregular mappings]\label{def:inverse analytic def for qr}
	A branched covering $f\colon \tilde{\Omega}\to \Omega$ is said to be \textit{inverse analytically $K$-quasiregular with exponent $Q$} if $g^{-1}\in N^{1,Q}_{loc}(\tilde{\Omega}^f,\tilde{\Omega})$ and 
	\begin{align*}
	|\nabla g^{-1}|(z)^Q\leq KJ_{g^{-1}}(z)
	\end{align*}
	for $\lambda$-a.e. $z\in \tilde{\Omega}^f$. 
\end{definition}

The inverse geometric definition also relies on certain inequalities for the modulus of curve families.

\begin{definition}[Inverse geometric quasiregular mappings]\label{def:inverse geometric def for qr}
	A branched covering $f\colon \tilde{\Omega}\to \Omega$ is said to be \textit{inverse geometrically $K$-quasiregular with exponent $Q$} if it satisfies the \textit{$K_I$-inequailty} or the \textit{Poletsky's inequality with exponent $Q$}, \ie, for every curve family $\Gamma$ in $\tilde{\Omega}$, we have
	\begin{align*}
	\Modd_Q(f(\Gamma))\leq K\Modd_Q(\Gamma).
	\end{align*}
\end{definition}

We also introduce the following strong inverse geometrically quasiregular mappings and it will be useful in our later proofs of the standard \textit{V\"ais\"al\"a's inequality}.
\begin{definition}[Strong inverse geometrically quasiregular mappings]\label{def:strong inverse geometric def for qr}
	A branched covering $f\colon \tilde{\Omega}\to \Omega$ is said to be a \textit{strong inverse geometric $K$-quasiregular mapping with exponent $Q$} if it satisfies the following \textit{generalized V\"ais\"al\"a's inequality with exponent $Q$}: For each open subset $\tilde{\Omega}_0\subset \tilde{\Omega}$, each curve family $\Gamma$ in $\tilde{\Omega}_0$, $\Gamma'$ in $\Omega$, and for each $\gamma'\in \Gamma'$, there are curves $\gamma_1,\dots,\gamma_m\in \Gamma$ and subcurves $\gamma_1',\dots,\gamma_m'$ of $\gamma'$ such that for each $i=1,\dots,m$, $\gamma_i'=f(\gamma_i)$, and for almost every $s\in [0,l(\gamma')]$, $\gamma_i(s)=\gamma_j(s)$ if and only if $i=j$, then 
	\begin{align*}
		\Modd_Q(\Gamma')\leq \frac{K}{m}\Modd_Q(\Gamma).
	\end{align*}
\end{definition}

By~\cite[Theorem 9.8]{hkst01}, when $f\colon \tilde{\Omega}\to \Omega$ is a homeomorphism, and $\tilde{\Omega}$ and $\Omega$ have locally $Q$-bounded geometry, the inverse (metric, weak metric, analytic, geometric) definitions for $f$ are, quantitatively, the forward (metric, weak metric, analytic geometric) definitions for $f^{-1}$. Moreover, in this case, each of these definitions is further equivalent to the local quasisymmetry, quantitatively.

\subsection{Analytic and geometric definitions}\label{subsec:Analytic and geometric definitions}
When $f\colon \tilde{\Omega}\to \Omega$ is a homeomorphism, the equivalence of the analytic and geometric definitions has been shown in great generality (cf.~\cite[Theorem 1.1]{w12proc}), with the same dilatation $K_O$. With
our decomposition, we can expand this almost immediately to the general case; in
fact, it turns out that $f$ satisfies the $K_O$ inequality with exponent $Q$ if and only if its lift $g$ does. For the ``inverse" geometric definition, the situation is a bit more involved, but likewise the inequality holds for $f$ if and only if it holds for $g$.

Our main result of this section is the following very general equivalence result.
\begin{theorem}\label{thm:equivalence of G and A}
	Let $f\colon \tilde{\Omega}\to \Omega$ be a branched covering. Then $f$ is analytically $K_O$-quasiregular with exponent $Q$ if and only if it is geometrically $K_O$-quasiregular with exponent $Q$. Similarly, $f$ is inverse analytically $K_I$-quasiregular with exponent $Q$ if and only if it is strong inverse geometrically $K_I$-quasiregular with exponent $Q$. 	
\end{theorem}

The latter assertion in Theorem~\ref{thm:equivalence of G and A} provides us a useful analytic characterization of the Poletsky's inequality, which will be crucial in our proof of ``the metric definition implies all the others". This characterization also gives us an alternative way of showing the Poletsky's inequality for quasiregular mappings, namely, we just need to verify that the lifting mapping $g$ from the pullback factorization satisfies the inverse analytic definition of quasiconformality. The advantage of this alternative approach is that the metric/geometric information of the underlying metric measure spaces in practice is inherited very well at the level of its lifting space and is often very elementary to verify. There is no surprise that it is much easier to deal with homeomorphisms than the more general branched coverings. However, we caution the readers that the information on the branch set will be transfered to the pullback measure of the lifting space. Thus, instead of analyzing (uniform) Ahlfors regular spaces, we have to deal with pointwise Ahlfors regular spaces; see Section~\ref{subsec:Spaces of locally $Q$-bounded geometry} below for a precise meaning. 

\subsubsection{Auxiliary results}

We will need the following result later, in order to prove that $f$ is geometrically $K_I$-quasiregular with exponent $Q$ if and only if $g$ is geometrically $K_I$-quasiconformal with exponent $Q$. 

\begin{lemma}\label{lemma:1.2}
	Let $\rho\colon \tilde{\Omega}^f\to \bR$ be a Borel function and let $\xi$ be a locally finite Borel regular measure on $\tilde{\Omega}^f$. Suppose that for every open subset $\tilde{\Omega}^f_1\subset \tilde{\Omega}^f$ (not necessarily connected) such that $N(\pi,\tilde{\Omega}^f_1)<\infty$,
	\begin{align*}
	\int_{\Omega} \sup(\rho,y,\pi,\tilde{\Omega}^f)d\nu(y)\leq \xi(\tilde{\Omega}^f_1).
	\end{align*}
	Then $\rho\leq \frac{d\xi}{d\lambda}$ $\lambda$-a.e. in $\tilde{\Omega}^f$.
\end{lemma}
\begin{proof}
Fix an open subset $\tilde{\Omega}_1^f\subset \tilde{\Omega}^f$ with $N(\pi,\tilde{\Omega}^f_1)<\infty$. For each $y\in \Omega$ and each $z\in \pi^{-1}(y)$, we have
\begin{align*}
	\sup(\rho,y,\pi,\tilde{\Omega}_1^f)\geq \frac{\sum(\rho,y,\pi,\tilde{\Omega}_1^f)}{N(\pi(z),\pi,\tilde{\Omega}_1^f)}.
\end{align*}
It follows that
\begin{align*}
	\int_{\tilde{\Omega}_1^f}\frac{\rho(z)}{N(\pi(z),\pi,\tilde{\Omega}_1^f)}d\lambda(z)&=\int_{\Omega}\frac{\sum(\rho,y,\pi,\tilde{\Omega}_1^f)}{N(y,\pi,\tilde{\Omega}_1^f)}d\nu(y)\\
	&\leq \int_{\Omega}\sup(\rho,y,\pi,\tilde{\Omega}_1^f)d\nu(y)\leq \xi(\tilde{\Omega}_1^f).
\end{align*}
If $\tilde{\Omega}_2^f\subset \tilde{\Omega}_1^f$ is another open set, then for each $y\in \Omega$, $\frac{1}{N(y,\pi,\tilde{\Omega}_1^f)}\leq \frac{1}{N(y,\pi,\tilde{\Omega}_2^f)}$, and so we have
\begin{align}\label{eq:AG 2}
\int_{\tilde{\Omega}_2^f}\frac{\rho(z)}{N(\pi(z),\pi,\tilde{\Omega}_1^f)}d\lambda(z)\leq \xi(\tilde{\Omega}_2^f).	
\end{align}
Since inequality~\eqref{eq:AG 2} holds for all open subset $\tilde{\Omega}_2^f\subset \tilde{\Omega}_1^f$, and $\xi$ is locally finite and Borel regular, the inequality
\begin{align}\label{eq:AG 3}
	\frac{\rho(z)}{N(\pi(z),\pi,\tilde{\Omega}_1^f)}\leq \frac{d\xi}{d\lambda}(z)
\end{align}
holds for $\lambda$-a.e. $z\in \tilde{\Omega}_1^f$.

Finally, let $\{\tilde{\Omega}_i^f\}$ be a countable basis of open sets for the topology of $\tilde{\Omega}^f$, such that $N(\pi,\tilde{\Omega}_i^f)<\infty$ for each $i\in \mathbb{N}$. Then for $\lambda$-a.e. $z\in \tilde{\Omega}^f$, the inequality~\eqref{eq:AG 3} holds for each $i$ such that $z\in \tilde{\Omega}_i^f$. Since $\pi$ is discrete, there is some $i$ such that $N(z,\pi,\tilde{\Omega}_i^f)=1$, which, together with~\eqref{eq:AG 3}, completes our proof.
\end{proof}

As an easy consequence of the change of variable formula, we have the following $K_O$-inequality for the projection mapping $\pi$.
\begin{lemma}\label{lemma:1.3}
	Let $\tilde{\Omega}^f_1\subset \tilde{\Omega}^f$ and let $\Gamma$ be a family of curves in $\tilde{\Omega}^f_1$. Then for every Borel function $\rho$ admissible for $\pi(\Gamma)$,
	\begin{align*}
	\Modd_Q(\Gamma)\leq \int_\Omega \rho^Q(y)N(y,\pi,\tilde{\Omega}^f_1)d\nu(y).
	\end{align*}
\end{lemma}
\begin{proof}
	Suppose that $\rho$ is admissible for $\pi(\Gamma)$. Then for each $\gamma\in \Gamma$,
	\begin{align*}
		\int_{\gamma}\rho\circ\pi ds=\int_{\pi(\gamma)}\rho ds\geq 1.
	\end{align*}
	It follows that
	\begin{align*}
		\Modd_Q(\Gamma)\leq \int_{\tilde{\Omega}_1^f}\rho(\pi(z))^Qd\lambda(z)=\int_{\Omega}\rho(y)^QN(y,\pi,\tilde{\Omega}_1^f)d\nu(y).
	\end{align*}
\end{proof}

The following lemma will be important in our later proofs to deduce the standard V\"ais\"al\"a's inequality. It can be viewed as a fact that the projection mapping $\pi$ satisfies the generalized V\"ais\"al\"a's inequality.
\begin{lemma}\label{lemma:1.4}
	Let $\Gamma$ and $\Gamma'$ be families of curves in $\tilde{\Omega}^f$ and $\Omega$, respectively. Suppose for each $\gamma'\in \Gamma'$, there are curves $\gamma_1,\dots,\gamma_m\in \Gamma$ and subcurves $\gamma_1',\dots,\gamma_m'$ of $\gamma'$ such that for each $i=1,\dots,m$, $\gamma_i'=\pi(\gamma_i)$, and for a.e. $s\in [0,l(\gamma')]$, $\gamma_i(s)=\gamma_j(s)$ if and only if $i=j$. Then
	\begin{align*}
	m\Modd_Q(\Gamma')\leq \Modd_Q(\Gamma).
	\end{align*}\
	In particular, for every $\Gamma\subset \tilde{\Omega}^f$,
	\begin{equation*}
		\Modd_Q(\pi(\Gamma))\leq \Modd_Q(\Gamma).
	\end{equation*}
\end{lemma}
\begin{proof}
	It suffices to prove the first inequality, since the second inequality follows when $\Gamma'=\pi(\Gamma)$ and $m=1$.
	
	Suppose, then, that $\rho\colon \tilde{\Omega}^f\to \R$ is admissible for $\Gamma$. Define $\rho'\colon \Omega\to \R$ by
	\begin{align*}
		\rho'(y)=\frac{1}{m}\sup_{S\in \mathcal{S}_m(y)}\sum_{z\in S}\rho(z),
	\end{align*}
	where $\mathcal{S}_m(y)$ is the family of subsets of $\pi^{-1}(y)\cap \tilde{\Omega}^f$ of cardinality at most $m$.
	
	Now for each rectifiable curve $\gamma'\in \Gamma'$, we assume for ease of notation that $\gamma'$ is parameterized by arc-length, $\gamma'\colon [0,L]\to \Omega$, and each $\gamma_i'$ is parametrized so that $\gamma_i'=\pi\circ \gamma_i|_{[a_i,b_i]}$. Then at a.e. $t\in [0,L]$,
	\begin{align*}
		\rho'(\gamma'(t))\geq \frac{1}{m}\sum_{i:t\in [a_i,b_i]}\rho(\gamma_i(t)).
	\end{align*}
	We therefore have
	\begin{align*}
		\int_{\gamma'}\rho'ds&=\int_0^L\rho'(\gamma'(t))dt\geq \int_0^L\frac{1}{m}
\sum_{i:t\in [a_i,b_i]}\rho(\gamma_i(t))\\
&=\frac{1}{m}\sum_{i=1}^m\int_{a_i}^{b_i}\rho(\gamma_i(t))dt=\frac{1}{m}\sum_{i=1}^{m}\int_{\gamma_i}\rho ds\geq 1. 
	\end{align*}
	Thus $\rho'$ is admissible for $\Gamma'$ and by H\"older's inequality,
	\begin{align*}
		\rho'(y)^Q\leq \frac{1}{m^Q}\sup_{S\in \mathcal{S}_m(y)}m^{Q-1}\sum_{z\in S}\rho(z)^Q\leq \frac{1}{m}\sum(\rho,y,\pi,\tilde{\Omega}^f)^Q. 
	\end{align*}
	Therefore, by~\eqref{eq:change of variable general}, we have
	\begin{align*}
		\Modd_Q(\Gamma')\leq \int_{\Omega}\rho'(y)^Qd\nu(y)\leq \int_{\Omega}\frac{1}{m}\sum(\rho,y,\pi,\tilde{\Omega}^f)^Q d\nu(y)=\frac{1}{m}\int_{\tilde{\Omega}^f}\rho^Q d\lambda. 
	\end{align*}
\end{proof}

\begin{corollary}\label{coro:1.5}
	A curve family $\Gamma$ of curves in $\tilde{\Omega}^f$ is exceptional if and only if $\pi(\Gamma)$ is.
\end{corollary}
\begin{proof}
	If $\Gamma\subset \tilde{\Omega}_1^f\subset \tilde{\Omega}^f$, where $\tilde{\Omega}_1^f$ is relatively compact, then the corollary is trivial. The general case then follows from the separability of $\tilde{\Omega}^f$ and the countable subadditivity of modulus.
\end{proof}

We now discuss a few general facts about upper gradients in a locally compact
metric measure space $\Delta=(\Delta,d,\sigma)$.

First, we recall the following alternative characterization of upper gradients, which essentially follows from the change of variables formula for path integrals; see, e.g.,~\cite[Proposition 3.6]{w12proc}.

\begin{proposition}\label{prop:1.6}
	Let $\rho\in L^p_{loc}(\Delta)$ and $h\colon \Delta\to W$ be a continuous mapping. Then $\rho$ is a $p$-weak upper gradient of $h$ if and only if for $p$-almost every curve $\gamma\colon [a,b]\to \Delta$, the following two conditions are satisfied:
	\begin{itemize}
		\item $h$ is absolutely continuous along $\gamma$;
		\item If the parametrization of $\gamma$ itself is absolutely continuous as well, then the inequality
		\begin{align}\label{eq:upper graident pointwise characterization}
		\rho(\gamma(t))v_\gamma'(t)\geq v_{h(\gamma)}'(t)
		\end{align}
		holds a.e. on the parametrizing interval $[a,b]$.
	\end{itemize}
\end{proposition}

Note that in the previous proposition every rectifiable curve (and hence $p$-almost
every curve) has an absolutely continuous parametrization, namely, the arc-length
parametrization.

Let $\mathcal{C}_\varepsilon(W)$ be the collection of curves $\gamma\colon [a,b]\to W$ such that $d(\gamma(a),\gamma(b))\geq \varepsilon$. For every mapping $h\colon \Delta\to W$ from the metric space $W$, we use the notation $\mathcal{C}_\varepsilon h$ to denote the family of curves $\gamma$ in $\Delta$ such that $h(\gamma)\in \mathcal{C}_\varepsilon(W)$.

We need the following useful characterization of Sobolev mappings based on the limiting behavior of modulus of certain curve family, which generalizes~\cite[Theorem 3.10]{w12proc}.
\begin{theorem}\label{thm:characterization of Sobolev spaces via curves}
	A mapping $h\colon \tilde{\Omega}^f\to W$ belongs to the Sobolev space $N^{1,p}(\tilde{\Omega}^f,W)$, $1<p<\infty$, if and only if 
	\begin{align*}
	\liminf_{\varepsilon\to 0}\varepsilon^p\Modd_p(\pi(\mathcal{C}_\varepsilon h))<\infty.
	\end{align*}
	Moreover, if this is the case, then the liminf on the left-hand side is an actual limit, and 
	\begin{align*}
	\int_{\Omega} \sup(|\nabla h|,y,\pi,\tilde{\Omega}^f)d\nu(y)=\lim_{\varepsilon\to 0}\varepsilon^p\Modd_p(\pi(\mathcal{C}_\varepsilon h)).
	\end{align*}
\end{theorem}
\begin{proof}
	The proof is very similar to the proof of the standard case in~\cite[Theorem 3.10]{w12proc}, but a few modifications must be made. We quickly sketch the argument, and refer the readers to~\cite[Proof of Theorem 3.10]{w12proc} for the full details.
	
	As in~\cite{w12proc}, we say that a function $\rho$ is \textit{almost $p$-admissible} for a curve family $\Gamma$ in a metric measure space $(X,d,\mu)$ if $\rho$ is admissible for some subfamily $\tilde{\Gamma}\subset \Gamma$, with $\Modd_p(\Gamma\backslash \tilde{\Gamma})=0$. In this situation, we have
	\begin{align*}
		\Modd_p(\Gamma)=\Modd_p(\tilde{\Gamma})\leq \int_X\rho^p d\mu.
	\end{align*}
	Thus from the point of view of estimating the $p$-modulus, almost $p$-admissible functions work as well as admissible ones.
	
	Let $h\in N^{1,p}(\tilde{\Omega}^f,W)$. Then $|\nabla h|\in L^p(\tilde{\Omega}^f)$ and for each $\varepsilon>0$, the Borel function  $$\rho_\varepsilon(y)=\varepsilon^{-1}\sup(|\nabla h|,y,\pi,\tilde{\Omega}^f)$$ 
	is almost $p$-admissible for $\pi(\mathcal{C}_\varepsilon h)$. We therefore have
	\begin{align*}
	  \varepsilon^p\Modd_p(\pi(\mathcal{C}_\varepsilon h))\leq \int_{\Omega}\sup(|\nabla h|,y,\pi,\tilde{\Omega}^f)d\nu(y).
	\end{align*}
	
	Conversely, suppose $\liminf_{\varepsilon\to0}\varepsilon^p\Modd_p(\pi(\mathcal{C}_\varepsilon h))<\infty$. Let $\{\varepsilon_i\}$ be a sequence converging to zero that realizes the liminf. Then for each $\varepsilon_i$, we may select a Borel function $\rho_{\varepsilon_i}\colon \Omega\to [0,\infty)$, that is almost admissible for $\pi(\mathcal{C}_{\varepsilon_i} h)$, such that 
	\begin{align*}
		\Modd_p(\pi(\mathcal{C}_{\varepsilon_i} h))=\int_{\Omega}\rho_{\varepsilon_i}^p d\nu.
	\end{align*} 
	For each $i\in \mathbb{N}$, let $\varrho_i=\varepsilon_i\rho_{\varepsilon_i}\circ \pi$. Then the Borel function $\varrho_i$ is an upper gradient of $h$ along $p$-almost every curve $\gamma\in \mathcal{C}_\varepsilon h$, and also has the property that
	\begin{align*}
		\int_\Omega \sup(\varrho_i,y,\pi,\tilde{\Omega}^f)d\nu(y)\leq \varepsilon^p\Modd_p(\pi(\mathcal{C}_{\varepsilon_i}h)).
	\end{align*}
	Note that our starting assumption implies that the functions $\varrho_i$ is uniformly bounded in $L^p(\tilde{\Omega}^f)$. A standard limiting argument via Reflexivity, Mazur's lemma and Fuglede's lemma completes the proof (see e.g.~\cite[Proof of Theorem 3.10]{w12proc}).	
\end{proof}

Applying Theorem~\ref{thm:characterization of Sobolev spaces via curves} to each relatively compact open subset $\tilde{\Omega}_0^f\subset \tilde{\Omega}^f$, and invoking Lemma~\ref{lemma:1.2}, we immediately obtain the following corollary.
\begin{corollary}\label{coro:1.8}
	Let $p>1$ and let $\xi$ be a locally finite Borel regular measure on $\tilde{\Omega}^f$. Suppose for every relatively compact open subset $\tilde{\Omega}^f_0\subset \tilde{\Omega}^f$,
	\begin{align*}
	\liminf_{\varepsilon\to 0}\varepsilon^p\Modd_p(\pi(\mathcal{C}_\varepsilon h|_{\tilde{\Omega}_0^f}))<\xi(\tilde{\Omega}^f_0).
	\end{align*}
	Then $h\in N^{1,p}_{loc}(\tilde{\Omega}^f,W)$ and for $\lambda$-a.e. $z\in \tilde{\Omega}^f$,
	\begin{align*}
	|\nabla h|^p(z)\leq \frac{d\xi}{d\lambda}(z).
	\end{align*}
\end{corollary}

The previous result, when applied to the special case $f=\iota$, the inclusion mapping, reduces to the following result from~\cite[Theorem 3.10]{w12proc}.
\begin{corollary}\label{coro:1.9}
	Let $p>1$, and $h\colon \tilde{\Omega}\to W$, and let $\xi$ be a locally finite Borel regular measure on $\tilde{\Omega}$. Suppose for every relatively compact open subset $\tilde{\Omega}_0\subset \tilde{\Omega}$,
	\begin{align*}
	\liminf_{\varepsilon\to 0}\varepsilon^p\Modd_p(\mathcal{C}_\varepsilon h)<\xi(\tilde{\Omega}).
	\end{align*}
	Then $h\in N^{1,p}_{loc}(\title{\Omega},W)$ and for $\mu$-a.e. $x\in \Omega$,
	\begin{align*}
	|\nabla h|^p(x)\leq \frac{d\xi}{d\mu}(x).
	\end{align*}
\end{corollary}

As an immediate application of the above characterization of Sobolev mappings, we obtain the following characterization of analytic quasiregularity via the pullback factorization.
\begin{proposition}\label{prop:characterization of QR via pullback factorization}
	Let $f\colon \tilde{\Omega}\to \Omega$ be a branched covering. Then $f$ is analytically $K$-quasiregular with exponent $Q$ if and only if $g\colon \tilde{\Omega}\to \Omega^f$ is analytically $K$-quasiconformal with exponent $Q$. 
\end{proposition}
\begin{proof}
	First, note that by Theorem~\ref{thm:characterization of Sobolev spaces via curves}, $f$ is analytically $K$-quasiregular with exponent $Q$ if and only if for each relatively compact open subset $\tilde{\Omega}_0\subset \tilde{\Omega}$,
	\begin{align*}
	\liminf_{\varepsilon\to 0}\varepsilon^Q\Modd_Q(\mathcal{C}_\varepsilon f|_{\tilde{\Omega}_0})\leq Kf^*\nu(\tilde{\Omega}_0).	
	\end{align*}
    Similarly, $g$ is analytically $K$-quasiconformal with exponent $Q$ if and only if each relatively compact open subset $\tilde{\Omega}_0\subset \tilde{\Omega}$,
	\begin{align*}
	\liminf_{\varepsilon\to 0}\varepsilon^Q\Modd_Q(\mathcal{C}_\varepsilon g|_{\tilde{\Omega}_0})\leq Kg^*\lambda(\tilde{\Omega}_0).	
	\end{align*}
	
	Secondly, according to our definition of $\lambda$, $f^*\nu=g^*\lambda$ and so it suffices to show that $\Modd_Q(\mathcal{C}_\varepsilon f|_{\tilde{\Omega}_0})=\Modd_Q(\mathcal{C}_\varepsilon g|_{\tilde{\Omega}_0})$. Since $\pi$ is 1-Lipschitz, $\mathcal{C}_\varepsilon f|_{\tilde{\Omega}_0}\subset \mathcal{C}_\varepsilon g|_{\tilde{\Omega}_0}$, and so we have
	\begin{align*}
		\Modd_Q(\mathcal{C}_\varepsilon f|_{\tilde{\Omega}_0})\leq \Modd_Q(\mathcal{C}_\varepsilon g|_{\tilde{\Omega}_0}).
	\end{align*} 
	For the other direction, we need to observe that if $\rho$ is admissible for $\mathcal{C}_\varepsilon f|_{\tilde{\Omega}_0}$, then it is also admissible for $\mathcal{C}_\varepsilon g|_{\tilde{\Omega}_0}$, which follows directly from the definition of line integral and from the fact that each the arc-length parametrization $\gamma^s$ of each $\gamma\in \mathcal{C}_\varepsilon g|_{\tilde{\Omega}_0}$ is an element of $\mathcal{C}_\varepsilon f|_{\tilde{\Omega}_0}$.
	
\end{proof}


Parallel to Proposition~\ref{prop:characterization of QR via pullback factorization}, we have the following comparison result on metric quasiregularity via the pullback factorization.

\begin{proposition}\label{prop:characterization of QR via pullback factorization II}
	Let $f\colon \tilde{\Omega}\to \Omega$ be a branched covering. Suppose $\Omega$ has $c$-bounded turning. Then for each $x_0\in \tilde{\Omega}$, we have
	\begin{align*}
		\frac{1}{c}h_f(x_0)\leq h_g(x_0)\leq ch_f(x_0)
	\end{align*}
	and
	\begin{align*}
		\frac{1}{c}h_f^*(x_0)\leq h_g^*(x_0)\leq ch_f^*(x_0).
	\end{align*}
	The analogous results hold with $h_f$ being replaced by $H_f$, and $h_f^*$ being replaced by $H_f^*$, respectively.
\end{proposition}
\begin{proof}
	This follows from the fact that when $\Omega$ has $c$-bounded turning, $\pi\colon \tilde{\Omega}^f\to \Omega$ satisfies the following property: for each $z_0\in \tilde{\Omega}^f$ and $z\in U(z_0,\pi,r)$ (with $r$ sufficiently small),
	\begin{align*}
		d(\pi(z),\pi(z_0))\leq f^*d(z,z_0)\leq cd(\pi(z),\pi(z_0)),
	\end{align*}
	 which is a direct consequence of~Lemma~\ref{lemma:pullback property 2}. Indeed, for each $r>0$ small enough and each $x\in \partial B(x_0,r)$, we have
	 \begin{align*}
	 	d(f(x),f(x_0))\leq f^*d(g(x),g(x_0))\leq cd(f(x),f(x_0)).
	 \end{align*}
	 In particular, this implies that $L_g(x_0,r)\leq cL_f(x_0,r)$ and $l_g(x_0,r)\geq l_f(x_0,r)$ and so
	 \begin{align*}
	 	h_g(x_0)=\liminf_{r\to 0}\frac{L_g(x_0,r)}{l_g(x_0,r)}\leq c\liminf_{r\to 0}\frac{L_f(x_0,r)}{l_f(x_0,r)}=ch_f(x_0).
	 \end{align*}
	 Similarly, since $L_f(x_0,r)\leq L_g(x_0,r)$ and $l_g(x_0,r)\leq cl_f(x_0,r)$, we also have 
	 \begin{align*}
	 h_g(x_0)=\liminf_{r\to 0}\frac{L_g(x_0,r)}{l_g(x_0,r)}\geq \frac{1}{c}\liminf_{r\to 0}\frac{L_f(x_0,r)}{l_f(x_0,r)}=\frac{1}{c}h_f(x_0).
	 \end{align*}
	 The other assertions follow similarly.
\end{proof}

\subsubsection{Proof of the main result}
We are now ready to prove Theorem~\ref{thm:equivalence of G and A} following closely the approach in~\cite[Proof of Theorem 1.1]{w12proc}.

\begin{proof}[Proof of Theorem~\ref{thm:equivalence of G and A}]
Note first that by Proposition~\ref{prop:characterization of QR via pullback factorization}, $f$ is analytically $K_O$-quasiregular with exponent $Q$ if and only if $g$ is analytically $K_O$-quasiconformal with exponent $Q$. 

Next, we show that if $g$ is geometrically $K_O$-quasiconformal with exponent $Q$, then $f$ is also geometrically $K_O$-quasiregular with exponent $Q$. Let $\tilde{\Omega}_0\subset \tilde{\Omega}$ and let $\Gamma$ be a curve family in $\title{\Omega}_0$. If $\tilde{\rho}\colon f(\tilde{\Omega}_0)\to \R$ is admissible for $f(\Gamma)$, then $\tilde{\rho}$ is also admissible for the curve family $\pi(\Gamma')$ with $\Gamma'=g(\Gamma)$. As a consequence of Lemma~\ref{lemma:1.3}, we have
\begin{align*}
	\Modd_Q(\Gamma)&\leq K_O\Modd_Q(g(\Gamma))\leq K_O\int_{\Omega}\rho(y)^QN(y,\pi,\tilde{\Omega}_0)d\nu(y)\\
	 &=K_O\int_{\Omega}\rho(y)^QN(y,f,\tilde{\Omega}_0)d\nu(y).
\end{align*}

Finally, we show that if $f$ is geometrically $K_O$-quasiregular with exponent $Q$, then it is analytically $K_O$-quasiregular with exponent $Q$ as well. For every $\varepsilon>0$ and every open subset $\tilde{\Omega}_0\subset \tilde{\Omega}$, the function $\rho=\varepsilon^{-1}\chi_{f(\tilde{\Omega}_0)}$ is admissible for $f(\mathcal{C}_\varepsilon f|_{\tilde{\Omega}_0})$, and so we have the inequalities
\begin{align*}
	\varepsilon^Q\Modd_Q(\mathcal{C}_\varepsilon f|_{\tilde{\Omega}_0})\leq K_O\varepsilon^Q\int_{\Omega}\varepsilon^{-Q}N(y,f,\tilde{\Omega}_0)d\nu(y)=K_Of^*\nu(\tilde{\Omega}_0).
\end{align*}
The proof now follows by applying Corollary~\ref{coro:1.9} with $\xi=K_Of^*\nu$.

We now turn to the second set of equivalences. Note that if $g$ satisfies the $K_I$-inequality with exponent $Q$, then it follows from Lemma~\ref{lemma:1.4} that $f$ satisfies the generalized V\"ais\"al\"a's inequality with exponent $Q$ and hence also the Poletsky's inequality with exponent $Q$. On the other hand, if $f$ satisfies the Poletsky's inequality with exponent $Q$, then for every $\varepsilon>0$ and every open subset $\title{\Omega}^f_0\subset \title{\Omega}^f$, we have
\begin{align*}
	\varepsilon^Q\Modd_Q(\pi(\mathcal{C}_\varepsilon g^{-1}|_{\title{\Omega}^f_0}))&=\varepsilon^Q\Modd_Q(f(\mathcal{C}_\varepsilon \iota|_{g^{-1}(\title{\Omega}^f_0)}))\leq K_I\varepsilon^Q\Modd_Q(\mathcal{C}_\varepsilon \iota|_{g^{-1}(\title{\Omega}^f_0)})\\
	&\leq K_I\varepsilon^{Q}\varepsilon^{-Q}\mu(g^{-1}(\title{\Omega}^f_0))=K_I(g^{-1})^*\mu(\title{\Omega}^f_0),
\end{align*}
where in the last inequality we have used~\cite[Lemma 5.3.1]{hkst15}. It follows from Corollary~\ref{coro:1.8} that $g^{-1}$ is analytically $K_I$-quasiconformal with exponent $Q$. 	
\end{proof}

\subsection{Ahlfors $Q$-regularity and the metric definitions}\label{subsec:Ahlfors regularity and the metric definitions}
When beginning from metric assumptions, the results and methods of~\cite{hk98,bkr07,w14} are particularly useful. For this setting, to get some analytic information, we need to relate the measures $\mu$ and $\nu$ to the diameters of balls, and so we will need to assume that the spaces under consideration are Ahlfors $Q$-regular.
 
For the quasiconformal case, it was shown in~\cite{hk95,bkr07,w14}, that even a bound on the weak linear dilatation $h_f$ was sufficient to prove the analytic characterization of quasiregular mappings (and hence also the geometric one). It has also been observed in~\cite{w14} that a bound on $h_f^*$ also suffices to prove the $K_O$-inequality, and so by symmetry, a bound on either quantity suffices to prove the analytic (and also, by the previous section, geometric) characterization of quasiregularity. It follows fairly quickly from the results
above that the same is true for the general branched case, though some care is needed to deal with the local multiplicity of $f$.

Our main result of this section will be the following.

\begin{theorem}\label{thm:metric implies all the other}
	If $\tilde{\Omega}$ and $\Omega$ are locally Ahlfors $Q$-regular, and $\Omega$ has $c$-bounded turning, then either of the following two conditions
	\begin{itemize}
		\item[i).] $h_f(x)\leq h$ for all $x\in \tilde{\Omega}$;
		\item[ii).] $h_f^*(x)\leq h$ for all $x\in \tilde{\Omega}$,
	\end{itemize}
	implies that $f$ is analytically $K_O$-quasiregular with exponent $Q$ and inverse analytically $K_I$-quasiregular with exponent $Q$, with both constants $K_O$ and $K_I$ depending only on the constant of Ahlfors $Q$-regularity, and on $c$ and $h$.
\end{theorem}

Here, as was done in~\cite{bkr07}, we have assumed a stronger condition, that the dilatation is everywhere bounded, rather than simply everywhere finite and essentially bounded. We may drop this assumption (again, as in \cite{bkr07}), in the presence of a Loewner condition. In fact, even without the Loewner condition, it seems not be unnecessary to bound the dilatation everywhere, rather than essentially. However, this issue is rather technical, and so we eschew such considerations in this paper, as they would lead us too far astray. 

Note also that we did not impose the usual LLC condition on either domains, as we have done in \cite{w14}, in the above theorem, though this condition is rather mild, and simplifies the exposition considerably.

\subsubsection{Auxiliary results}
The main result of this section is the following criterion for analytic quasiregularity, which generalizes~\cite[Theorem 1.1]{bkr07}.
\begin{theorem}\label{thm:for analytic regularity}
	Let $1\leq p\leq Q$ and let $h\colon \Delta\to W$ be a homeomorphism between (pointwise) doubling metric measure spaces $(\Delta,d_\Delta,\sigma)$ and $(W,d_W,\tau)$. Suppose there is a subset $E\subset \Delta$ such that for $p$-almost every curve $\gamma$ in $\Delta$, $\mathcal{H}^1\big(h(\gamma\cap E)\big)=0$, and a function $\eta\colon \Delta\backslash E\to \R$ such that the following condition is satisfied:
	
	For every $v\in \Delta\backslash E$ and $\varepsilon>0$, there are neighborhoods $D_{v,\varepsilon}$, $D_{v,\varepsilon}'$ and $D_{v,\varepsilon}''$ of $v$ such that $D_{v,\varepsilon}'\cup D_{v,\varepsilon}''\subset D_{v,\varepsilon}\subset B(v,\varepsilon)$, satisfying the inequalities
	\begin{align*}
		\Big(\frac{\diam\big(h(D_{v,\varepsilon})\big)}{\diam\big(D_{v,\varepsilon}\big)}\Big)^Q\sigma(D_{v,\varepsilon}')\leq \eta(v)\tau\big(h(D_{v,\varepsilon}'')\big),
	\end{align*}
	and
	\begin{align*}
		\sigma\big(B(v,10\diam(D_{v,\varepsilon}))\big)\leq C\sigma(D_{v,\varepsilon}'),
	\end{align*}
	and satisfying the following property: For every subset $A\subset \Delta$, and every set of indices $I\subset\Delta\times (0,\infty)$ such that $A\subset \bigcup_{\alpha\in I}D_\alpha$, there is a countable subset $\{\alpha_i\}\subset I$ such that $A\subset \bigcup_{i=1}^\infty D_{\alpha_i}$ and whenever $i\neq j$, 
	$$D_{\alpha_i}'\cap D_{\alpha_j}'=\emptyset=D_{\alpha_i}''\cap D_{\alpha_j}''.$$
	Then if $p<Q$ and $\eta$ is essentially bounded, or $p=Q$ and $\eta$ is bounded, then $h\in N^{1,p}(\Delta,W)$, and the minimal $p$-weak upper gradient $|\nabla h|$ satisfies 
	\begin{equation}\label{eq:analytic inequality in general mms}
		\int_{\Delta}|\nabla h|^Q d\sigma\leq C'\tau(\Delta),
	\end{equation}
	where $C'$ depends only on $C$ and $\esssup \eta$.
\end{theorem}

We require the following generalization of a well-known covering lemma.
\begin{lemma}\label{lemma:harmonic analysis lemma}
	Let $p\geq 1$, let $\xi$ be a measure on $\tilde{\Omega}$, and let $\{A_i\}$ and $\{B_i\}$ be sequences of $\xi$-measurable sets and balls, respectively, such that for each $i\in \mathbb{N}$, $A_i\subset B_i\subset \tilde{\Omega}$, and $\xi(5B_i)\leq C\xi(A_i)$, and let $\{c_i\}$ be a sequence of nonnegative integers. Then
	\begin{align*}
	\int_{\tilde{\Omega}}\Big(\sum_{i=1}^\infty c_i\chi_{B_i}\Big)^pd\xi\leq C_p\int_{\tilde{\Omega}}\Big(\sum_{i=1}^\infty c_i\chi_{A_i}\Big)^pd\xi,
	\end{align*}
	where $C_p$ is a constant depending only on $C$ and $p$.
\end{lemma}
\begin{proof}
	The proof is the same as the usual argument found in, for example~\cite{b88}, only instead of the usual (uncentered) Hardy-Littlewood maximal operator, one uses the operator
	\begin{align*}
		M(\rho)(v)=\sup_{i\in \mathbb{N}:B_i\ni v}\dashint_{B_i}\rho(x) d\xi(x).
	\end{align*}
	Then the mapping $\rho\mapsto M(\rho)$ is a bounded operator from $L^1(\xi)$ to weak-$L^1(\xi)$, i.e.,
	for every $\rho\in L^1(\xi)$ and $t>0$, we have
	\begin{align*}
		t\xi\big(\{v\in \tilde{\Omega}:M(\rho)(v)>t\}\big)\leq c\|\rho\|_{L^1(\xi)}.
	\end{align*}
	The proof of this is identical to the analogous proof in~\cite{h01,hkst15}, except their covering lemma is replaced by Lemma~\ref{lemma:covering lemma}, and the doubling condition is replaced by the inequality $\xi(5B_i)\leq C\xi(A_i)$ (since $A_i\subset B_i$). From the weak estimate above, we obtain boundedness of the operator on $L^q(\xi)$ for each $q>1$:
	\begin{align*}
		\|M(\rho)\|_{L^p(\xi)}\leq c_q\|\rho\|_{L^q(\xi)}.
	\end{align*}
	The proof of the above inequality is entirely measure-theoretic, and thus is identical to that found in \cite{h01,hkst15}. The proof of the theorem now proceeds as in the case for doubling spaces (see e.g. \cite{b88}). 
	
	The theorem is trivial for $p=1$. For $p>1$, let $q=p/(p-1)$. Then for all $\rho\in L^q(\xi)$,
	\begin{align*}
		\int_{\tilde{\Omega}}\Big(\sum_{i=1}^\infty c_i\chi_{B_i}\Big)\rho d\xi&=\sum_{i=1}^{\infty}c_i\xi(B_i)\dashint_{B_i}\rho d\xi\leq C\sum_{i=1}^\infty c_i\xi(A_i)\dashint_{B_i}\rho d\xi\\
		&\leq C\sum_{i=1}^{\infty}c_i\xi(A_i)\dashint_{A_i}M(\rho)d\xi=C\sum_{i=1}^{\infty}c_i\int_{A_i}M(\rho)d\xi\\
		&=C\int_{\tilde{\Omega}}\Big(\sum_{i=1}^{\infty}c_i\chi_{B_i}\Big)M(\rho)d\xi\leq C\Big\|\sum_{i=1}^{\infty}c_i\chi_{B_i}\Big\|_{L^p(\xi)}\|M(\rho)\|_{L^q(\xi)}\\
		&\leq Cc_q\Big\|\sum_{i=1}^{\infty}c_i\chi_{B_i}\Big\|_{L^p(\xi)}\|\rho\|_{L^q(\xi)}.
	\end{align*}
	Since this holds for all $\rho\in L^q(\xi)$, we conclude that our claim holds with $C_p=(Cc_q)^p$.
\end{proof}

We are now ready to prove Theorem~\ref{thm:for analytic regularity}. The proof of~\cite[Theorem 1.1]{bkr07} generalizes almost word for word, so we do not repeat the argument in full details, but instead outline it in the following proof and comment on the necessary changes.
\begin{proof}[Proof of Theorem~\ref{thm:for analytic regularity}]

	Let $p'=p$ if $p>1$, otherwise choose $1 < p' < Q$. We first fix $\varepsilon> 0$ and construct an approximate upper gradient
	as follows: Decompose $\Delta\backslash E=\bigcup_{k=0}^\infty A_k$, where
	\begin{align*}
		A_k=\eta^{-1}\big((k\esssup \eta, (k+1)\esssup \eta] \big),
	\end{align*}
	so that $\sigma\big(\bigcup_{k=1}^\infty A_k\big)=0$. Then choose
	(omitting subscripts $\varepsilon$ here on out for ease of notation) neighborhoods $U_k$ of $A_k$ for each $k\geq 1$ such that $\tau(U_k)\leq \delta_k$, for some small $\delta_k$ to be chosen momentarily. By
	the covering assumption, we obtain a sequence of subsets $D_i=D_{v_i,\varepsilon_i}$ such that each $i$, $\varepsilon_i\leq \varepsilon$, and such that if $I_k=\{i\in \mathbb{N}:v_i\in A_k\}$, then $\bigcup_{i\in I_k}D_i\subset U_k$. Then if 
	\begin{align*}
		T_k=\sum_{i\in I_k}\frac{\diam(h(D_i))}{\diam(D_i)}\chi_{B(v_i,2\diam(D_i))},
	\end{align*}
	then applying Lemma~\ref{lemma:harmonic analysis lemma}, along with the hypotheses on the sets $D_i$, $D_i'$ and $D_i''$, we obtain 
	\begin{align*}
		\int_{\Delta}&T_0^Qd\sigma\leq C\int_{\Delta}\sum_{i\in I_0}\Big(\frac{\diam(h(D_i))}{\diam(D_i)}\Big)^Q\chi_{D_i'}d\sigma\\
		&\leq C\sum_{i\in I_0}\Big(\frac{\diam(h(D_i))}{\diam(D_i)}\Big)^Q\sigma(D_i')\leq C\sum_{i\in I_0}\tau(h(D_i''))\leq C\tau(\Delta).\numberthis\label{eq:uniform upper bound}
	\end{align*}
	For each $k\geq 1$ and $p<Q$, we have the estimates
	\begin{align*}
		\int_{\Delta}T_k^{p'}d\sigma&\leq C\int_{\Delta}\sum_{i\in I_k}\Big(\frac{\diam(h(D_i))}{\diam(D_i)}\Big)^{p'}\chi_{D_i'}d\sigma\leq C\sum_{i\in I_k}\Big(\frac{\diam(h(D_i))}{\diam(D_i)}\Big)^{p'}\sigma(D_i')\\
		&\leq C(k+1)^{p'/Q}\sum_{i\in I_k}\tau(h(D_i''))^{p'/Q}\sigma(D_i)^{(Q-p')/Q}\\
		&\leq C(k+1)^{p'/Q}\Big(\sum_{i\in I_k}\tau(h(D_i''))\Big)\Big(\sum_{i\in I_k}\sigma(D_i')\Big)^{(Q-p')/p'}\\
		&\leq C(k+1)^{p'/Q}\tau(\Omega)\sigma(U_k)^{(Q-p')/Q}\leq C(k+1)^{p'/Q}\tau(\Delta)\delta_k^{(Q-p')/Q}.
	\end{align*}
	In particular, this implies that $T_k=0$ when $k\geq 1$, since may choose $\delta_k$ so that $\|T_k\|_p\leq \varepsilon/2^k$. Note also that the above estimate holds for $p=Q$ as well.
	
	Now, let $\rho_{\varepsilon,0}=T_0$, $\rho_{\varepsilon,+}=\sum_{k=1}^\infty T_k$, and $\rho_{\varepsilon}=\rho_{\varepsilon,0}+\rho_{\varepsilon,+}$. Observe that if a rectifiable curve $\gamma$ satisfies $\mathcal{H}^{Q-p}\big(h(\gamma\cap E)\big)=0$ and $\diam(\gamma)\geq 4\varepsilon$, $\gamma$ intersects one of the sets $D_i$, then $\gamma$ joins $D_i\subset B(v_i,\diam(D_i)))$ with $\Delta\backslash B(v_i,2\diam(D_i))$, so that $l\big(\gamma\cap B(v_i,2\diam(D_i))\big)\geq \diam(D_i)$. Therefore, 
	\begin{equation}\label{eq:admissible lower bound}
		\int_{\gamma}\rho_\varepsilon ds\geq \sum_{i=1}^\infty\Big(\frac{\diam(h(D_i))}{\diam(D_i)}\Big)l\big(\gamma\cap D_i\big)\geq \diam(h(\gamma)),  
	\end{equation}
	where the last inequality follows from the fact that $\mathcal{H}^1(h(\gamma\cap E))=0$.
	
	Since the functions $\rho_{\varepsilon,0}$ uniformly satisfy inequality~\eqref{eq:uniform upper bound}, we may apply reflexivity and Mazur's Lemma to obtain a sequence of convex combinations $\rho_n=\rho_{n,0}+\rho_{n,+}$
	of the functions $\rho_\varepsilon$ such that $\rho_{n,0}$ converges strongly in $L^Q$ to a function $\rho$ satisfying~\eqref{eq:uniform upper bound} as well, and such that inequality~\eqref{eq:admissible lower bound} is satisfied for $p$-almost every curve of diameter at least $1/n$. A fortiori, $\rho_{n,0}$ also converges to $\rho$ in $L^{p'}$, and hence so does $\rho_n$, since $\rho_{n,+}$ converges to 0. A well-known theorem of Fuglede (see e.g.~\cite[Fuglede's Lemma, P.131]{hkst15}) yields a further subsequence $n_j$ such that $\int_{\gamma}\rho_{n,j}ds$ converges to $\int_{\gamma}\rho ds$ on $p'$-almost every (and hence $p$-almost every) curve $\gamma$. It follows that $\rho$ is a $p$-weak upper gradient satisfying inequality~\eqref{eq:uniform upper bound}, and so the proof for the case $p<Q$ is complete. 
	
	It only remains to show that in the case that $p = Q$ and $\eta$ is bounded, the final constant $C$ depends only on the essential bound on $\eta$, rather than the actual bound. To see this, we apply the
	theorem once to see that $h\in N^{1,Q}(\Delta,W)$, and so $h$ is absolutely continuous on $Q$-almost every curve. Let $E=\{v\in \Delta:\eta(v)>\esssup \eta\}$. Since $\mu(E)=0$, $Q$-almost every curve intersects $E$ with Hausdorff 1-measure 0. Absolute continuity on curves then implies that $E$ satisfies the requirements of the theorem, and so we apply the theorem again, this time excluding $E$, to obtain a result depending only on $\esssup \eta$.
\end{proof}

\begin{remark}\label{rmk:on removable sets}
	The results in \cite[Theorem 1.1]{bkr07} allow for removable sets in the case $p<Q$. We have avoided this discussion for purposes of simplicity, but the same removability arguments can be made here to obtain similar results. Our main use for the case $p<Q$ in Theorem~\ref{thm:for analytic regularity} will be for the analysis of Loewner spaces in the forthcoming
	sections, where the self-improving nature of the Poincar\'e inequality will allow us to remove the requirement that, for the case $p=Q$, $h_f(x)$ must be bounded. We could also remove the aforementioned need for boundedness without requiring a Poincar\'e inequality, but this result is more technical than what we are interested in here.
	
\end{remark}

\subsubsection{Proof of the main result}

With the aid of Theorem~\ref{thm:for analytic regularity}, we are now ready to prove our main result of this section. The idea used in the implication, from Theorem~\ref{thm:metric implies all the other} i) to the quantitative analytic characterization, is quite similar to~\cite[Proof of Theorem~1.1]{bkr07}. But some care is needed in order to show that i) implies the quantitative inverse analytic characterization, due to the multiplicity issue. We will overcome this technical difficulty by applying Theorem~\ref{thm:metric implies all the other} twice in a suitable way. We would like to point out that there is a great asymmetry between the proofs with assuming i) or ii). This symmetry was already observed in~\cite{w14} in the homeomorphism case. It is also suggestive to compare our following proofs with~\cite[Proof of Theorem 1.1]{bkr07} and~\cite[Proof of Theorem 1.6]{w14} to understand the differences with the homeomorphic case.

\begin{proof}[Proof of Theorem~\ref{thm:metric implies all the other}]
Case i). $h_f(x)\leq h$ for all $x\in \tilde{\Omega}$.

We first show that $f$ is analytically $K$-quasiregular with exponent $Q$, quantitatively. This case is slightly easier as there is no index issue and we invoke Theorem~\ref{thm:for analytic regularity}, with $\Delta=\tilde{\Omega}_0\subset \tilde{\Omega}$, where $\tilde{\Omega}_0$ is an arbitrary relatively compact open subset of $\tilde{\Omega}$, $W=\tilde{\Omega}^f$, $h=g$, and $p=Q$. We may additionally assume that $\Omega^f$ is 1-bounded turning, otherwise just replace the original metric $d^*=f^*d_Y$ on $\tilde{\Omega}^f$ by $\hat{d}^*(x,z)=\inf_{\alpha}d^*(x,z)$, where the infimum is taken over all continua $\alpha$ that connect $x$ and $z$ in $\tilde{\Omega}^f$.

For each $x\in \tilde{\Omega}$, $\varepsilon>0$ and $A\subset \Delta$, consider the family of balls $B(x,r_\varepsilon)$, where $r_\varepsilon<\varepsilon$ is chosen so that $\frac{L_g(x,r_\varepsilon)}{l_g(x,r_\varepsilon)}<2h_f(x)$. It follows from the covering lemmas 2.2 and 2.3 in~\cite{bkr07} that there is a sequence of balls $B_i=B(x_i,r_i)$ from the family so that

\begin{itemize}
	\item For each $i\neq j$, $B_i/3\cap B_j/3=\emptyset$ and 
	$$B\Big(g(x_i),\frac{L_g(x_i,r_i)}{10h_f(x_i)^2}\Big)\cap B\Big(g(x_j),\frac{L_g(x_j,r_j)}{10h_f(x_i)^2}\Big)=\emptyset.$$
	
	\item $A\subset \bigcup_{i}B_i$.
\end{itemize}
%
We now set $D_{x,\varepsilon}=B(x,r_\varepsilon)$, $D_{x,\varepsilon}'=B(x,r_\varepsilon/3)$, and $D_{x,\varepsilon}''=g^{-1}\big(B(g(x),L_g(x,r_\varepsilon)/10h_f(x)^2)\big)$. The covering properties of the family of balls $B(x,r_\varepsilon)$ imply that these neighborhoods satisfy all the covering assumptions of Theorem~\ref{thm:for analytic regularity}. 

On the other hand,  since $\mu$ is locally Ahlfors $Q$-regular and $\lambda$ satisfies~\eqref{eq:estimate for lambda}, we have
\begin{align*}
	\Big(\frac{\diam(g(D_{x,\varepsilon}))}{\diam(D_{x,\varepsilon})}\Big)^Q&\mu(D_{x,\varepsilon}')\leq c_\mu\Big(\frac{L_g(x,r_\varepsilon)}{r_\varepsilon}\Big)^Qr_\varepsilon^Q\\
	&\leq ch_f(x)^{2Q}\Big(\frac{L_g(x,r_\varepsilon)}{10h_f^2(x)}\Big)^Q\leq ch_f(x)^{2Q}\lambda\big(g(D_{x,\varepsilon}'')\big)
\end{align*}
and
\begin{align*}
	\mu\big(B(x,10\diam(D_{x,\varepsilon}))\big)\leq c_\mu\big(10\diam D_{x,\varepsilon}\big)^Q\leq cr_\varepsilon^Q\leq c\mu(D_{x,\varepsilon}'),
\end{align*}
where the constant $c$ depends only on the bounded turning constant of $\Omega$ and the Ahlfors regularity constants of $\mu$ and $\lambda$. Thus all the assumptions of Theorem~\ref{thm:for analytic regularity} are satisfied upon choosing $\eta(x)=ch_f(x)^{2Q}$. 

By Theorem~\ref{thm:for analytic regularity}, $g\in N^{1,Q}(\tilde{\Omega}_0,\tilde{\Omega}^f)$ and there exists some constant $K_O$ (depending only on $c$ and $\esssup h_f$) such that
\begin{align*}
	\int_{\tilde{\Omega}_0}|\nabla g|^Qd\mu(x)\leq K_O\lambda(g(\tilde{\Omega}_0)).
\end{align*}
Since $\tilde{\Omega}_0$ is arbitrary, it follows by definition,
\begin{align*}
	|\nabla g|^Q(x)\leq K_OJ_g(x)\quad \mu\text{-a.e. in }\tilde{\Omega}.
\end{align*}

We next show that $f$ is inverse analytically $K$-quasiregular with exponent $Q$, quantitatively. Observe that by the proof of Theorem~\ref{thm:bounded geometry I}, it suffices to show that $g^{-1}$ is analytically $K_O$-quasiregular with exponent $Q$, quantitatively. Thus we apply Theorem~\ref{thm:for analytic regularity} with $\Delta=\tilde{\Omega}_0^f$, where $\tilde{\Omega}_0^f\subset \tilde{\Omega}^f$ is relatively compact and open, and $W=\tilde{\Omega}$, $h=g^{-1}$, and $p=Q$. 

For each $y\in \tilde{\Omega}_0^f$ and $\varepsilon>0$, let $r_\varepsilon<\varepsilon$ be chosen so that $\frac{L_g(x,r_\varepsilon)}{l_g(x,r_\varepsilon)}<2h_f(x)$. We may argue as in the previous case to choose neighborhoods $D(y,\varepsilon)=g(B(x,s_\varepsilon))$, $D_{y,\varepsilon}'=g(B(x,s_\varepsilon/3))$, and $D_{y,\varepsilon}''=B(g(x),L_g(x,s_\varepsilon)/10h_f(x)^2)$, where $s_\varepsilon=L_g(x,r_\varepsilon)$ such that the covering requirement in Theorem~\ref{thm:for analytic regularity} are satisfied.

On the other hand,  since $\lambda$ satisfies~\eqref{eq:estimate for lambda} and $\mu$ is locally Ahlfors $Q$-regular, we have
\begin{align*}
\Big(\frac{\diam(g^{-1}(D_{y,\varepsilon}))}{\diam(D_{y,\varepsilon})}\Big)^Q&\lambda(D_{y,\varepsilon}')\leq ci_{\ess}(x,f)\Big(\frac{s_\varepsilon}{L_g(x,s_\varepsilon)}\Big)^QL_g(x,s_\varepsilon)^Q\\
&\leq ci_{\ess}(x,f)h_f(x)^{2Q}\mu\big(g^{-1}(D_{y,\varepsilon}'')\big)
\end{align*}
and
\begin{align*}
\lambda\Big(B(y,10\diam(D_{y,\varepsilon}))\Big)\leq ci_{\ess}(x,f)\lambda(D_{y,\varepsilon}'),
\end{align*}
where the constant $c$ depends only on the bounded turning constant of $\Omega$ and the Ahlfors regularity constants of $\mu$ and $\lambda$. Set $\eta(y)=ci_{\ess}(g^{-1}(y),f)h_f(g^{-1}(y))^{2Q}$.

Applying Theorem~\ref{thm:for analytic regularity} as before, we obtain that $g^{-1}\in N(\tilde{\Omega}_0^f,\tilde{\Omega}_0)$ and there exists some constant $K_I'$ (this time depending also on $N(f,\tilde{\Omega}_0)$) such that
\begin{align*}
|\nabla g^{-1}|^Q(y)\leq K_I'J_{g^{-1}}(y)\quad \lambda\text{-a.e. in }\tilde{\Omega}_0^f.
\end{align*}

We thus obtain, by Theorem~\ref{thm:equivalence of G and A}, the $K_O$-inequality with exponent $Q$ for $g^{-1}$. Since $\lambda(g(\mathcal{B}_f^e))=0$, almost every curve in 
$\tilde{\Omega}_0^f$ intersects $g(\mathcal{B}_f^e)$ with $\mathcal{H}^1$-measure 0. But now the $K_O$-inequality implies that for almost every curve $\gamma$ in $\tilde{\Omega}_0^f$, $\mathcal{H}^1\big(g^{-1}(\gamma)\cap \mathcal{B}_f^e\big)=0$, so that we may apply Theorem~\ref{thm:for analytic regularity}, this time with $E=g(\mathcal{B}_f^e)$, to obtain the theorem with $K_O$ independent of the multiplicity.


Case ii). $h_f^*(x)\leq h$ for all $x\in \tilde{\Omega}$.

%
We only show that $f$ satisfies the $K_I$-inequality with exponent $Q$ and the fact that $f$ also satisfies the  $K_O$-inequality with exponent $Q$ can be argued similar as in the first part of case i) (and is indeed simpler since there is no index issue). Again, by Theorem~\ref{thm:equivalence of G and A}, it suffices to show that $g^{-1}\colon \tilde{\Omega}^f\to \tilde{\Omega}$ is analytically $K_I$-quasiconformal with exponent $Q$. Thus we apply the exact same argument in reverse to invoke Theorem~\ref{thm:for analytic regularity}, with $\Delta=\tilde{\Omega}_0^f$ and $W=\tilde{\Omega}_0:=g^{-1}(\tilde{\Omega}_0^f)$ for an arbitrary relatively compact open subset $\tilde{\Omega}_0^f$ of $\tilde{\Omega}^f$. The only difference is that this time around, we need to take care of the essential branch set $\mathcal{B}_f^e$ of $f$. 

For $y=g(x)\in \tilde{\Omega}_0^f$ and $\varepsilon>0$, we select $r_\varepsilon$ small enough so that $h_g^*(x)<2h_f^*(x)$ (note that $h_{g^{-1}}(y)=h_g^*(x)$). As in the proof of case i), we may select the neighborhoods $D_{y,\varepsilon}$, $D_{y,\varepsilon}'$ and $D_{y,\varepsilon}''$  with $g^{-1}$ and $h_f$ being replaced by $g$ and $h_f^*$, respectively, so that these neighborhoods satisfy all the covering assumptions of Theorem~\ref{thm:for analytic regularity}.  

On the other hand,  since $\lambda$ satisfies~\eqref{eq:estimate for lambda} and $\mu$ is locally Ahlfors $Q$-regular, we have
\begin{align*}
\Big(\frac{\diam(g^{-1}(D_{y,\varepsilon}))}{\diam(D_{y,\varepsilon})}\Big)^Q&\lambda(D_{y,\varepsilon}')\leq ci_{\ess}(x,f)\Big(\frac{L_{g^{-1}}(y,r_\varepsilon)}{r_\varepsilon}\Big)^Qr_\varepsilon^Q\\
&\leq ci_{\ess}(x,f)h_f^*(x)^{2Q}\Big(\frac{L_{g^{-1}}(y,r_\varepsilon)}{10h_f^{*2}(x)}\Big)^Q\leq ci_{\ess}(x,f)h_f^*(x)^{2Q}\mu\big(g^{-1}(D_{y,\varepsilon}'')\big)
\end{align*}
and
\begin{align*}
	\lambda\Big(B(y,10\diam(D_{y,\varepsilon}))\Big)\leq ci_{\ess}(x,f)\lambda(D_{y,\varepsilon}'),
\end{align*}
where the constant $c$ depends only on the bounded turning constant of $\Omega$ and the Ahlfors regularity constants of $\mu$ and $\lambda$. Set $\eta(y)=ci_{\ess}(g^{-1}(y),f)h_f^*(g^{-1}(y))^{2Q}$ and we can get rid of the index in the exact argument as in the second part of case i).

%

\end{proof}

\begin{remark}\label{rmk:on proof of m imply A}
   1). Our reasoning in this section was a bit more general than necessary
	for only these results, and perhaps might seem repetitive and roundabout. However, we need Theorem~\ref{thm:for analytic regularity} for later use anyway, and after it is proved, the remaining arguments seem to be the simplest way to complete the proof of Theorem~\ref{thm:metric implies all the other}.
	
   2). Using the pullback integration theory developed in Section~\ref{sec:basic pullback studies}, one can also argue directly to prove Theorem~\ref{thm:metric implies all the other}. Indeed, our original approach was based on it. However, the presentation will become much more complicated since the ``pullback line integral" and the ``pullback upper gradient" were necessary. 
   
   3). In Theorem~\ref{thm:metric implies all the other}, $h_f$ or $h_f^*$ is assumed to be everywhere bounded. As observed in~\cite[Remark 4.1]{bkr07}, the proof of Theorem~\ref{thm:metric implies all the other} works if $h_f$ (or $h_f^*$) is assumed to be finite everywhere and bounded outside a countable set when $p=Q$, finite everywhere and bounded outside a set of $\sigma$-finite $\mathcal{H}^{Q-p}$-measure when $1<p<Q$. By~\cite[Theorem 5.1 and Remark 5.3]{bkr07}, if $X$ supports a local $(1,q)$-Poincar\'e inequality (i.e. each point $x$ in $X$ has a neighborhood $U_x$ such that $(U_x,d,\mu)$ supports a local $(1,q)$-Poincar\'e inequality) for some $1\leq q\leq Q$, then the conclusion of Theorem~\ref{thm:metric implies all the other} remains valid if we assume that $h_f$ (or $h_f^*$) is finite everywhere and bounded from above $\mu$-a.e. in $X$.
   
   The assumption of a local $(1,q)$-Poincar\'e inequality on $X$ ensures that if $|\nabla f|\in L^Q_{loc}(X)$ is a $q$-weak upper gradient of a continuous mapping $f\colon X\to Y$, $1\leq q<Q$, then $f\in N^{1,Q}_{loc}(X,Y)$~(see e.g.~\cite[Proposition 4.4]{km98}). The self-improving nature of the Poincar\'e inequality, as discovered by Keith and Zhong~\cite{kz08}, implies that the above fact holds for $q=Q$ as well. Note, however, that in general (without a Poincar\'e inequality) the $p$-weak upper gradient depends on $p$; see~\cite[Theorem 1.1]{ds15}.

\end{remark}

\subsection{Spaces of locally $Q$-bounded geometry}\label{subsec:Spaces of locally $Q$-bounded geometry}

In this section we consider the case where $\tilde{\Omega}$ and/or $\Omega$ have locally $Q$-bounded geometry. The ``data" of $\tilde{\Omega}$ and/or $\Omega$, for this section, includes both the constant of (local) Ahlfors $Q$-regularity and the Loewner function.

In the previous theorem, (local) Ahlfors $Q$-regularity allows us to obtain analytic and geometric conditions from metric ones. In the absence of rectifiable curves, however, analytic and geometric conditions become vacuous. To get any mileage out of these, we need some conditions guaranteeing the existence of sufficently many rectifiable curves. For the next results, then, we consider the case where the domain and/or target have locally $Q$-bounded geometry. Here, the homeomorphic case was proved, mostly, in~\cite{hkst01}, though the equivalence between the liminf definition and the others was only known for $\R^n$~\cite{hk95} until that result was generalized to the Loewner case in~\cite{bkr07}.

When the source domain $\tilde{\Omega}$ has (locally) bounded geometry, we have the following forward-to-inverse result.

\begin{theorem}\label{thm:bounded geometry I}
	Suppose $\tilde{\Omega}$ has locally $Q$-bounded geometry, and $\Omega$ is locally Ahlfors $Q$-regular and locally LLC. Let $f\colon\tilde{\Omega}\to \Omega$ be an onto branched covering.  Then the following conditions are quantitatively equivalent:
	\begin{enumerate}
		\item $f$ is metrically $H$-quasiregular;
		\item $f$ is inverse metrically $H^*$-quasiregular;
		\item $f$ is weak metrically $h$-quasiregular with exponent $Q$;
		\item $f$ is weak metrically $h^*$-quasiregular with exponent $Q$;
	    \item $f$ is analytically $K$-quasiregular with exponent $Q$;
	    \item $f$ is geometrically $K$-quasiregular with exponent $Q$.
	\end{enumerate}
	Moreover, if these equivalent conditions are satisfied, then the inverse geometric and analytic definitions of quasiregularity hold with exponent $Q$ as well.
\end{theorem}

Bounded geometry in the target does not give us quite as much, but we still have
the following inverse-to-forward result.
\begin{theorem}\label{thm:bounded geometry II}
  If $\tilde{\Omega}$ is locally Ahlfors $Q$-regular and locally LLC, and $\Omega$ has locally $Q$-bounded geometry, then the inverse geometric or analytic definitions with exponent $Q$ imply the corresponding forward with exponent $Q$, quantitatively.
\end{theorem}

Combining Theorems~\ref{thm:bounded geometry I} and~\ref{thm:bounded geometry II} gives us our main result of this section. 
\begin{theorem}\label{thm:bounded geometry III}
	If both $\tilde{\Omega}$ and $\Omega$ have locally $Q$-bounded geometry, then all of the metric, geometric and analytic definitions with exponent $Q$ are quantitatively equivalent.
\end{theorem}

Theorem~\ref{thm:bounded geometry III} can be viewed as a branched version of~\cite[Theorem 9.8]{hkst01}, except that we do not yet have the notion of ``branched quasisymmetric mappings". We will introduce such a natural class of branched quasisymmetric mappings in Section~\ref{subsec:Branched quasisymmetric mappings} and prove the quantitatively equivalence of (weak) metic quasiregularity and local branched quasisymmetry in Theorem~\ref{thm:qr with bounded multiplicity is branched qs}.

\subsubsection{Auxiliary results}

We assume in thi section that $\tilde{\Omega}$ is locally Ahlfors $Q$-regular, locally LLC with constant $\lambda_{\tilde{\Omega}}$, and that $\Omega$ has locally $Q$-bounded geometry. Note in particular that $\Omega$ has the local LLC property as well, quantitatively. 
Since $\tilde{\Omega}$ is locally compact, connected and locally connected, the locall LLC property is equivalent to the local path LLC property.

Since $\tilde{\Omega}$ is locally path LLC, it is not hard to see that $\tilde{\Omega}$ (and hence $\tilde{\Omega}^f$ as well) must be locally path connected, via the path-lifting
property (cf.~Lemma~\ref{lemma:lift Floyd}).

We introduce one harmless topological assumption on $\tilde{\Omega}$; a non-quantitative
version of the LLC property, which will always hold when we equip $\tilde{\Omega}$ with a metric giving it locally $Q$-bounded geometry. Namely, we assume that $\tilde{\Omega}$ is \textit{strongly locally connected}, i.e., for every $x\in \tilde{\Omega}$, and every open neighborhood $\tilde{\Omega}_0$ of $x$, there are open neighborhoods $\tilde{\Omega}_3\subset\tilde{\Omega}_2\subset\tilde{\Omega}_1\subset\tilde{\Omega}_0$ of $x$ such that every two points in $\tilde{\Omega}_1\backslash \tilde{\Omega}_2$ can
be joined in $\tilde{\Omega}_0\backslash \tilde{\Omega}_3$ is connected.

For the moment, it should be emphasized that we are not equipping $\tilde{\Omega}$ with a metric, and instead we will analyze the pullback metric space $\tilde{\Omega}^f$ under the assumption that $\Omega$ has locally $Q$-bounded geometry. We will use the notation $\mathbb{A}(x,r,R)$ to denote the annulus $B(x,R)\backslash \overline{B}(x,r)$.

We need the following result, which strengthens the strong local path LLC condition for $\tilde{\Omega}$. Recall that the local path LLC condition requires  that for each $x\in \Omega$, there exists a radius $R_x>0$ such that $B(x,R_x)$ is path LLC with constant $\lambda_{\tilde{\Omega}}$.

\begin{lemma}\label{lemma:strong local path LLC}
	Let $n\in \mathbb{N}$. Then there is a constant $\lambda_n$, depending only on $n$ and $\lambda_{\tilde{\Omega}}$, such that for every $x\in \tilde{\Omega}$, every $x_1,\dots,x_n\in B(x,R_x/2)$, every $r<R_x/2$, and	every pair of points $x',x''\in B(x,R_x)\backslash \bigcup_{i=1}^nB(x_i,r)$, there is a path joining $x'$ to	$x''$ in $B(x,\lambda_nR_x)\backslash \bigcup_{i=1}^nB(x_i,r/\lambda_n)$.
\end{lemma}
\begin{proof}
	We argue by induction, as the case for $n=1$ follows easily from the local
	path LLC property of $\tilde{\Omega}$. Suppose the result holds for $n=k$, and that we are given
	$r<R_x/2$, $k+1$ points $x_1,\dots,x_{k+1}$, and points $x'$ and $x''$ as in the statement of the lemma. By assumption, we may join $x'$ and $x''$ with a path $\gamma\colon [0,1]\to B(x,\lambda_kR_x)\backslash \bigcup_{i=1}^kB(x_i,r/\lambda_k)$. We may assume with no loss of generality that $d(x_{k+1},x_i)\geq 2r/3\lambda_k$ for each $i=1,\dots,k$, since otherwise the conclusion holds with $\lambda_{k+1}=3\lambda_k$.
	
	Let $t_1,t_2\in [0,1]$ be the first and last values of $t$, respectively, for which $d(\gamma(t),x_{k+1})\leq r/(3\lambda_{\tilde{\Omega}}\lambda_k)$, if any such values exist. Then we may modify $\gamma$ by replacing $\gamma|_{[t_1,t_2]}$ with a path joining $\gamma(t_1)$ to $\gamma(t_2)$ lying entirely in the annulus $\mathbb{A}\big(x_{k+1},r/(3\lambda_{\tilde{\Omega}}^3\lambda_k),r/3\lambda_k\big)$. Since
	\begin{align*}
		\mathbb{A}\big(x_{k+1},r/(3\lambda_{\tilde{\Omega}}^3\lambda_k),r/3\lambda_k\big)\subset B(x,\lambda_kR_x)\backslash \bigcup_{i=1}^k B(x_i,r/(3\lambda_{\tilde{\Omega}}^3\lambda_k)),
	\end{align*}
	the conclusion holds for $\lambda_{k+1}=3\lambda_{\tilde{\Omega}}^3\lambda_k$.
\end{proof}

For a couple of standard modulus estimates we are going to present below, it is convenient to introduce the following concept, which essentials requires the projection mapping fullfills the assumption of Lemma~\ref{lemma:1.4}.
\begin{definition}[Good branching property]\label{def:good branch property}
The projection $\pi\colon \tilde{\Omega}^f\to \Omega$ is said to have the \textit{good branching property} if the following holds: Let $\Gamma$ and $\Gamma'$ be families of curves in $\tilde{\Omega}^f$ and $\Omega$, respectively. Suppose for each $\gamma'\in \Gamma'$, there are curves $\gamma_1,\dots,\gamma_m\in \Gamma$ and subcurves $\gamma_1',\dots,\gamma_m'$ of $\gamma'$ such that for each $i=1,\dots,m$, $\gamma_i'=\pi(\gamma_i)$, and for a.e. $s\in [0,l(\gamma')]$, $\gamma_i(s)=\gamma_j(s)$ if and only if $i=j$.
\end{definition}

For each $y\in \Omega$ and $0<r<s$, we let $\mathcal{A}(y,r,s)$ denote the family of curves in $B(y,s)$ joining $\partial B(y,r)$ and $\partial B(y,s)$.

We need the following rather standard modulus estimates.
\begin{lemma}\label{lemma:modulus estimates BG case}
	The radii $R(z)$ in Proposition~\ref{prop:existence of normal nbhd} may be chosen so that there are constants $C_1>C_2>0$, and $C_0>0$, depending only on the data of $\Omega$, such that for every $z\in \tilde{\Omega}^f$, every $C>2$, and every $2r<s<R(z)$,
	\begin{align}\label{eq:modulus estimates 1 BG case}
	 \min\{C_1\log(s/r)^{1-Q},C_0\}\leq \Modd_Q(\pi_z^{-1}(\mathcal{A}(z,r,s)))\leq i_{\ess}(z,\pi)C_2\log(s/r)^{1-Q}.
	\end{align}
	Moreover, if $\pi$ has the good branching property, then the radii may chosen so that
		\begin{align}\label{eq:modulus estimates 2 BG case}
		\min\{C_1\log(s/r)^{1-Q},C_0\}i(z,\pi)\leq \Modd_Q(\pi_z^{-1}(\mathcal{A}(z,r,s)))\leq i_{\ess}(z,\pi)C_2\log(s/r)^{1-Q}.
		\end{align}
\end{lemma}
\begin{proof}
	The lower modulus estimate in~\eqref{eq:modulus estimates 1 BG case refined} follows directly from Lemma~\ref{lemma:1.4} and from the Loewner property of $\Omega$. The upper estimate follows from~\eqref{eq:estimate for lambda} and the standardard modulus inequality for local Ahlfors $Q$-regular metric spaces (see e.g.~\cite[Lemma 3.14]{hk98}). If $\pi$ has the good branch property, we may use the stronger conclusion of Lemma~\ref{lemma:1.4} to obtain the lower estimiate as in~\eqref{eq:modulus estimates 2 BG case}.
\end{proof}

The next lemma shows that near a Lebesgue point $z$ of a subset  $S\subset \tilde{\Omega}^f$, many of the curves in $\pi_z^{-1}(\mathcal{A}(\pi(z),r,s))$ intersect $S$. Denote by $\Gamma_S$ the family of curves in $\tilde{\Omega}^f$ intersecting $S$ in at least one point.

\begin{lemma}\label{lemma:Lebesgue point BG case}
	Let $z\in S$ be a $\pi$-Lebesgue point of a Borel set $S\subset X$ and let $C\geq 2$. Then
	\begin{align*}
		\lim_{r\to 0}\Modd_Q(\pi_z^{-1}(\mathcal{A}(\pi(z),r,s))\backslash \Gamma_S)=0.
	\end{align*}
\end{lemma}
\begin{proof}
	Let $z\in S$ be a $\pi$-Lebesgue point of $S$. Note that the function $\rho=\frac{1}{(C-1)r}\chi_{B(z,Cr)\backslash S}$ is admissible for $\pi_z^{-1}(\mathcal{A}(z,r,Cr))\backslash \Gamma_S$, and so
	\begin{align*}
		\Modd_Q(\pi_z^{-1}(\mathcal{A}(\pi(z),r,s))\backslash \Gamma_S)\leq \int_{\tilde{\Omega}^f}\rho^Q d\lambda=\frac{1}{(C-1)^Qr^Q}\lambda(B(z,Cr)\backslash S).
	\end{align*}
	The right hand side of the last equation converges to 0 with $r\to 0$, since $z$ is a $\pi$-Lebesgue point of $S$ and $\lambda$ is pointwise Ahlfors $Q$-regular (see e.g.~\eqref{eq:estimate for lambda}).
\end{proof}

Applying the previous result to the sets $S=D_n$, $1\leq n\leq N(\pi,\tilde{\Omega}_0^f)$ immediately yields a refinement of Lemma~\ref{lemma:modulus estimates BG case}. Recall that the sets $D_n$ are defined in Section~\ref{subsec:some topological fact} as those points in $\tilde{\Omega}^f$ whose image under $\pi$ has multiplicity $n$ in $\tilde{\Omega}^f$.

\begin{lemma}\label{lemma:modulus estimates BG case refined}
	The radii $R(z)$ in Proposition~\ref{prop:existence of normal nbhd}, and the constants $C_1>C_2>0$ from Lemma~\ref{lemma:modulus estimates BG case}, may be chosen so that for $1\leq n\leq N(\pi,\tilde{\Omega}_0^f)$, $\lambda$-a.e. $z\in D_n$, and $2r<s<R(z)/\lambda_\Omega$,
	\begin{align}\label{eq:modulus estimates 1 BG case refined}
	\min\{C_1\log(s/r)^{1-Q},C_0\}\leq \Modd_Q(\pi_z^{-1}(\mathcal{A}(\pi(z),r,s))\cap \Gamma_{D_n})\leq C_2\log(s/r)^{1-Q}.
	\end{align}
	Moreover, if $\pi$ has the good branching property, then the radii may chosen so that
	\begin{align}\label{eq:modulus estimates 2 BG case refined}
	\min\{C_1\log(s/r)^{1-Q},C_0\}i(z,\pi)\leq \Modd_Q(\pi_z^{-1}(\mathcal{A}(\pi(z),r,s))\cap \Gamma_{D_n})\leq C_2\log(s/r)^{1-Q}.
	\end{align}
\end{lemma}

We now come to our key geometric argument for this section. Recall that by the pullback factorization, $f=\pi\circ g$, where $g\colon \tilde{\Omega}\to \tilde{\Omega}^f$ is a homeomorphism and $\pi\colon \tilde{\Omega}^f\to \Omega$ is a 1-BDD mapping.
\begin{proposition}\label{prop:key geometric BG case}
	Let $\tilde{\Omega}$ be (pointwise) Ahlfors $Q$-regular and (locally) LLC, let $\Omega$ be a locally geodesic metric space with locally $Q$-bounded geometry, and let $g$ satisfy the $K_I$-inequality with $K_I=K$ with exponent $Q$. Then there is a constant $c=c(K,n)\geq 1$, for each $n\in \mathbb{N}$, depending only on $K$, $n$, and the data of $\Omega$ and $\tilde{\Omega}$, such that $R(z)$ may be chosen (not necessarily quantitatively) small enough, so that
	\begin{align}\label{eq:regularity of g inverse}
		\diam(g^{-1}(B(z,5r)))^Q\leq c\mu(g^{-1}(B(z,r)))
	\end{align}
	for all $z\in \tilde{\Omega}^f$ with $i_{\ess}(z,\pi)=n$ and $r<R(z)/10$.
\end{proposition}
\begin{proof}
	Without loss of generality we assume $\tilde{\Omega}$ is relatively compact. Let $z\in \tilde{\Omega}^f$, $y=\pi(z)$, $x=g^{-1}(z)$, and $i_{\ess}(z,\pi)=n$. Fix a number $\lambda>2$, to be determined
	momentarily.
	
	Let $r<R(z)/10$ and let $L= L_g^*(x, 5r)= L_{g^{-1}}(z, 5r)$.
	
	We pause to note that since $\Omega$ is locally geodesic, we may assume $R(z)$ was chosen so that 
	\begin{align*}
		U(z,\pi,s)=g(U(x,f,s))=B(y,s)
	\end{align*}
	for each $s\leq 10r$, and thus for the remainder of the proof, we reserve the ``$U$" notation for $\tilde{\Omega}$, and simply abbreviate $U(x, s) := U (x, f, s)$. With our notation, it suffices to prove that
	\begin{align}\label{eq:key geometric BG aim}
		L^Q\leq c(K,n)\mu(U(x,r)).
	\end{align}
	
	By the definition of the essential index, there is a set $S\ni z$, contained in $B(z,R(z))$, such that $z$ is a $\pi$-Lebesgue point of $S$, and such that 
	\begin{align*}
		\lim_{t\to 0}N(\pi(z),\pi, B(z,t)\cap S)=n.
	\end{align*}
	We assume that $R(z)$ was chosen to guarantee that in fact $N(\pi(z),\pi, B(z,t)\cap S)=n$.
	
	Suppose then that there is a curve $\gamma\in \mathcal{A}(y,r/2,r)\cap \pi(\Gamma_S)$ such that for every lift $\tilde{\gamma}$ of $\gamma$ along $f$ in $U(x,r)$, $\diam(\tilde{\gamma})<L/(2n\lambda\lambda_n)$, where $\lambda_n$ is given as in Lemma~\ref{lemma:strong local path LLC}. Let $y'\in \gamma\cap \pi(S)$. Then it follows from Lemma~\ref{lemma:lift Floyd} that there are points $x_1',\dots,x_m'\in U(x,r)$, $m\leq n$, such that every such lift $\tilde{\gamma}$ of $\gamma$ intersects some $x_i'$, $1\leq i\leq m$. It then follows that each $\tilde{\gamma}$ is contained entirely in $\bigcup_{i=1}^mB(x_i',L/(2n\lambda\lambda_n))$.
	
	Since $U(x, 5r)$ is a connected set with diameter at least $L$, and $m\leq n$, there is some point $x'\in U(x,5r)\backslash \bigcup_{i=1}^m B(x_i',L/2n)$. By Lemma~\ref{lemma:strong local path LLC}, there is a curve $\tilde{\gamma}'$ joining $x'$ and $\partial U(x, 10r)$, contained in $\overline{U(x,10r)}\backslash \bigcup_{i=1}^mB(x_i',L/(2n\lambda_n))$.
	
	Let $\gamma'=f(\tilde{\gamma}')$. By the path lifting property, and the normality of $B(z, r)$, we may lift every curve joining $\gamma'$ to $\gamma$ to a curve joining $\tilde{\gamma}'$ to one of the lifts of $\gamma$.
	
	Now we consider the family $\Gamma$ of curves in $B(y, 10r)$ joining $\gamma'$ to $\gamma$, and the family $\tilde{\Gamma}=f^{-1}(\Gamma)\cap \mathcal{C}(U(x,5r))$. By the previous remark, $f(\tilde{\Gamma})=\Gamma$. Since every curve in $\Gamma$ intersects both $\overline{U(x,10r)}\backslash \bigcup_{i=1}^mB(x_i',L/(2n\lambda_n))$ and $\bigcup_{i=1}^mB(x_i',L/(2n\lambda\lambda_n))$, the local Ahlfors $Q$-regularity of $\mu$ ($R$ having been chosen appropriately) implies the modulus estimate
    \begin{align}\label{eq:key geometric BG 14}
    	\Modd_Q(\tilde{\Gamma})\leq \sum_{i=1}^{m}\Modd_Q\big(\mathcal{A}(x_i',L/(2n\lambda\lambda_n),L/(2n\lambda_n))\big)\leq cn(\log \lambda)^{1-Q}
    \end{align}
    Note that 
    \begin{align*}
    	\frac{d(\gamma',\gamma)}{\min\{\diam(\gamma'),\diam(\gamma)\}}\leq \frac{5r}{\min\{6r,r\}}=6.
    \end{align*}
	Invoking the local Loewner property and the $K_I$-inequality for $g$ and hence for $f$ by Theorem~\ref{thm:equivalence of G and A}, we obtain
	\begin{align}\label{eq:key geometric BG 15}
		C\leq \Modd_Q(\Gamma)\leq K\Modd_Q(\tilde{\Gamma}).
	\end{align}
	Combining inequalities~\eqref{eq:key geometric BG 14} and~\eqref{eq:key geometric BG 15}, and applying the geometric definition of quasiregularity,
	we obtain the estimate
	\begin{align}\label{eq:key geometric BG 17}
		\lambda\leq C_1(K,n)
	\end{align}
	where $C_1(K, n)$ depends only on $K$, $n$, and the data of the spaces. However, we have not yet chosen $\lambda$. Thus we choose $\lambda> C_1(K, n)$, so that inequality~\eqref{eq:key geometric BG 17} is a contradiction, invalidating our original assumptions about $\gamma$. We thus conclude
	that for every $\gamma\in \mathcal{A}(y,r/2,r)\cap \pi(\Gamma_S)$ has a lift $\tilde{\gamma}$ along $f$ in $U(x,r)$, such that $\diam(\tilde{\gamma})\geq L/(2n\lambda\lambda_n)$. 
	
	Let $\tilde{\mathcal{A}}$ be the family of all such lifts, so that $f(\tilde{\mathcal{A}})=\mathcal{A}(y,r/2,r)\cap \pi(\Gamma_S)$. By the definition of $\tilde{\mathcal{A}}$, the function $\frac{2n\lambda\lambda_n}{L}\chi_{U(x,r)}$ is admissible for $\tilde{\mathcal{A}}$. Combining this with the $K_I$-inequality and Lemma~\ref{lemma:modulus estimates BG case refined}, we have
	\begin{align*}
		C_0\leq \Modd_Q(\mathcal{A}(y,r/2,r)\cap \pi(\Gamma_S))\leq K_I\Modd_Q(\tilde{\mathcal{A}})\leq C_2(K,n)\mu(U(x,r))/L^Q
	\end{align*}
	for some constant $C_2(K, n)$ depending only on $K$, $n$, and the data of the underlying spaces, and $C_0$ depending on the local Loewner function. Letting $c(K, n)=C_2 (K, n)/C_0$ completes the proof.
\end{proof}

\begin{remark}\label{rmk:on Ahlfors regularity for key geometric argument}
	It is clear from the above proof that the (local) Ahlfors $Q$-regularity for the measure $\mu$ is not essential. Indeed, the only place where we have used this regularity assumption was in~\eqref{eq:key geometric BG 14}. The conclusion remains true, if the upper $Q$-regularity for $\mu$ is replaced by a pointwise upper $Q$-regularity, i.e., $\mu(B(x,r))\leq Ci_{\ess}(x,f)r^Q$ for $r$ small.
\end{remark}

\subsubsection{Absolute precontinuity}
In this section, we introduce the notion of absolute precontinuity of a branched covering $f\colon X\to Y$. The main result states that if a branched covering satisfy both the $K_O$-inequality and the $K_I$-inequality with the same exponent $Q$, then $f$ satisfies both Condition $N$ and Condition $N^{-1}$. We will use this fact later in the proof of Theorem~\ref{thm:bounded geometry I}.

Let $f\colon X\to Y$ be a branched covering. Let $\beta\colon I_0\to Y$ be a closed rectifiable curve and let $\alpha\colon I\to X$ be a curve in $X$ such that $f\circ \alpha\subset \beta$, i.e. $f\circ \alpha$ is the restriction of $\beta$ to some subinterval of $I_0$. If the length function $s_\beta\colon I_0\to [0,l(\beta)]$ is constant on some interval $J\subset I$, $\beta$ is also constant on $J$, and the discreteness of $f$ implies that $\alpha$ is constant on $J$ as well. It follows that there is a unique mapping $\alpha^*\colon s_\beta(I)\to X$ such that $\alpha=\alpha^*\circ(s_\beta|_I)$. 
The curve $\alpha^*$ is called the \textit{$f$-representation of $\alpha$ with respect to $\beta$}. We say that $f$ is \textit{absolutely precontinuous on $\alpha$} if $\alpha^*$ is absolutely continuous. 

\begin{figure}[h]
	\[
	\xymatrix{
		                                       &   I_0  \ar[dr]^{\beta}   &     \\
		I \ar[r]^{\alpha} \ar[ur]^{\iota} \ar[dr]_{s_\beta|_I} & X \ar[r]^{f} & Y  \\
		& [0,l(\beta)] \ar[u]_{\alpha^*} &
	}
	\]
	\caption{The absolute precontinuity of $f$ on $\alpha$}
\end{figure}
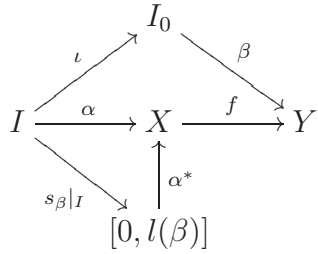

\begin{remark}\label{rmk:on absolutely precontinuity}
i). It follows immediately from the definition that if $f\colon X\to Y$ is a homeomorhpism, then $f$ is absolutely precontinuous on $\alpha$ if and only if $f^{-1}$ is absolutely continuous on $f\circ \alpha$. 

ii). It follows easily from the definition that a BLD mapping $h\colon X\to Y$ is always absolutely precontinuous on every absolutely continuous curve $\alpha$ of $X$.
\end{remark}

\begin{lemma}\label{lemma:absolutely precontinuity}
Suppose the branched covering $f\colon X\to Y$ satisfies both $K_O$-inequality and $K_I$-inequality with the same exponent $Q$. Then $f$ is absolutely precontinous on $Q$-almost every curve $\gamma$ in $X$.
\end{lemma}
\begin{proof}
	Recall that the pullback factorization gives us $f=\pi\circ g$, where $g\colon X\to X^f$ is a homeomorphism and $\pi\colon X^f\to Y$ is 1-BLD. By Remark~\ref{rmk:on absolutely precontinuity} ii), it suffices to show that $g$ is absolutely precontinuous on $Q$-almost every curve $\gamma$ in $X$. Note also that since $f$ satisfies the $K_I$-inequality, so is $g$. 
	
	Let $\tilde{\Gamma}$ be the curve family in $X$ such that $g$ fails to be absolutely precontinuous on $\gamma$, and let $\Gamma$ be the family of curves in $X^f$ such that $g^{-1}$ fails to be absolutely continuous. Note first that, by Remark~\ref{rmk:on absolutely precontinuity} i), $g$ is absolutely precontinuous on $\gamma$ if and only if $g^{-1}$ is absolutely continuous on $g\circ \gamma$. This together with the fact that $g$ satisfies the $K_O$-inequality implies that it suffices to show that $g^{-1}$ is absolutely continuous on $Q$-almost every curve $\gamma'\in \Gamma$. 
	
	On the other hand, the latter property follows immediately from the proof of Theorem~\ref{thm:equivalence of G and A}, which implies that $g^{-1}$ is analytically $K_I$-quasiconformal.  
\end{proof}

\begin{remark}\label{rmk:on KI imply abs precont}
Denote by $\Gamma_0$ the family of curves $\gamma$ in $X$ on which $f$ fails to be absolutely precontinuous on a subcurve of $\gamma$. The proof of Lemma~\ref{lemma:absolutely precontinuity} implies that if we only assume the $K_I$-inequality with exponent $Q$ for $f$, then $\Modd_Q(f(\Gamma_0))=0$. In other word, $f$ satisfies the \emph{Poletsky's lemma} with exponent $Q$ (e.g.~\cite[Section II 5]{r93}) whenever it satisfies the Poletsky's inequality with exponent $Q$.
\end{remark}

It is clear that if $f$ is absolutely continuous on a curve $\gamma\colon [a,b]\to X$, then it satisfies Condition $N$ on the curve $\gamma$, i.e. if $E$ intersects $\gamma$ on a set of zero $\mathcal{H}^1$-measure, then $f(E)$ intersects $f(\gamma)$ on a set of zero $\mathcal{H}^1$-measure. We have the following reverse statement for absolute precontinuity.

\begin{lemma}\label{lemma:Condition N inverse on curve}
	If $f\colon X\to Y$ is absolutely precontinuous on $\gamma\colon [a,b]\to X$, then $f$ satisfies Condition $N^{-1}$ on the curve $\gamma$, i.e. if $E$ intersects $\gamma$ on a set of positive $\mathcal{H}^1$-measure, then $f(E)$ intersects $f(\gamma)$ on a set of positive $\mathcal{H}^1$-measure. 
\end{lemma}
\begin{proof}
	It suffices to show that if $f(E)$ intersects $f(\gamma)$ on a set of zero $\mathcal{H}^1$-measure, then $E$ intersects $\gamma$ on a set of zero $\mathcal{H}^1$-measure. Let $\beta=f\circ \gamma$ and $E=\gamma(E_0)$, where $E_0\subset [a,b]$. Let $s_\beta\colon I\to [0,l(\beta)]$ be the length function of $\beta$. Then $\gamma=\gamma^*\circ s_\beta$ and $\gamma^*\colon [0,l(\beta)]\to X$ is absolutely continuous. 
	
	Given any $\varepsilon>0$, there exists some $\delta>0$ such that if $\{(b_k,b_{k+1})\}$ is a collection of disjoint intervals in $[0,l(\beta)]$ with $\sum_{k}|b_{k+1}-b_{k}|<\delta$, then $\sum_{k}d(\gamma^*(b_{k+1}),\gamma^*(b_k))<\varepsilon$.
	
	On the other hand, since $\mathcal{H}^1(\beta(E_0))=0$, for any given $\delta>0$, there exists $b_{k}=s_\beta(a_k)$, where $a_k\in E_0\subset I$, such that $\sum_k|s_\beta(a_{k+1})-s_\beta(a_k)|<\delta$. In particular, this implies that 
	$$\mathcal{H}^1(\gamma(E_0))\leq \sum_kd(\gamma(a_{k+1}),\gamma(a_k))=\sum_{k}d(\gamma^*(b_{k+1}),\gamma^*(b_k))<\varepsilon.$$ 
\end{proof}

We next point out the $K_O$-inequality and $K_I$-inequality together (with the same exponent $Q$) imply both Condition $N$ and Condition $N^{-1}$. In the lemma below, we assume that $X$ has locally $Q$-bounded geometry, and $Y$ is (locally) Ahlfors $Q$-regular and (locally) LLC.

\begin{lemma}\label{lemma:Condition N+N inverse}
Suppose the branched covering $f\colon X\to Y$ satisfies both $K_O$-inequality and $K_I$-inequality with the same exponent $Q$. Then $f$ satisfies both Condition $N$ and Condition $N^{-1}$.
\end{lemma}
\begin{proof}
	Note that by Theorem~\ref{thm:equivalence of G and A}, $f\in N^{1,Q}_{loc}(X,Y)$ and since $f$ is open, $f$ is \emph{locally psudomonotone} by~\cite[Lemma 3.3]{g14}, and so it follows from~\cite[Theorem 7.2]{hkst01} that $f$ satisfies Condition $N$.
	
	We next show that $f$ also satisfies Condition $N^{-1}$. Fix a set $E\subset X$ with $\mu(E)>0$. Let $\Gamma$ be the family of curves in $X$ that intersects $E$ with positive length, i.e., $\mathcal{H}^1(\gamma\cap E)>0$. Then $\Modd_Q(\Gamma)>0$ (see e.g.~\cite[Remark 6.7]{s00}), and so by the $K_O$-inequality, $\Modd_Q(f(\Gamma))>0$. On the other hand, Lemma~\ref{lemma:absolutely precontinuity} implies that $f$ is absolutely precontinuous on $Q$-almost every curve $\gamma\in \Gamma$. By Lemma~\ref{lemma:Condition N inverse on curve}, this implies that for $Q$-almost every curve $\gamma\in \Gamma$, if it intersects $E$ with positive length, then its image $f(\gamma)$ intersects $f(E)$ with positive length. Thus the family $\hat{\Gamma}$ of curves in $Y$ that intersects $f(E)$ with positive length has positive $Q$-modulus. Consequently, $\nu(f(E))>0$ (see e.g.~\cite[Lemma 5.2.15]{hkst15}). 
\end{proof}


\subsubsection{Proof of the main results}

It is clear, by Theorem~\ref{thm:equivalence of G and A}, to prove Theorem~\ref{thm:bounded geometry II}, it suffices to show that the Poletsky's inequality for $f$ with exponent $Q$ implies also the $K_O$-inequality for $f$ with exponent $Q$. This can be done quite directly with the help of Theorem~\ref{thm:for analytic regularity} and Proposition~\ref{prop:key geometric BG case}. 

\begin{proof}[Proof of Theorem~\ref{thm:bounded geometry II}]
	Note first that both the Poletsky's inequality and the property of locally $Q$-bounded geometry are easily seen to be preserved under locally Bilipschitz transformations. By	the uniform local quasi-convexity and local compactness of $\Omega$, the length metric $l_\Omega$ on $\Omega$ is uniformly locally bi-Lipschitz equivalent to the original metric $d$ on $\Omega$.
	We therefore may assume with no loss of generality that $\Omega$ is a length space. We also assume with no loss of generality that $f$ has multiplicity bounded by $N$. This is a harmless assumption, since proving our theorem on every compact neighborhood immediately proves it for all of $\tilde{\Omega}$, by the separability of $\tilde{\Omega}$. 
	
	We now observe that by Proposition \ref{prop:key geometric BG case}, we have, for every $\varepsilon>0$, and every $x\in \tilde{\Omega}$, some radius
	$r_x$ satisfying the conclusion of Proposition \ref{prop:key geometric BG case} with the additional property that $U(x,f,10r_x)\subset B(x,\varepsilon)$. We then set $D_{x,\varepsilon}= U(x, f, 10r_x )$, and $D_{x,\varepsilon}'=D_{x,\varepsilon}''=U(x, f, 2r_x)$. We claim that these neighborhoods satisfy the inequalities in Theorem~\ref{thm:for analytic regularity}, for $\eta(x)=c_3(n,i_{\ess}(x,f),K)$, where $c_3(n,i_{\ess}(x,f),K)$ depends only on the data of $\tilde{\Omega}$ and $\Omega$, $n$, $i_{\ess}(x,f)$ and $K$. Indeed, it follows from~\eqref{eq:estimate for lambda} that
	\begin{align*}
		\Big(\frac{\diam g(D_{x,\varepsilon})}{\diam D_{x,\varepsilon}}\Big)^Q\mu(D_{x,\varepsilon}')&\leq \Big(\frac{\diam B(g(x),10r_x)}{\diam D_{x,\varepsilon}}\Big)^Q\mu(D_{x,\varepsilon})\\
		&\leq c{r_x}^Qi(x,f)\leq ci_{\ess}(x,f)\lambda(B(g(x),2r_x))
	\end{align*}
	and additionally from Proposition~\ref{prop:key geometric BG case} that 
	\begin{align*}
	\mu\big(B(x,10 \diam D_{x,\varepsilon})\big)&\leq ci_{\ess}(x,f)\big(\diam g^{-1}(B(g(x),10r_x))\big)^Q\\
	&\leq ci_{\ess}(x,f)\mu\big(g^{-1}(B(g(x),2r_x))\big)=ci_{\ess}(x,f)\mu(D_{x,\varepsilon}').
	\end{align*}
    Moreover, by Lemma~\ref{lemma:covering lemma}, we may additionally assume that the covering criterion in Theorem~\ref{thm:for analytic regularity} is satisfied. Note that finite multiplicity of  $f$ implies that $\eta$ is bounded.
    As in the proof of Theorem~\ref{thm:metric implies all the other}, we may apply  Theorem~\ref{thm:for analytic regularity} twice to get rid of the depend on the multiplicity. In particular, this implies that $g$ is in $N^{1,Q}(\tilde{\Omega}, \tilde{\Omega}^f)$, and there is a constant $c(K)$ such that the minimal $Q$-weak upper gradient $|\nabla g|$ satisfies
	\begin{align*}
		\int_{\tilde{\Omega}}|\nabla g|^Qd\mu\leq c(K)\lambda(g(\tilde{\Omega})).
	\end{align*}
	
	Applying this inequality to every open subset of $\tilde{\Omega}$, and thus also to every Borel	subset of $\tilde{\Omega}$, yields the analytic definition for $g$, with the same constant $K'=c(K)$.
	By Theorem~\ref{thm:equivalence of G and A}, this is equivalent to the $K_O$-inequality with exponent $Q$ for $f$, with the same constant $K'$.
\end{proof}

We next turn to the proof of our main result, Theorem~\ref{thm:bounded geometry I}. For simplicity, we introduce the following notation 
$$\Delta(E,F)=\frac{\min\{\diam E,\diam F\}}{d(E,F)}.$$
\begin{proof}[Proof of Theorem~\ref{thm:bounded geometry I}]
	Our strategy to prove Theorem~\ref{thm:bounded geometry I} is to show that the geometric definition $G$ (with exponent $Q$) implies directly the inverse geometric definition $G^{-1}$ (with exponent $Q$) and that the combination of $G$ and $G^{-1}$ implies both the metric definition $H$ and inverse metric definition $H^*$. 
	
	\textit{Step 1}: $G$ implies $G^{-1}$.
	
	The proof of this part is almost identical to the proof of Theorem~\ref{thm:bounded geometry II} (which essentially deals with the case $G^{-1}$ implies $G$) and hence we quickly indicate the difference. As we have observed in the proof of Theorem~\ref{thm:bounded geometry II}, we essentially only need a version of Proposition~\ref{prop:key geometric BG case}. Note that when replace $g$ with $g^{-1}$ in Proposition~\ref{prop:key geometric BG case}, we only lose the fact that for $g^{-1}\colon \tilde{\Omega}^f\to \tilde{\Omega}$, the source domain $\tilde{\Omega}^f$ is not Ahlfors $Q$-regular. But this does not affect the proof of Proposition~\ref{prop:key geometric BG case}, since, as observed in Remark~\ref{rmk:on Ahlfors regularity for key geometric argument}, the only place where we have used this assumption was in~\eqref{eq:key geometric BG 14}, whereas the upper estimate in this case now depends on the essential index of $f$. Thus we may replace the Ahlfors $Q$-regularity of $\lambda$ by the pointwise Ahlfors $Q$-regular estimate~\eqref{eq:estimate for lambda}. It is easy to check, via the previous observation, that the conclusion of Proposition~\ref{prop:key geometric BG case} remains valid and hence we infer $G\Rightarrow G^{-1}$.
	
	\textit{Step 2}: $G+G^{-1}$ implies $H$ and $H^*$.
	
	\emph{First proof: } By Theorem~\ref{thm:equivalence of G and A}, we know $g\colon \tilde{\Omega}\to \tilde{\Omega}^f$ satisfies both the $K_O$-inequality and the $K_I$-inequality with the same exponent $Q$. We may thus view $g$ as a geometric quasiconformal mapping with exponent $Q$ in the sense of~\cite{hkst01} (with double side modulus inequalities) from a space of locally $Q$-bounded geometry to a LLC pointwise Ahlfors $Q$-regular space, where $\lambda$ satisfies~\eqref{eq:estimate for lambda}. Then one can run a similar argument as in~\cite[Proof of Theorem 9.8]{hkst01},~\cite[Section 4]{hk98} or essentially~\cite[Proof of Theorem 22.3]{v71} to show that for each $x\in \tilde{\Omega}$, $H_g(x)\leq C$, where $C$ depends on $i_{\ess}(x,f)$ (because of the upper bound from~\eqref{eq:estimate for lambda} on $\lambda$). The condition $i_{\ess}(x,f)=1$ $\mu$-a.e. in $\tilde{\Omega}$ then implies that $g$ is metrically quasiconformal, quantitatively, and so by Proposition~\ref{prop:characterization of QR via pullback factorization II} $f$ is metrically quasiregular, quantitatively.
	
	\emph{Second proof: } We will only show that $G+G^{-1}\Rightarrow H$, since the proof of the implication $G+G^{-1}\Rightarrow H^*$ is almost identical (and indeed is more standard in literature, e.g.~\cite[Theorem 1.4]{gw15}). 
	We will basically follow the idea of~\cite[Theorem 5.2]{mrv71} but in a ``inverse" manner (similar to the proof of Theorem 2 in~\cite{c06}). 
	
	For $x\in \tilde{\Omega}$ and $r>0$, we set $L=L_f(x,r)$, $l=l_f(x,r)$, $L^*=L^*(x,L)$ and $l^*=l^*(x,l)$. Let $D$ be the $x$-component of the ball $B(x,\lambda_{\tilde{\Omega}} r)$, $D'$ be the $y$-component of $B(x,r/\lambda_{\tilde{\Omega}})^c$ for some $y\notin B(x,r)$, $U_L=U(x,f,L)$, $U_l=U(x,f,l)$, and $U_R=U(x,f,R(x))$. We assume that $r$ is small enough so that $L^*\leq R(x)/2$. Then it is clear that
	\begin{align*}
		U_l\subset D\subset B(x,\lambda_{\tilde{\Omega}}r)\subset D'\subset U_L\subset U_R.
	\end{align*} 
	We may additionally assume that $N(y,f,U_R)\leq i_{\ess}(x,f)$ for all $y\in f(U_R)$ (otherwise repeat the argument as in the proof of Proposition~\ref{prop:key geometric BG case} to reduce to this case).
	
	Let $\gamma_1''$ be a path in $\overline{U_R\backslash D}$ joining $\partial U_R$ and some point $x_1\in \partial D$ such that $d(f(x),f(x_1))\geq L/\lambda_\Omega$. Now let $\gamma_1$ be a subcurve of $\gamma_1''$, still intersecting $\partial U_R$, but truncated at the first point of intersection with $\partial U_L$. 
	
	We construct a curve $\gamma_2$ similarly. Let $\gamma_2''$ be a path in $\overline{D'}$ joining $x$ and some point $x_2\in \partial B(x,r)$ such that $d(f(x),f(x_2))\leq \lambda_\Omega l$. Let $\gamma_2$ be a subcurve of $\gamma_2''$, still beginning at $x$, but truncated at the first point of intersection with $\partial U_l$.
	
	We may assume without loss of generality that $L/l\geq 2$.
	
	Let $E=\gamma_1\cap B(x,r)$ and $F=\gamma_2$. For simplicity, we write $B=B(x,r)$.
	
	Since $\diam(f(B))\geq \diam(f(\partial B))\geq L$, $\diam(f(F))\geq l/\lambda_\Omega$, and $d(f(B),f(F))\leq d(f(\partial B),f(F))\leq \lambda_\Omega l$, we have a quantitative lower bound for $\Delta(f(B),f(E))$, and similarly for $\Delta(f(B),f(F))$. Since every curve joining $E$ and $B$ has
	a subcurve joining $B(x,l^*)$ but not hit $B(x,r/\lambda_{\tilde{\Omega}})$, we may apply Lemma~\ref{lemma:modulus estimates BG case refined}, the
	Loewner property and $K_I$-inequality to obtain the following estimate
	\begin{align}\label{eq:BG 45}
		C\leq \Modd_Q(f(B),f(E))\leq K_I\Modd_Q(B,E)\leq C'K_I\log(r/l^*)^{1-Q}.
	\end{align}
	Similarly, we have
	\begin{align}\label{eq:BG 46}
	C\leq \Modd_Q(f(B),f(F))\leq K_I\Modd_Q(B,F)\leq C'K_I\log(L^*/r)^{1-Q}.
	\end{align}
	Combining the above two inequalities gives us
	\begin{align}\label{eq:BG 47}
		\log(L^*/l^*)\leq \log(L^*/r)+\log(r/l^*)\leq c(C,K_I,C').
	\end{align}
	
	On the other hand, by the $K_O$-inequality, and the fact that $f(E)\subset B(f(x),l)$ and $f(F)\subset B(f(x),L)^c$, we have
	\begin{align*}
		C\log(L/l)^{1-Q}&\geq K_O\Modd_Q(f(E),f(F))\geq \frac{\Modd_Q(E,F)}{i_{\ess}(x,f)}\\
		&\geq \frac{1}{i_{\ess}(x,f)}\min\big\{C_1\log(L^*/l^*)^{1-Q},C_0\big\}.\numberthis\label{eq:BG 48}
	\end{align*}
	Combining the estimates~\eqref{eq:BG 47} and~\eqref{eq:BG 48}, we conclude that $\frac{L}{l}\leq C(i_{\ess}(x,f),C,C_0,C_1)$. Note that $i_{\ess}(x,f)=1$ $f^*\nu$-a.e. in $\tilde{\Omega}$ and so by Lemma~\ref{lemma:Condition N+N inverse}, $i_{\ess}(x,f)=1$ $\mu$-a.e. in $\tilde{\Omega}$. In particular, this implies that $H_f(x)\leq C(i_{\ess}(x,f))<\infty$ for all $x\in X$ and $H_f(x)\leq C$ for $\mu$-a.e. $x\in \tilde{\Omega}$.
	
\end{proof}

\begin{remark}\label{rmk:on direct proof of group property of gemetric/analytic qc mapping}
It is worth pointing out that the arguments above give a relatively direct proof of the fact that geometrically (and hence analytically) quasiconformal mappings form a \emph{group}. The typical proof of this fact relies on the (quantitative) local equivalence of with quasisymmetric mappings and the easy fact that quasisymmetric mappings form a group, whereas the proof given here relies on the (quantitative) equivalence of analytic and geometric definitions of quasiconformality.
\end{remark}

\subsection{Size of the branch set and V\"ais\"al\"a's inequality}\label{subsec:Size of the branch set and V\"ais\"al\"a's inequality}
The V\"ais\"al\"a's inequality was first proved by V\"ais\"al\"a~\cite{v72} and it plays an important role in the theory of quasiregular mappings, in particular, many profound value-distributional type results; see~\cite{r93}.

We first recall the definition of V\"ais\"al\"a's inequality.
\begin{definition}[V\"ais\"al\"a's inequality]\label{def:vaisala inequality}
	We say that $f\colon X\to Y$ satisfies \textit{V\"ais\"al\"a's inequality with exponent $Q$} and with constant $K_I$ if the following condition holds: Suppose $m\in \mathbb{N}$, and $\Gamma$ and $\Gamma'$ are curve families in $X$ and $Y$ respectively, such that for each $\gamma'\in \Gamma'$, there are curves $\gamma_1,\dots,\gamma_m\in \Gamma$ such that $f(\gamma_k)$ is a subcurve of $\gamma'$ for each $k$, and for each $t\in [0,l(\gamma)]$ and each $x\in X$, we have $\#\{k:\gamma_k(t)=x\}\leq i(x,f)$. Then
	$$
	\Modd_Q(\Gamma')\leq K_I\Modd_Q(\Gamma)/m.
	$$
\end{definition}

V\"ais\"al\"a's inequality can be viewed as a refined version of the Poletsky's inequality. The following result implies that under mild assumptions on the image of the branch set, V\"ais\"al\"a's inequality follows from the Poletsky's inequality.
\begin{theorem}\label{thm:update poletsky to vaisala}
	Let $f\colon X\to Y$ be a branched covering between two metric measure spaces. Suppose $f$ satisfies the Poletsky's inequality with exponent $Q$ and with constant $K_I$, i.e., for every curve family $\Gamma$ in $X$,
	\begin{align*}
		\Modd_Q(f(\Gamma))\leq K_I\Modd_Q(\Gamma).
	\end{align*}
	If, additionally, $\nu(f(\mathcal{B}_f))=0$, then $f$ satisfies V\"ais\"al\"a's inequality with exponent $Q$ and with the same constant $K_I$.
\end{theorem}
\begin{proof}
	Given $\Gamma$ and $\Gamma'=f(\Gamma)$ as in the definition of V\"ais\"al\"a's inequality, it is clear that we only need to consider the rectifiable curves $\gamma'\in \Gamma$ and hence we may assume each $\gamma'\colon [0,l(\gamma')]\to Y$ is parametrized by arc-length. We may further assume that $\Modd_Q(\Gamma')>0$, since otherwise there is nothing to prove. Denote by $\Gamma_0$ the family of curves $\gamma$ in $X$ on which $f$ fails to be absolutely precontinuous on a subcurve of $\gamma$. By Remark~\ref{rmk:on KI imply abs precont}, $\Modd_Q(f(\Gamma_0))=0$.
	
	By Lemma~\ref{lemma:1.4} and the above fact, we only need to show that for $Q$-almost every $\gamma'\in \Gamma'\backslash f(\Gamma_0)$, there are curves $\gamma_1,\dots,\gamma_m\in \Gamma\backslash \Gamma_0$ and subcurves $\gamma_1',\dots,\gamma_m'$ of $\gamma'$ such that for each $i=1,\dots,m$, $\gamma_i'=f(\gamma_i)$, and for almost every $s\in [0,l(\gamma')]$, $\gamma_i(s)=\gamma_j(s)$ if and only if $i=j$.
	
	Towards a contradiction, suppose this is not true, then there exists a subfamily $\Gamma_0'\subset \Gamma'\backslash f(\Gamma_0)$ with $\Modd_Q(\Gamma_0')>0$ such that for each $\gamma'\subset \Gamma_0'$, there are curves $\gamma_1,\dots,\gamma_m\in \Gamma\backslash \Gamma_0$ and subcurves $\gamma_1',\dots,\gamma_m'$ of $\gamma'$ with $\gamma_i'=f(\gamma_i)$ for each $i=1,\dots,m$, but for some $i\neq j$ and some set $E\subset [0,l(\gamma')]$ of positive $\mathcal{H}^1$-measure, $\gamma_i(s)=\gamma_j(s)$ for all $s\in E$. Since $\#\{k:\gamma_k(t)=x\}\leq i(x,f)$, $\gamma_i(E)\subset \mathcal{B}_f$ and so $\gamma_i'(E)=f(\gamma_i(E))\subset f(\mathcal{B}_f)$.
	On the other hand, since $f$ is absolutely precontinuous on every curve $\gamma$, we have by Lemma~\ref{lemma:Condition N inverse on curve} that
	$$\mathcal{H}^1(\gamma'(E)\cap f(\mathcal{B}_f))>0$$ 
	for every $\gamma'\in \Gamma_0'$. In particular, this means that the family $\hat{\Gamma}$ of curves in $Y$ that intersects $f(\mathcal{B}_f)$ with positive length is not null with respect to the $Q$-modulus, which  contradicts with the fact that $\nu(f(\mathcal{B}_f))=0$ (cf.~\cite[Lemma 5.2.15]{hkst15}).
\end{proof}

\subsection{Geometric modulus inequalities for mappings of finite linear dilatation}\label{subsec:Geometric modulus inequalities for mappings of finite linear dilatation}
In this section, we give a brief discussion on weighted geometric modulus inequalities for mappings with finite linear dilatation. These weighted modulus inequalities are basically known for the so-called \textit{mappings of finite distortion} with certain integrability assumptions on the distortion; see for instance~\cite{ko06} for the $K_I$-inequality, \cite{r04} for the $K_O$-inequality, and the monographs~\cite{im01,hk14LNM} for more information on the theory of mappings of finite distortion in Euclidean spaces. As far as we know, there are no known results of this type even in the Euclidean spaces, due to the fact that the linear dilatation is of infinitesimal nature and is often difficult to argue directly. We also refer the interested readers to~\cite{ka02,km02} for more (analytic) properties of such mappings.

Let $f\colon X\to Y$ be a homeomorphism between two Ahlfors $Q$-regular metric measure spaces. Let us assume that $h_f(x)<\infty$ for all $x\in X$. Then by~\cite[Theorem 1.1]{w14}, $f$ satisfies Condition $N$ on $Q$-a.e. curve $\gamma$ in $X$. Moreover, $\lip f$ is a $Q$-weak upper gradient of $f$.

Set the ``outer distortion" $K_O$ of $f$ as
\begin{equation*}
K_O(x)=
\begin{cases}
\frac{\lip f(x)^Q}{J_f(x)} & \text{if } J_f(x)>0 \\
1 & \text{if } J_f(x)=0.
\end{cases}
\end{equation*}
Then it is clear that $c\leq K_O(x)<\infty$. Thus if we define a new measure $\lambda$ on $X$ via $d\lambda(x)=\frac{d\mu(x)}{K_O(x)}$. Then we also have that $f$ satisfies Condition $N$ on $Q$-a.e. curve with respect to the measure $\lambda$ as well. Consequently, $\lip f$ is also a $Q$-weak upper gradient of $f$ (with respect to the measure $\lambda$). Moreover, by definition, we have 
\begin{align*}
	\lip f(x)^Q\leq cK_O(x)J_f(x)=c\frac{d\nu}{d\lambda} \quad \text{for }\lambda\text{-a.e. } x\in X,
\end{align*}
which means $f\colon (X,\lambda)\to (Y,\nu)$ is an analytically $c$-quasiconformal mapping with exponent $Q$. As a consequence of the equivalence of geometric and analytic definitions of quasiconformality for arbitrary metric measure spaces~\cite[Theorem 1.1]{w12proc}, we obtain the following weighted $K_O$-inequality with exponent $Q$:
\begin{align*}
		\Modd_{Q,(cK_O)^{-1}}(\Gamma)\leq \Modd_Q(f(\Gamma)).
\end{align*} 
Note that the above modulus inequality holds for every curve family $\Gamma$, without omitting a exceptional family of curves $\gamma$ where $f$ fails to be absolutely continuous along $\gamma$; compare with the $K_O$-inequality from~\cite{r04}.

Let us now turn to the appropriate weighted $K_I$-inequality. First of all, note that we have (see e.g.~\cite[Equation 8]{w14}) 
\begin{align*}
	\lip f^{-1}(y)^Q\leq ch_f(f^{-1}(y))^QJ_{f^{-1}}(y)
\end{align*}
for all $y\in Y$. If we set 
\begin{equation*}
K_I(x)=
\begin{cases}
\frac{\lip f^{-1}(f(x))^Q}{J_{f^{-1}}(f(x))} & \text{if } J_f(x)>0 \\
1 & \text{if } J_f(x)=0,
\end{cases}
\end{equation*}
then $c\leq K_I(x)<\infty$. If we define a new measure $\lambda$ on $X$ by $d\lambda(x)=K_I(x)d\mu(x)$, then $\lambda\ll \mu$. Thus if $f^{-1}$ satisfies Condition $N$ with respect to $\mu$, i.e., $(f^{-1})^*\mu\ll \nu$, then $f^{-1}$ also satisfies Condition $N^{-1}$ with respect to $\lambda$. By a similar reason as in the previous situation, we know that $\lip f^{-1}$ is a $Q$-weak upper gradient of $f$. Moreover, we have      
\begin{align*}
	\lip f^{-1}(y)^Q\leq c\frac{d(f^{-1})^*\lambda}{d\nu}(y)\quad \text{for }\nu\text{-a.e. }y\in Y.
\end{align*}
Indeed, since $f^{-1}$ satisfies Condition $N$, $J_f(x)>0$ for $\mu$-a.e. $x\in X$. It suffices to show that for $\nu$-a.e. $y\in Y$, we have
\begin{align*}
	K_I(f^{-1}(y))\frac{d(f^{-1})^*\mu}{d\nu}(y)=\frac{d(f^{-1})^*\lambda}{d\nu}(y)
\end{align*} 
or, equivalently,
\begin{align*}
	\int_B 	K_I(f^{-1}(y))\frac{d(f^{-1})^*\mu}{d\nu}(y)d\nu(y)=\int_B \frac{d(f^{-1})^*\lambda}{d\nu}(y)d\nu(y)
\end{align*}
for all Borel set $B$ in $Y$. This follows from the area formula~\eqref{eq:area inequality general 1} and our definition of $\lambda$. Namely,
\begin{align*}
\int_B 	K_I(f^{-1}(y))\frac{d(f^{-1})^*\mu}{d\nu}(y)d\nu(y)&=\int_B K_I(f^{-1}(y))d(f^{-1})^*\mu(y)=\int_{f^{-1}(B)}K_I(x)d\mu(x)\\
&=\int_{f^{-1}(B)}d\lambda(x)=\int_{B} \frac{d(f^{-1})^*\lambda}{d\nu}(y)d\nu(y).
\end{align*}
This means that $f\colon (Y,\nu)\to (X,\lambda)$ is analytically $c$-quasiconformal with exponent $Q$. Again, by the equivalence of geometric and analytic definitions of quasiconformality for arbitrary metric measure spaces, we obtain the following weighted $K_I$-inequality with exponent $Q$:
\begin{align*}
	\Modd_Q(\Gamma)\leq \Modd_{Q,cK_I}(f(\Gamma)).
\end{align*}

Note that the assumption $f^{-1}$ satisfies Condition $N$ is equivalent to that $f$ satisfies Condition $N^{-1}$ (or $J_f>0$ $\mu$-a.e. in $X$), we may summarize the preceding results in the following form.

\begin{theorem}\label{thm:K_O for map of finite linear dilatation}
	Let $f\colon X\to Y$ be a homeomorphism between two Ahlfors $Q$-regular metric measure spaces. If $h_f(x)<\infty$ for all $x\in X$, then $f$ satisfies the weighted $K_O$-inequality with exponent $Q$, i.e.,
	\begin{align*}
		\Modd_{Q,K_O^{-1}}(\Gamma)\leq \Modd_Q(f(\Gamma)),
	\end{align*}	
	for some measurable function $K_O\colon X\to [1,\infty)$. Moreover, if $f$ satisfies Condition $N^{-1}$, then there exists a measurable function $K_I\colon X\to [1,\infty)$, such that $f$ satisfies the weighted $K_I$-inequality with exponent $Q$
	\begin{align*}
	\Modd_Q(\Gamma)\leq \Modd_{Q,K_I}(f(\Gamma)).
	\end{align*}
\end{theorem}

For cleanness of our exposition, we did not restrict ourself to the most general situation. However, it is worth pointing out that similar results as in Theorem~\ref{thm:K_O for map of finite linear dilatation} hold if we replace the assumption $h_f<\infty$ by the symmetric one $h_f^*<\infty$. There is nothing essentially new taking into account the asymmetry of our preceding arguments. Secondly, the assumption that $f$ is a homeomorphism can be weakened as $f$ is a branched covering, i.e., continuous, discrete and open mapping with locally bounded multiplicity. But some care need to be taken to conclude that $f$ satisfies Condition $N$ and/or Condition $N^{-1}$ on $Q$-almost every curves. This can be done by following the arguments from~\cite{w14} and using the pullback factorization, but we do not repeat the arguments here and leave it as an exercise for those interested readers.

\subsection{Branched quasisymmetric mappings}\label{subsec:Branched quasisymmetric mappings}

In the theory of metrically quasiconformal mappings, there is a proper subclass of mappings, termed \textit{quasisymmetric mappings}, which carry stronger, global metric information, but less restrictive as those bi-Lipschitz mappings.

\begin{definition}[Quasisymmetric mappings]\label{def:quasisymmetric mapping}
Let $\eta\colon [0,\infty)\to [0,\infty)$ be a homeomorphism. A homeomorphism $f\colon X\to Y$ is $\eta$-\textit{quasisymmetric} if 
\begin{equation}\label{eq:def of QS}
\frac{d(f(x),f(y))}{d(f(x),f(z))}\leq \eta\Big(\frac{d(x,y)}{d(x,z)} \Big)
\end{equation}
for all distinct triple $x,y,z\in X$.
\end{definition} 

The class of $\eta$-quasisymmetric mappings was introduced by Tukia and V\"ais\"al\"a~\cite{tv80} in their study of geometric embeddings of metric spaces. The importance of these mappings was soon realized in the study of (metrically) quasiconformal mappings; we refer the interested readers to the fundamental paper of Heinonen and Koskela~\cite{hk98} for more information on these development.

We next define a proper subclass of metrically quasiregular mappings that carry similar global metric information, but less restrictive as those BDD mappings.

\begin{definition}[Branched quasisymmmetric mappings]\label{def:branched qs}
Let $f\colon X\to Y$ be a branched covering. We say that $f$ is \textit{branched quasisymmetric} (BQS) if there exists a homeomorphism $\eta\colon [0,\infty)\to [0,\infty)$ such that
\begin{equation}\label{eq:def of BQS}
\frac{\diam f(E)}{\diam f(F)}\leq \eta\Big(\frac{\diam E}{\diam F} \Big)
\end{equation}
for all intersected 
continua $E, F\subset X$.
\end{definition}

\begin{definition}[Generalized quasisymmetric mappings]\label{def:generalized qs}
A homeomorphism $f\colon X\to Y$ is \textit{generalized $\eta$-quasisymmetric} if it is branched $\eta$-quasisymmetric. 
\end{definition}

We have the following concrete characterization of branched quasisymmetric mappings via the pullback factorization.
\begin{proposition}\label{prop:branched qs via pullback factorization}
A branched covering	$f\colon X\to Y$ is branched $\eta$-quasisymmetric if and only if $g\colon X\to X^f$ is generalized $\eta$-quasisymmetric.
\end{proposition}
\begin{proof}
	Suppose first that $g\colon X\to X^f$ is generalized $\eta$-quasisymmetric. Fix two continua $E,F\subset X$ with $E\cap F\neq \emptyset$. Note that $\pi\colon X^f\to Y$ is 1-BDD, and we have 
	\begin{align*}
	\frac{\diam f(E)}{\diam f(F)}=\frac{\diam \pi\circ g(E)}{\diam \pi\circ g(F)}=\frac{\diam g(E)}{\diam g(F)}\leq \eta\Big(\frac{\diam E}{\diam F} \Big),
	\end{align*}
	from which we infer that $f$ is branched $\eta$-quasisymmetric.
	
	Next, suppose that $f\colon X\to Y$ is branched $\eta$-quasisymmetric. Fix two continua $E,F\subset X$ with $E\cap F\neq \emptyset$. Then
	\begin{align*}
	\frac{\diam g(E)}{\diam g(F)}=\frac{\diam \pi\circ g(E)}{\diam \pi\circ g(F)}=\frac{\diam f(E)}{\diam f(F)}\leq \eta\Big(\frac{\diam E}{\diam F} \Big),
	\end{align*}
	from which we conclude that $g$ is generalized $\eta$-quasisymmetric.
\end{proof}

We next show that generalized quasisymmetric mappings between bounded turning metric spaces are quasisymmetric, quantitatively.

\begin{proposition}\label{prop:injective branched qs}
	Let $X$ have $c_0$-bounded turning and $Y$ $c$-bounded turning, and let $f\colon X\to Y$ be a homeomorphism. Then $f$ is generalized $\eta$-quasisymmetric if and only if it is $\psi$-quasisymmetric, quantitatively.
\end{proposition}
\begin{proof}
	Let us first assume that $f$ generalized $\eta$-quasisymmetric. Fix $x,y,z\in X$. Since $Y$ has $c$-bounded turning, there exists a continuum $F'\subset Y$ that joins $f(x)$ and $f(z)$ with $\diam F'\leq cd(f(x),f(z))$. For any continuum $E'\subset Y$ that joins $f(x)$ and $f(y)$, we have
	\begin{align*}
	 \frac{d(f(x),f(y))}{d(f(x),f(z))}&\leq \frac{\diam E'}{c^{-1}\diam F'}=\frac{c\diam f(E)}{\diam f(F)}\\
	 &\leq c\eta\Big(\frac{\diam E}{\diam F}\Big)\leq c\eta\Big(\frac{\diam E}{d(x,z)}\Big),
	\end{align*}
	where $E=f^{-1}(E')$ and $F=f^{-1}(F')$. Now the $c_0$-bounded turning condition on $X$ implies that we may select a continuum $E$, as the preimage of $E'$, so that $\diam E\leq c_0d(x,y)$. Thus we obtain
	\begin{align*}
		\frac{d(f(x),f(y))}{d(f(x),f(z))}\leq c\eta\Big(\frac{c_0d(x,y)}{d(x,z)}\Big).
	\end{align*}
	This means that $f$ is $\psi$-quasisymmetric with $\psi(t)=c\eta(c_0t)$.
	
	We next assume that $f$ is $\eta$-quasisymmetric. Fix two continua $E,F\subset X$ so that $E\cap F\neq \emptyset$. Let $x\in E\cap F$. Since $x\in E\cap F$, we may select $y\in E$ and $z\in F$ so that 
    $$\diam f(E)\leq 2d(f(x),f(y))\quad \text{and}\quad \diam F\leq 2d(x,z).$$
    Then it follows
    \begin{align*}
    	\frac{\diam f(E)}{\diam f(F)}&\leq \frac{2d(f(x),f(y))}{d(f(x),f(z))}\leq 2\eta\Big(\frac{d(x,y)}{d(x,z)}\Big)\\
    	&\leq 2\eta\Big(\frac{\diam E}{2^{-1}\diam F}\Big)=2\eta\Big(\frac{2\diam E}{\diam F}\Big).
    \end{align*}
    This implies that $f$ is generalized $\psi$-quasisymmetric with $\psi(t)=2\eta(2t)$.
\end{proof}
\begin{remark}\label{rmk:on qs to generalized qs}
	It is clear from the above proof that an $\eta$-quasisymmetric mapping $f\colon X\to Y$, between two arbitrary metric spaces (not necessarily having bounded turning), is generalized  $\psi$-quasisymmetric with $\psi(t)=2\eta(2t)$. 
\end{remark}

It is well-known that weak metrically quasiconformal mappings between metric spaces with locally $Q$-bounded geometry are locally quasisymmetric, quantitatively. The following result can be viewed as an analog in the branched category.
\begin{theorem}\label{thm:qr with bounded multiplicity is branched qs}
	Let $f\colon X\to Y$ be a weak metrically $H$-quasiregular mapping such that $N=N(f,X)<\infty$. Assume that both $X$ and $Y$ have locally $Q$-bounded geometry. Then $f$ is locally $\eta$-branched quasisymmetric, quantitatively, with $\eta$ depending only on $H$, $N$, and the data of $X$ and $Y$. 
\end{theorem}
\begin{proof}
	Let $f=\pi\circ g$ be the pullback factorization. Since $f$ has bounded multiplicity, $X^f$ has locally $Q$-bounded geometry. Since $f\colon X\to Y$ is weak metrically $H$-quasiregular, $g\colon X\to X^f$ is weak metrically $H$-quasiconformal, and by~\cite[Theorem 5.2 and Remark 5.3]{bkr07}, $g$ is locally $\eta$-quasisymmetric, quantitatively. The claim follows directly from Proposition~\ref{prop:branched qs via pullback factorization} and Proposition~\ref{prop:injective branched qs}. 
\end{proof}
\begin{remark}\label{rmk:on qr not branched qs}
	In Theorem~\ref{thm:qr with bounded multiplicity is branched qs}, the homeomorphism $\eta$ depends, quantitatively on the multiplicity $N$. In general, one can not get rid of this dependence from the theorem, as the simple analytic function $z\mapsto z^k$ indicated.
\end{remark}

\newpage
\section{Quasiregular mappings between equiregular subRiemannian manifolds}\label{sec:Quasiregular mappings between subRiemannian manifolds}
In this section, we collect the most recent development on the theory of quasiregular mappings in \textit{equiregular subRiemannian manifolds}. It is not known whether these manifolds have locally bounded geometry and thus our results from Section~\ref{sec:foundations of QR mappings in mms} cannot be applied directly in this setting. However, a combination of our results with that  from~\cite{gnw15,gw15} do give a relatively complete picture of the theory in these manifolds.

\newcommand{\Dst}{\mathcal{D}}

Let $M$ be a $C^\infty$-smooth manifold of dimension $n$ and let $\Dst\subset TM$ be a subbundle of constant rank $k$.
Define the following \emph{flag of distributions} inductively for $s\in\N$:
\[
\begin{cases}
\Dst^{0} &:= \{0\} \\
\Dst^{1} &:= \Gamma(\Dst) \\
\Dst^{s+1} &:= \Dst^{s} + \Co^\infty(M)\text{-}\Span\left\{[X,Z] : X\in\Dst^{1},\ Z\in\Dst^{s} \right\} ,
\end{cases}
\]
where $\Gamma(\Dst)$ is the set of all smooth sections of $\Dst$. For any set $S$ of vector fields, $\Co^\infty(M)\text{-}\Span(S)$ is the set of linear combinations of elements of $S$ with coefficients in the ring $\Co^\infty(M)$ of smooth functions $M\to\R$.

By definition we have
\[
\{0\}\subset \Dst^{1} \subset\dots\subset\Dst^{s}\subset\Dst^{s+1}\subset\dots\subset\Vec(M), 
\]
where $\Vec(M)$ is the space of all vector fields on $M$.
For any point $p\in M$, we have a \emph{pointwise flag} 
\[
\Dst^{s}_p := \{ Z(p):Z\in\Dst^{s}\} \subset T_pM .
\]

To such a flag we associate the following functions $M\to\N\cup\{+\infty\}$:
\begin{description}
	\item[Ranks] 	
	$k_s(p) := \dim(\Dst_p^{s})$, $s=1,2,\ldots$. 
	\item[Growth vector] 
	$n_s(p) := k_s(p)-k_{s-1}(p) = \dim( \Dst_p^{s}/\scr \Dst_p^{s-1} )$. 
	
	\item[Step]	
	$m(p) := \inf\{ s: \Dst_p^{s} = T_pM\}$. 
	
	\item[Weight]	
	for $i\in\{1,\dots,n\}$: $w_i:=s$ if and only if $i\in\{k_{s-1}+1,\dots,k_s\}$. 
	
\end{description} 
We have that that $k=k_1\le k_2\le\dots\le n$ and also $\sum_{i=1}^s n_i = k_s$. The function $p\mapsto (n_1(p),n_2(p),\dots)\in\N^\N$ is usually called the \emph{growth vector}. Notice that if $m(p)<\infty$, then 
\[
\{0\}\subset\Dst_p^{1}\subset\dots\subset\Dst_p^{m(p)}=T_pM .
\]
The subbundle $\Dst$ is usually called the \emph{horizontal distribution} and it is said to be \emph{equiregular} if $k_s$, and hence $n_s$ and $m$, are constants.
If $m<\infty$, then $\Dst$ is said to be \emph{bracket generating} and we have $k_s(p)=k_s=n$.
\begin{definition}[subRiemannian manifold]
	An \emph{equiregular subRiemannian manifold} is a triple $(M,\Dst,g)$ where $M$ is a smooth and connected manifold, $\Dst\subset TM$ is a bracket generating equiregular subbundle, and $g$ is a smooth inner product on the fibers $\mD_p$, $p\in M$, of $\mD$. 
\end{definition}

The inner product $g$ is called a \emph{horizontal metric} of $\Dst$. We use the notation $|v|_{g_p}$ or $\|v\|_{g_p}$ for the norm $\sqrt{g_p(v,v)}$ of a horizontal vector $v\in \Dst_p$. 
When there is no risk of confusion, we sometimes remove the subscript and write simply $|v|$ or $|v|_g$ etc.

\begin{definition}[subRiemannian distance]
	An absolutely continuous curve 
	$\gamma\colon [0,1]\to M$ is called a \emph{horizontal curve} if $\gamma'(t)\in \Dst_{\gamma(t)}$ for almost every $t\in[0,1]$. 
	
	The \emph{length} $l(\gamma)$ of a horizontal curve $\gamma\colon [0,1]\to M$ is
	\[
	l(\gamma) := \int_0^1\|\gamma'(t)\|\dd t.
	\]
	The \emph{subRiemannian distance} is defined by:
	\[
	d_g(p,q) := \inf_\gamma\left\{
	l(\gamma):
	\text{ $\gamma$ is a horizontal curve joining $p\in M $ to $q\in M$}
	\right\} .
	\]
\end{definition}

An equiregular subRiemannian manifold $M$ can be endowed in a canonical way with a smooth volume $\Vol_M$ that is called \emph{Popp measure}. The construction can be found in \cite{br13,m02}. Moreover, when edowed with the subRiemannian distance and the Popp volume measure, an equiregular subRiemannian manifold $M=(M,d,\Vol_M)$ becomes a geodesic metric measure space. 

For each $x_0\in M$, there exists a neighborhood $U$ of $x_0$ such that $(U,d,\Vol_M)$ has locally $Q$-bounded geometry. But we caution that it is not known whether the data associated to the local $Q$-bounded geometry at each point is uniformly bounded. Namely, for two dist points $x$ and $y$, both the metric spaces $(U,d,\Vol)$ and $(V,d,\Vol)$ has locally $Q$-bounded geometry, but the associated data (for locally $Q$-bounded geometry) might depend on $U$ and $V$, respectively.

All the different definitions of quasiregularity, as introduced in Section~\ref{subsec:Definitions of quasireguarity in general metric measure spaces}, directly make senses in the subRiemannian manifolds. However, there is a ``better" definition of quasiregularity that reflects the geometry of subRiemannian manifolds. This definition was introduced in~\cite{gl15}\footnote{There is one more equivalent definition, the so-called subRiemannian quasiregular mappings, introduced in~\cite{gl15}, but the formulation requires the local Popp extension and so we do not include it here.}.
\begin{definition}(Horizontally $K$-quasiregular mappings)
	Let $f: (M,g)\to(N,h)$ be a branched covering between equiregular subRiemannian manifolds $(M,g)$ and $(N,h)$. We say that the mapping $f$ is \textit{horizontally $K$-quasiregular with exponent $Q$} if $f\in N^{1,Q}_{loc}(M,N)$ and it satisfies
	$$
	||g^{-1}f^*h||^r\leq K \det(g^{-1}f^*h) \mbox{ a.e. in } M.
	$$
\end{definition}
The norm $||g^{-1}f^*h||$ is the sup-norm of $g^{-1}f^*h$. We remark that the horizontal $\binom{1}{1}$-tensor $g^{-1}f^*h$ is the same as 
$$
Df^*Df : (\mD_M,g) \to (\mD_N,h),
$$
where the adjoint $Df^*$ is defined as the adjoint of $Df$ between inner product spaces $(\mD_M,g)$ and $(\mD_N,h)$, cf.~\cite{l14}. We also note that we can regard $g^{-1}f^*h$ as an element of $\mbox{End}(\mD_M)$ implying that the eigenvalue problem for $g^{-1}f^*h$ is well defined, i.e. independent of the choice of basis for $\mD_M$.

The following result was obtained very recently by Liimatainen and the first-named author~\cite{gl15}.
\begin{theorem}\label{thm:equivalence of quasiregular mappings}
		Let $f\colon (M,g)\to (N,h)$ be a branched covering between two equiregular subRiemannian manifolds of homogeneous dimension $Q\geq 2$ and rank $k$. Then the following conditions are quantitatively equivalent:
		\begin{align*}
		1)\quad  f  &\mbox{ is a metrically $H$-quasiregular mapping}, \\
		2)\quad f  &\mbox{ is a weak metrically $H$-quasiregular mapping}, \\ 
		3)\quad f &\mbox{ is a horizontally $\widehat{H}$-quasiregular mapping with exponent $Q$}, \\ 
		4)\quad f  &\mbox{ is an analytically $K$-quasiregular mapping with exponent $Q$}, \\ 
		5)\quad f &\mbox{ is a geometrically $K$-quasiregular mapping with exponent $Q$}.
		\end{align*}
		Moreover, we have the following precise dependences on the quasiregularity constants $H,\widehat{H}$, and $K$:
		\begin{itemize}
			\item If $f$ is weak metrically $H$-quasiregular, then it is analytically $K$-quasiregular with $K=H^{Q-1}$ and horizontally $\widehat{H}$-quasiregular with $\widehat{H}=H^{k-1}$. 
			
			\item If $f$ is analytically $K$-quasiregular, then it is metrically $H$-quasiregular with $H=K$ and horizontally $\widehat{H}$-quasiregular with $\widehat{H}=K$. 
			
			\item If $f$ is horizontally $\widehat{H}$-quasiregular, then $f$ is analytically $K$-quasiregular with $K=\widehat{H}^{Q-1}$ and metrically $H$-quasiregular with $H=\widehat{H}$.
		\end{itemize}
\end{theorem}

Based on Theorem~\ref{thm:equivalence of quasiregular mappings}, we will simply say that $f\colon M\to N$ is a $K$-quasiregular mapping if it is $K$-quasiregular according to one of five definitions in Theorem~\ref{thm:equivalence of quasiregular mappings}. We also refer the interested readers to~\cite{gnw15} for more analytic properties of quasiregular mappings in the subRiemannian manifolds.

As a particular application of our general theory of quasiregular mappings studied in the previous section, we obtain the important V\"ais\"al\"a's inequality when the subRiemannian manifolds are Ahlfors regular. 

\begin{corollary}\label{coro:QR subRiemannian}
	Let $f\colon M\to N$ be a $K$-quasiregular mapping between two Ahlfors $Q$-regular, $Q\geq 2$, equiregular subRiemannian manifolds. Then $f$ satisfies the V\"ais\"al\"a's inequality with constant $K_I$, where $K_I$ depends only on $K$ and the Ahlfors regularity constants of $M$ and $N$.
\end{corollary}
\begin{proof}
By either~\cite[Corollary 1.2]{gw15} or~\cite[Theorem B]{gnw15}, we have $\mathcal{H}^Q\big(f(\mathcal{B}_f)\big)=0$ whenever $f\colon M\to N$ is metrically $K$-quasiregular. By the proof of Theorem~\ref{thm:metric implies all the other} and Remark~\ref{rmk:on proof of m imply A}, we know that $f$ satisfies the Poletsky's inequality with some constant $K_I'$ that depends quantitatively on $K$ and on the Ahlfors regularity constants of the spaces. Finally, the claim follows from Theorem~\ref{thm:update poletsky to vaisala}.
\end{proof}

We do not know whether an Ahlfors $Q$-regular equiregular subRiemannian manifold $M$ necessarily has locally $Q$-bounded geometry. If so, Corollary~\ref{coro:QR subRiemannian} would follow directly from Theorem~\ref{thm:bounded geometry III}. It is also clear from the above proof that Corollary~\ref{coro:QR subRiemannian} remains valid if both $M$ and $N$ are \textit{uniformly locally Ahlfors $Q$-regular}, i.e., there exists a positive constant $C_d$ such that for each point $x$, there is some positive radius $r_x$ making the measure $\Vol$ Ahlfors $Q$-regular with constant $C_d$ for the metric measure space $\big(B(x,r_x),d,\Vol\big)$. In other words, we require the Ahlfors regularity constant is uniform, but allowing the radius vary at each point.

As an obvious consequence of Theorem~\ref{thm:equivalence of quasiregular mappings} and Corollary~\ref{coro:QR subRiemannian}, we point out that under the same assumptions as in Corollary~\ref{coro:QR subRiemannian}, the class of quasiconformal mappings form a group. Namely, if $f\colon M\to N$ is a $K$-quasiconformal mapping, then $f^{-1}\colon N\to M$ is $K'$-quasiconformal, with $K'$ depending only on $K$ and the Ahlfors regularity constants of $M$ and $N$.

\newpage
\section{Bi-Lipschitz embeddability of BLD Euclidean spaces}\label{sec:biLipschitz embeddability}

\subsection{Bi-Lipschitz embedding problem}
An interesting question, also from the point of view of applications, asks when a given metric space admits a bi-Lipschitz embedding into some finite dimensional Euclidean spaces (see e.g.~\cite[Section 16.4]{h07}). However, in general, it turns out to be a  very difficult problem to find nontrivial naturally bi-Lipschitz invariant sufficient criteria for the embeddability; see~\cite{ds93,ds97,s93,s96b,s96c,s99,h03} and the references therein for both positive and negative results along this direction.

In this section, we consider the bi-Lipschitz embedding problem by assuming a priori the existence of a BDD covering with bounded multiplicity. This kind of assumption was imposed earlier by Heinonen and Rickman~\cite{hr02} in their study of geometric branched coverings and later it played an important role in the remarkable papers of Heinonen--Sullivan~\cite{hs02} and Heinonen--Keith~\cite{hk11} to obtain local bi-Lipschitz parametrization of certain metric spaces. 

Our first main result of this section is the following bi-Lipschitz embedding theorem for BDD coverings with finite multiplicity.

\begin{theorem}\label{thm:biLip embedding I}
	Let $f\colon X\to Y$ be a BDD branched covering, with $N=N(f,X)<\infty$, and suppose that $Y$ is doubling and that $X$ and $Y$ have bounded turning. Then there is a bi-Lipschitz embedding of $X$ into $Y\times \R^{c_d(N-1)}$, where $c_d$ depends only on the data of $Y$.
\end{theorem}

Recall that a metric space $X$ is \emph{doubling with constant $M$}, $M\geq 1$ an integer, if for each ball $B(x,r)$ in $X$, every $r/2$-separated subset of $B(x,r)$ has at most $M$ points. 

As a consequence of Theorem~\ref{thm:biLip embedding I}, we obtain that the domain of a BDD branched covering inherits bi-Lipschitz embeddability into Euclidean space from the image. Taking into account of Proposition~\ref{prop:BLD to BDD}, we deduce the following corollary.

\begin{corollary}\label{coro:biLip embedding II}
	Let $f\colon X\to Y$ be a BDD branched covering, with $N=N(f,X)<\infty$, where $X$ and $Y$ have bounded turning and $Y$ embeds bi-Lipschitzly into some Euclidean space. Then $X$ may also be embedded bi-Lipschitzly into some (possibly larger) Euclidean space. In particular, every BLD covering of a locally compact quasiconvex metric space over $\R^n$ with finite maximal multiplicity embeds bi-Lipschitzly into some Euclidean space. 
\end{corollary}

In general, we are restricting our attention as much as possible to the global statements to simplify the exposition; the numerous formulations and proofs of similar local statements are left as routine exercises for the interested readers. We do, however, wish to highlight a few consequences arising from localization.

First, by localizing Corollary \ref{coro:biLip embedding II}, we obtain that \emph{every locally BLD-Euclidean metric space of dimension $n$ admits local bilipschitz embeddings into Euclidean spaces, with the dimension
of the local embeddings depending only on $n$ and the local degree of the BLD covering}. This answers in the affirmative an open question of Heinonen and Rickman~\cite[Remark 6.32 (a)]{hr02}.

Of course, under such localized assumptions, the dimensions of the Euclidean spaces receiving the neighborhoods of $X$ might very well be unbounded. As a result, we cannot hope to obtain, from our methods, a single globally defined, locally bijective mapping from $X$ into a Euclidean space. If, on the other hand, conditions are ``uniformly local", in the sense that their associated constants are still defined globally, then we are able to obtain such a mapping.

\begin{theorem}\label{thm:biLip embedding III}
Let $f\colon X\to Y$ be a locally BDD branched covering, and suppose that $X$ and $Y$ have uniformly locally bounded turning, that $Y$ is uniformly locally doubling, and that $f$ has uniformly locally bounded multiplicity $M<\infty$. Then there is a globally defined, uniformly locally bi-Lipschitz mapping $\phi\colon X\to Y\times \R^{c_d(M-1)}$, where $c_d$ depends only on the data of $Y$.
\end{theorem}

By using the typical gluing argument, together with the Nashing embedding theorem, it is plausible that one can update Theorem~\ref{thm:biLip embedding III} so that $\phi$ is globally injective, and hence is 
bi-Lipschitz with respect to the length metric. For simplicity of our exposition, we do not get involved with this technical issue, but leave it to those interested readers. 

Another observation we would like to point out is that as a consequence of Theorem~\ref{thm:biLip embedding I}, we have for each $K$-quasiregular mapping $f\colon \R^n\to \R^n$ a canonical factorization $f=\pi\circ g$, where $g\colon \R^n\to X^f$ is $K$-quasiconformal and $\pi\colon X^f\to \R^n$ with the metric space $X^f$ being \emph{a generalized $n$-manifold of type $A$} (in the sense of Heinonen-Rickman~\cite{hr02}). This provides a ncie motivation of the study of BLD mappings from generalized $n$-manifolds of type $A$ to $\R^n$ and quasiconformal mappings from $\R^n$ to generalized $n$-manifolds of type $A$.

\subsection{Proof of the main results}

We first make a simple reduction, which is not strictly speaking necessary, but will allow us a somewhat cleaner exposition. Since $X$ and $Y$ have bounded turning, and $f$ is BDD, we have by Remark 3.7 that $Y$ is bi-Lipschitz equivalent to $\iota^*Y$ and $X$ is bi-Lipschitz equivalent to $f^*Y=f^*\iota^*Y$, and so we may assume with no loss of generality that $X$ and $Y$ have
1-bounded turning, and $f$ is 1-BDD. Notice that in the proof of Theorem~\ref{thm:biLip embedding I} given below, $c_d$ will depend explicitly only  on the doubling constant, but because of our reduction, it implicitly depends on the constant of bounded turning in $Y$ as well.

Motivated by the fact that the theorem is trivial when $f$ is one-to-one, our proof of Theorem~\ref{thm:biLip embedding I} begins with the following simple lemma, which makes quantitative the fact that the lack of injectivity is our only obstacle.

\begin{lemma}\label{lemma:for biLip embedding}
	Let $f\colon X\to Y$ satisfy the hypotheses of Theorem~\ref{thm:biLip embedding I}. Suppose that for some $\varepsilon>0$, there is an $L$-Lipschitz mapping $\phi\colon X\to W$ 
	into some metric space $W$ such that for all $x,x'\in X$ such that $f(x)=f(x')$, we have $\varepsilon d(x,x')\leq d(\phi(x),\phi(x'))$. Then the mapping
	\begin{align*}
		\psi=f\times \phi\colon X\to Y\times W
	\end{align*}
	is a bi-Lipschitz embedding.
\end{lemma}
\begin{proof}
	That $\psi$ is Lipschitz is clear. For the reverse inequality, let $0<\delta<\frac{\varepsilon}{L+\varepsilon}$, so that $\varepsilon(1-\delta)-L\delta>0$, and suppose first that $x_1,x_2\in X$, with $d(y_1,y_2)\leq \delta d(x_1,x_2)$. Here we	denote $y_i=f(x_i)$, $w_i=\phi(z_i)$. Since $f$ is 1-BDD and $X, Y$ have 1-bounded turning, there is some $x_1'\in f^{-1}(y_1)$ such that
	\begin{align*}
		d(x_1',x_2)=d(y_1,y_2)\leq \delta d(x_1,x_2),
	\end{align*}
	whereby 
	\begin{align*}
		(1-\delta)d(x_1,x_2)\leq d(x_1,x_2)-d(x_1',x_2)\leq d(x_1,x_1').
	\end{align*}
	Then we have
	\begin{align*}
		(1-\delta)d(x_1,x_2)&\leq d(x_1,x_1')\leq \frac{d(w_1,w_1')}{\varepsilon}\leq \frac{d(w_1,w_2)+d(w_2,w_1')}{\varepsilon}\\
		&\leq \frac{d(w_1,w_2)+Ld(x_2,x_1')}{\varepsilon}\leq \frac{d(w_1,w_2)+L\delta d(x_1,x_2)}{\varepsilon},
	\end{align*}
	and so
	\begin{align*}
		(\varepsilon(1-\delta)-L\delta)d(x_1,x_2)\leq d(w_1,w_2).
	\end{align*}
	Thus, for every $x_1,x_2\in X$, we have
	\begin{align*}
		\min\big\{\varepsilon(1-\delta)-L\delta,\delta\big\}d(x_1,x_2)\leq \max\big\{d(w_1,w_2),d(y_1,y_2)\big\}.
	\end{align*}	
\end{proof}

\begin{remark}\label{rmk:on quantitative bi-Lipschitz dependence}
	It is clear from the proof that the bi-Lipschitz constant associated to the embedding depends quantitatively only on $\varepsilon$, $L$ and the bounded turning constants of $X$ and $Y$.
\end{remark}

\begin{proof}[Proofs of Theorem~\ref{thm:biLip embedding I} and Theorem~\ref{thm:biLip embedding III}]
	For each $y\in Y$, and each $k=1,\dots,N-1$, let $R^k(y)$ be the smallest
	number such that $f^{-1}(B(y,5R^k(y)))$ consists of at most $k$ (connected) components. We have	that $R^k(y)>0$ if and only if $N(y,f)>k$, that 
	 \begin{align*}
	 	0=R^N(y)=R^{N(y,f)}(y)<R^{N(y,f)-1}(y)\leq \cdots\leq R^1(y),
	 \end{align*}
	 and that for all $x_1\neq x_2\in f^{-1}(y)$, there is some $k$, $1\leq k\leq N(y,f)-1$, such that
	 \begin{align}\label{eq:biLip embed 5}
	 	\frac{1}{2}d(x_1,x_2)\leq 5R^k(y)\leq d(x_1,x_2);
	 \end{align}
	 the last statement follows from Lemma~\ref{lemma:pullback property 1}.
	 
	 Fix $k$, and let $B_y^k=B(y,R^k(y))$ for each $y\in Y$, so that $B_y^k$ is nonempty if and only if	$N(y,f)>k$. Also, let 
	 \begin{align*}
	 	Y^k=\big\{y\in Y:N(y,f)>k\big\}=\big\{y\in Y:R^k(y)>0\big\}.
	 \end{align*} 
	 Since balls are connected in $Y$, and $f^{-1}(B(y,r))$ has more than $k$ components for each $r<5R^k(y)$, Lemma~\ref{lemma:lift Floyd} implies that each $y'\in 5B_y^k$ has more than $k$ preimages, hence $5B_y^k\subset Y^k$.
	 (Observe this technically holds for all $y\in Y$, since the open ball $5B_y^k$ is empty for $y\notin Y^k$.)
	 
	 Thus we actually have $Y^k=\bigcup_{y\in Y}5B_y^k=\bigcup_{y\in Y}B_y^k$. Moreover, for every $y,y'\in Y$, the inclusion $B(y,5R^k(y))\subset B(y',5R^k(y)+d(y,y'))$ implies
	 \begin{align}\label{eq:biLip embed 6}
	 	5R^k(y')<5R^k(y)+d(y,y'),
	 \end{align}
	 so that $R^k$ is $\frac{1}{5}$-Lipschitz in $y$. This in turn implies that for every $y,y'\in Y$ such that $B_y^k\cap B_{y'}^k\neq \emptyset$, we have
	 \begin{align*}
	 	R^k(y')-R^k(y)\leq \frac{1}{5}d(y',y)\leq \frac{1}{5}(R^{k}(y')+R^k(y)),
	 \end{align*}
	 so that
	 \begin{align}\label{eq:biLip embed 7}
	 	\frac{2}{3}R^k(y)\leq R^k(y')\leq \frac{3}{2}R^k(y),
	 \end{align}
	 and similarly, for every $y,y'\in Y$ such that $2B_y^k\cap 2B_{y'}^k\neq \emptyset$, 
	 	 \begin{align}\label{eq:biLip embed 8}
	 	 \frac{3}{7}R^k(y)\leq R^k(y')\leq \frac{7}{3}R^k(y).
	 	 \end{align}
	 	 
	 We next choose a maximal subset $T^k\subset Y^k$ such that for each $y,y'\in T^k$, $d(y,y')\geq \frac{1}{2}R^k(y)$. By the inequality~\eqref{eq:biLip embed 7}, $Y^k=\bigcup_{y\in T^k}B_y^k$ and by inequality~\eqref{eq:biLip embed 8}, for each $y\in T^k$ there are at
	 most $c_d-1$ other points $y\in T_k$ such that $2B_y^k\cap 2B_{y'}^k\neq \emptyset$, where $c_d$ depends only on the doubling constant of $Y$. We may therefore choose $c_d$ disjoint subsets $T_j^k\subset T^k$, $j=1,\dots,c_d$, such that $T^k=\bigcup_{j=1}^{c_d}T_j^k$, and for each $y,y'\in T_j^k$, $y\neq y'$, we have $2B_y^k\cap 2B_{y'}^k=\emptyset$.
	 
	Now let $X^k=f^{-1}(Y^k)$. For each $x\in X^k$, let $U_x^k=U(x,f,R^k(f(x)))$, and for $\lambda\geq 0$, let $\lambda U_x^k=U(x,f,\lambda R^k(f(x)))$. We also let $S^k=f^{-1}(T^k)$, and for each $j=1,\dots, c_d$, let $S_j^k=f^{-1}(T_j^k)$. We then have $X^k=\bigcup_{x\in S^k}U_x^k$, and for each $x,x'\in S_j^k$, the sets $2U_x^k$ are	either disjoint or equal, the latter holding if and only if $f(x)=f(x')$. It follows that we may
	choose a function $\eta_j^k\colon S_j^k\to \{1, \dots , N\}$ such that for each $y\in Y^k$ and each $x_1$, $x_2\in f^{-1}(y)$, $\eta_j^k(x_1)=\eta_j^k(x_2)$ if and only if $2U_{x_1}^k=2U_{x_2}^k$. Note, as well, that since a continuum contained in $f^{-1}(B_{f(x)}^k)$ must lie completely inside $2U_x^k$, and $X$ has 1 bounded turning, we have $d(X\backslash 2U_x^k,U_x^k)\geq R^k(f(x))$, for all $x\in X^k$.
	
	Now let $\phi_j^k\colon X\to \R$ be defined by
	\begin{align*}
		\phi_j^k(x)=\sum_{x'\in S_j^k}\eta_j^k(x')\min\big\{d(x,X\backslash 2U_{x'}^k),R^k(f(x'))\big\},
	\end{align*} 	 
	and let $\phi\colon X\to \R^{c_d(N-1)}$ be given by
	\begin{align*}
		\phi=\phi_1^1\times \cdots\times \phi_j^k\times \cdots\times \phi_{c_d}^{N-1}.
	\end{align*}
	To complete the proof, it suffices to show that $\phi$ satisfies the hypotheses of Lemma~\ref{lemma:for biLip embedding}.
	
	To this end, we first note that each $\phi_j^k$ is clearly $N$-Lipschitz, and satisfies $\eta(x')R^k(f(x'))$ whenever $x\in U_x^k$. It follows that $\phi$ itself is Lipschitz.
	
	On the other hand, if $x_1=x_2$, and $y=f(x_1)=f(x_2)$, then let $k$ be chosen so that inequality~\eqref{eq:biLip embed 5} holds. A fortiori, $y\in Y_k$, so $y\in B_y^k$ for some $y\in T_j^k$, some $j$. By Lemma~\eqref{lemma:lift Floyd},	we have $x_1\in U_x^k$, $x_2\in U_x^k$ for some $x_1',x_2'\in f^{-1}(y')\subset S_j^k$. In fact, we may take
    $x_i\in B_{x_i'}^k:=B(x_i',R_k(y))$, since by lemma~\ref{lemma:pullback property 1}, $U_x^k$ is a union of such balls. Thus, applying inequalities \eqref{eq:biLip embed 5} and \eqref{eq:biLip embed 6}, we obtain
    \begin{align*}
    	d(x_1',x_2')&\geq d(x_1,x_2)-d(x_1',x_1)-d(x_2,x_2')>5R^k(y)-2R^k(y')\\
    	&\geq 5R^k(y')-d(y,y')-2R^k(y)\geq 2R^k(y'),
    \end{align*}
    and so by Lemma~\ref{lemma:pullback property 1}, $U_{x_1'}^k\neq U_{x_2'}^k$, so that $\eta_j^k(x_1')\neq \eta_j^k(x_2')$. Moreover, since $x_i\in U_{x_i'}$, we have $\phi_j^k(x_i)=\eta_j^k(x_i')R^k(y')$, $i=1,2$. We therefore have, again, by inequalities~\eqref{eq:biLip embed 5} and~\eqref{eq:biLip embed 6}, that
    \begin{align*}
    	d(x_1,x_2)&\leq 10R^k(y)\leq 10R^k(y')+2d(y,y')\leq 12R^k(y')\\
    	&\leq 12|\eta_j^k(x_1')-\eta_j^k(x_2')|R^k(y')=12|\phi_j^k(x_1)-\phi_j^k(x_2)|\leq 12|\phi(x_1)-\phi(x_2)|.
    \end{align*}
    Thus $\phi$ satisfies the hypotheses of Lemma \eqref{lemma:for biLip embedding}, as desired.
\end{proof}

\subsection{Some natural examples and counter-examples}\label{subsec:some natural examples}
In this section, we use the pullback metric to construct branched coverings between nice metric spaces so that the image of the branch set under the branched covering will have positive Lebesgue measure. 

Follow~\cite{hr02,hs02}, we say that an oriented generalized $n$-manifold $X$ is \emph{locally BLD-Euclidean of dimension $n$} if each point $x\in X$ has a neighborhood $U_x$ together with a Lipschitz BLD mapping $f\colon U_x\to \R^n$. By~\cite[Proposition 6.31]{hs02}, a locally BLD-Euclidean space $X$ of dimension $n$ is a Lipschitz $n$-manifold outside a singular set of topological dimension at most $n-2$. Moreover, $X$ is locally Ahlfors $n$-regular, $n$-rectifiable and has locally finite $\mathcal{H}^{n}$-measure.
 
Our main result of this section reads as follows.
\begin{theorem}\label{thm:counter-example branch set I}
	Let $Z$ be a topological (or generalized) $n$-manifold, and $g\colon Z\to \mathbb{M}$ a branched covering onto a Riemannian $n$-manifold $\mathbb{M}$, such that $N(g,Z)<\infty$ and that $\mathcal{L}^n(g(\mathcal{B}_g))>0$. Then $\mathbb{M}^g$ is a locally BLD-Euclidean metric space of dimension $n$ that is neither locally linearly locally contractible nor locally metrically orientable. In particular, $\mathbb{M}^g$ is $n$-rectifiable, locally Ahlfors $n$-regular, locally geodesic, locally satisfies a $(1,1)$-Poincar\'e inequality, and has local bi-Lipschitz embeddings into some Euclidean space, but is not locally quasiconformally equivalent to any neighborhood of $\R^n$.
\end{theorem}

It is not hard to construct topological branched coverings for which the image of the branch set has positive measure. As a result, Theorem~\ref{thm:counter-example branch set I} provides a rather rich source of examples and counterexamples.

For example, a conjecture of Heinonen and Rickman~\cite[Remark 6.32 (b)]{hr02}  is that the branch set of a Lipschitz BLD map $f\colon X\to \R^n$ from a generalized $n$-manifold into Euclidean space has Hausdorff $n$-measure 0, provided that $X$ has local bi-Lipschitz embeddings into some larger Euclidean space. As a corollary to Theorem~\ref{thm:counter-example branch set I} (and invoking also Corollary~\ref{coro:biLip embedding II}), we can show that this conjecture is false, and that there are in fact counterexamples of degree 2 in all dimensions, even under the restrictions that $X$ is homeomorphic to $\bR^n$ and the bilipschitz embedding is global.

\begin{corollary}\label{coro:counter-example branch set II}
	For each $n\geq 3$, there is a subspace $X\subset \R^N$ for some $N\geq n$, homeomorphic to $\R^n$, and a 1-BLD branched covering $f\colon X\to \R^n$ of degree 2 such that $\mathcal{L}^n(\mathcal{B}_f)>0$.
\end{corollary}

Heinonen and Semmes also asked~\cite[Question 33]{hs97} if every closed topological 4-manifold admits a metric so that it is Ahlfors 4-regular and locally linearly locally contractible. (The answer to this question is ``yes" for manifolds of dimension $n\neq 4$, by the work of Sullivan~\cite{s79} showing that every such manifold has a Lipschitz structure.) A motivation for this is that an affirmative answer would imply, by the work of Semmes~\cite{s96}, that every closed topological 4-manifold admits an $n$-regular, $n$-rectifiable
metric satisfying a $(1,1)$-Poincar\'e inequality.

The pullback construction gives another possible avenue to the question of whether 4-manifolds can be metrized to admit Poincar\'e inequalities. Heinonen and Semmes~\cite[Question 31]{hs97} asks if every closed topological 4-manifold is a branched cover of $\mathbb{S}^4$. If the answer to this question is ``yes", then by Theorem~\ref{thm:counter-example branch set I}, we may indeed give each 4-manifold such a metric.

\begin{proof}[Proof of Theorem~\ref{thm:counter-example branch set I}]
	Let $\pi\colon g^*\mathbb{M}\to \mathbb{M}$ be the projection (from the pullback factorization). That $g^*\mathbb{M}$ is a locally BLD-Euclidean space follows directly from Lemma~\ref{lemma:pullback property 1}. Note that if $g^*M$ is either locally linearly locally contractible or locally metrically orientable, then it would follow from~\cite[Theorem 6.4]{hr02} that $\mathcal{L}^n(g(\mathcal{B}_g))=\mathcal{L}^n(\pi(\mathcal{B}_\pi))=0$. Thus $g^*\mathbb{M}$ is neither locally linearly locally contractible nor locally metrically orientable. That $g^*M$ is locally $n$-rectifiable, locally Ahlfors $n$-regular, locally geodesic also follows immediately from Lemma~\ref{lemma:pullback property 1}, Lemma~\ref{lemma:pullback property 2}, and the local geometry of the Riemannian $n$-manifolds $\mathbb{M}$. The fact that $g^*\mathbb{M}$ supports a $(1,1)$-Poincar\'e inequality follows directly from~\cite[Theorem 9.8]{hr02}. The last assertion is a direct consequence of Corollary~\ref{coro:biLip embedding II}.
\end{proof}

\begin{proof}[Proof of Corollary~\ref{coro:counter-example branch set II}]
	Choose a topological branched covering $g\colon \R^n\to \R^n$ with degree 2 (for instance the standard winding mapping; e.g.~\cite[Example I 3.1]{r93}) so that $\mathcal{L}^n(g(\mathcal{B}_g))>0$ and then invoke Theorem~\ref{thm:counter-example branch set I}.
\end{proof}

\subsection{Geometric parametrization of metric spaces}
In the field of analysis and geometry on metric spaces, one of the celebrated open problems is to find good geometric parametrization of certain classes of metric spaces (see e.g.~\cite[Section 16.5]{h07}). 

For the bi-Lipschitz parametrization, it was initiated earliest by an observation due to Siebenmann and Sullivan~\cite{ss77}. Remarkable positive parametrization results were achieved by Toro~\cite{t94,t95}. There are also several highly non-trivial results, both positive and negative, were obtained by Semmes~\cite{s91,s96c,s96b}, David--Semmes~\cite{ds93,ds97}, Bonk--Lang~\cite{bl03}, Bonk--Heinonen--Saksman~\cite{bhs08}, and Heinonen--Rickman~\cite{hr02}. Inspired by Sullivan's work~\cite{s79,s95}, a simple geometric condition, that is sufficient for a space to admit local bi-Lipschitz parametrization by the Euclidean spaces, was provided in the remarkable works of Heinonen--Sullivan~\cite{hs02} and Heinonen--Keith~\cite{hk11}.

Instead of the bi-Lipschitz parametrization, people also ask for weaker geometric parametrizations, e.g. quasiconformal or quasisymmetric parametrization. The research along this direction often seeks for a version of the classical uniformization theorem for a certain class of two-dimensional metric spaces.  
Uniformization problems concerning quasiconformal and quasisymmetric mappings have received considerable attention in recent years, and they have found significant applications in geometry, complex dynamics, geometric topology and geometric measure theory, among other areas. In particular, several problems in the theory of hyperbolic groups can be interpreted as uniformization problems
concerning boundaries of the groups in question; see for instance~\cite{b06,b11,bk05,bm13,bm15,k06} and also \cite{tv80,c94,abt15,m02,m10}.

We would like to mention the remarkable result of Bonk--Kleiner~\cite{bk02}, where it has been shown that for an Ahlfors 2-regular topological sphere, it is quasisymmetric to the standard sphere $\mathbb{S}^2$ if and only if it is linearly locally contractible (LLC*). LLC* is a geometric condition that rules out cusp-like spaces. This beautiful result was later extended in several consequent works~\cite{mw13,w08,w10}. Very recently, in a celebrated result of Rajala~\cite{r14}, a quasiconformal analogy of the Bonk--Kleiner result was obtained via a geometric approach.

In higher dimensions, the uniformization problem does not have a satisfactory answer even for very nice metric spaces. Examples by Semmes~\cite{s96b}
show that the result of Bonk--Kleiner mentioned above does not generalize to dimension 3. Heinonen and Wu~\cite{hw10}, Pankka and Wu~\cite{pw14} and Pankka and Vellies~\cite{pv15} gave further examples of geometrically nice spaces without quasisymmetric  parametrizations.
 
\newpage
\section{Characterizations of BLD mappings in metric spaces with bounded geometry}\label{sec:characterization of BLD mappings}

\subsection{Background and formulation}
In this section, we give another application of the pullback factorization and the theory of quasiregular mapping that we have developed in Section~\ref{sec:foundations of QR mappings in mms}. 

BLD mappings in Euclidean spaces were first introduced by Martio and V\"ais\"al\"a~\cite{mv88} in their study of second-order elliptic operators. They established many interesting analytic and geometric properties of BLD mappings in $\R^n$ via the theory of quasiregular mappings. In particular, they obtained the following quantitative analytic characterization of BLD mappings: \textit{a continuous mapping $f\colon \Omega\to \R^n$, $n\geq 2$, is BLD if and only if $f$ is locally uniformly Lipschitz and the Jacobian determinant $J_f=\det Df$ is positive and uniformly bounded away from zero almost everywhere in $\Omega$}. This description of BLD mappings in Euclidean spaces  does not include the assumption that the mapping is discrete and open, nor that it is sense-preserving. In fact, mappings satisfying the latter analytic conditions as above form a strict subclass of quasiregular mappings. A deep theorem of Reshetnyak~\cite{r93} implies that (non-constant) quasiregular mappings are both discrete and open, and consequently they are sense-preserving as well.

In~\cite{hs02}, Heinonen and Sullivan successfully generalized Reshetnyak's theorem to quasiregular mappings from generalized $n$-manifolds of type $A$ into $\R^n$. As a typical application of this result, Heinonen and Rickman~\cite[Theorem 6.18]{hr02} have obtained a similar analytic characterization of BLD mappings in the setting of mappings from generalized $n$-manifolds of type $A$ to $\R^n$. 

Our main result of this section is the following quantitative characterizations of BLD mappings in Ahlfors $Q$-regular $Q$-Loewner spaces. Throughout the entire section, $Q$ will be a real number strictly larger than one and $n\geq 2$ an integer.
\begin{theorem}\label{thm:main thm BLD}
	Let $f\colon X\to Y$ be a branched covering between two Ahlfors $Q$-regular $Q$-Loewner spaces. Then the following statements are equivalent:
	
	1). $f$ is $L$-BLD;
	
	2). For each $x\in X$, there exists $r_x>0$ such that 
	\begin{align*}
	\frac{d(x,y)}{c}\leq d(f(x),f(y))\leq cd(x,y)
	\end{align*}
	for all $y\in B(x,r_x)$;
	
	3). $L_f(x)\leq c$ and $l_f(x)\geq \frac{1}{c}$ for each $x\in X$;
	
	4). $f$ is metrically $H$-quasiregular, locally $M$-Lipschitz, and $J_f(x)\geq c$ for a.e. $x\in X$.
	
	Moreover, all the constants involved depend quantitatively only on each other and on the data associated to $X$ and $Y$. 
\end{theorem}
Recall that for a mapping $f\colon X\to Y$, $L_f$ and $l_f$ are defined as
\begin{align*}
L_f(x)=\limsup_{y\to x}\frac{d(f(x),f(y))}{d(x,y)}\quad \text{and}\quad l_f(x)=\liminf_{y\to x}\frac{d(f(x),f(y))}{d(x,y)}.
\end{align*}

The assumption that $f$ is $K$-quasiregular in Theorem~\ref{thm:main thm BLD} 4) can be dropped, since it is implied by the other two conditions; see Lemma~\ref{lemma:Lip with positive Jacobian is QR} below. We prefer the current formulation simply because we expect that Theorem~\ref{thm:main thm BLD} would hold in a wider class of metric spaces, where the (metric) quasiregularity does not necessarily follow from the locally Lipschitz regularity and the lower positive bounds on the (volume) Jacobian.

As commented in the beginning of this section, Theorem~\ref{thm:main thm BLD} was first proved in the Euclidean spaces by Martio and V\"ais\"al\"a~\cite[Theorem 2.16]{mv88}, and Later, generalized by Heinonen and Rickman~\cite[Theorem 6.18]{hr02}, to mappings from generalized $n$-manifolds of type $A$ into the Euclidean space $\bR^n$. The proof of Martio and V\"ais\"al\"a depends not only on the geometry of Euclidean spaces, but also on the differentiable structure of Euclidean spaces, thus it cannot be generalized to our setting. Heinonen and Rickman were able to give a proof independent of the differentiable structure of Euclidean spaces, but their proof still depends heavily on the geometry of Euclidean spaces (as the target space). It seems for us not so easy to adjust their proof to the setting of mappings between two generalized $n$-manifolds of type $A$.  

The new approach we have developed here is to use the pullback factorization from Section~\ref{subsec:Canonical factorization} to factorize $f$ as $f=\pi\circ g$ and then transfer the information of $f$ to its lift mapping $g\colon X\to X^f$. Then using the techniques from Section~\ref{sec:foundations of QR mappings in mms} and the basic properties of pullback factorizations to show that $g$ is bi-Lipschitz, quantitatively. The main obstacle here is to update the $\mu$-a.e. pointwise information of $f$ in Theorem \ref{thm:main thm BLD}~4) to the everywhere defined BLD condition in Theorem \ref{thm:main thm BLD}~1). At the level of $f$, this is not an easy task as already observed in~\cite[Lemma 6.19, Lemma 6.23 and Proof of Theorem 6.18]{hr02}. Somewhat surprisingly, this is quite easy to handle, via a simple lemma (see~Lemma \ref{lemma:constant weak ug implies Lipschitz} below), at the level of the lifting mapping $g$.

Very recently, Luisto~\cite{l15b} has shown that the equivalences of 1)--3) hold in very general metric spaces. Moreover, for a continuous discrete mapping $f\colon X\to Y$ between length spaces, it has been shown that $f$ is $L$-BLD if and only it is $L$-\textit{Lipschitz Quotient} ($L$-LQ for short), \ie, for each $x\in X$ and $r>0$,
\begin{align}\label{def:LQ mapping}
B\big(f(x),L^{-1}r\big)\subset f\big(B(x,r)\big)\subset B\big(f(x),Lr\big).
\end{align} 
Based on the afore-mentioned characterizations, Luisto has obtained an interesting convergence result that generalizes earlier works of Martio-V\"ais\"al\"a~\cite[Theorem 4.7]{mv88}, Heinonen-Keith~\cite[Lemma 6.2]{hk11} and Luisto~\cite[Corollary 1.3]{l15}. Using the recent result on the size of the branch set of a quasiregular mapping obtained by the authors~\cite{gw15}, we are able to improve on Luisto's result into the following form.

\begin{theorem}\label{thm:convergence result}
	Let $X$ and $Y$ be two Ahlfors $Q$-regular generalized $n$-manifolds. Assume additionally that $X$ is locally linearly locally contractible and locally quasiconvex and that $Y$ is locally quasiconvex. Let $\{f_i\}_{i\in \mathbb{N}}$ be a sequence of $L$-BLD mappings converging locally uniformly to a continuous mapping $f\colon X\to Y$. Then $f$ is $L$-BLD. 
\end{theorem}

Recall that $X$ is \textit{locally linearly locally contractible} if for each $x\in X$, there exists a neighborhood $U_x$ of $x$ such that $U_x$ is $c_x$-linearly locally contractible, i.e., each ball $B(y,r)\subset U_x$ is contractible in $B(y,c_xr)$.

The assumptions that $X$ and $Y$ are locally quasiconvex can be further relaxed; see Remark~\ref{rmk:on removing quasiconvexity} below. Our interest in Theorem~\ref{thm:convergence result} lies in seeking for an easy proof of the geometric porosity for the branch set of a BLD mapping in metric spaces via a blow up argument. We will not enter any technical details, since it clearly beyonds the scope of the current paper, but we will address this issue in a sequential paper.  

\subsection{Auxiliary results}
The following lemma is certainly well-known to experts. However, we present a simple proof here due to lack of a precise reference.
\begin{lemma}\label{lemma:constant weak ug implies Lipschitz}
	Let $h\colon X\to Y$ be a continuous $N^{1,Q}_{loc}(X,Y)$-mapping between two Ahlfors $Q$-regular $Q$-Loewner spaces such that the constant function $c$ is a $Q$-weak upper gradient of $h$. Then $h$ is $C$-Lipschitz, quantitatively.
\end{lemma}
\begin{proof}
	Since $X$ $Q$-Ahlfors regular and $Q$-Loewner, it supports a $(1,Q)$-Poincar\'e inequality and so by the well-known pointwise characterization of Sobolev spaces $N^{1,Q}_{loc}$ (cf.~\cite[Theorem 9.5]{h01}), we have
	\begin{align*}
	d(h(x_1),h(x_2))\leq Cd(x_1,x_2)\big(M(g_h^Q(x_1))+M(g_h^Q(x_2))\big)^{1/Q}
	\end{align*}
	for a.e. $x_1,x_2\in X$, where the positive constant $C$ depends quantitatively on the data of the spaces. Hence, for a.e. $x_1,x_2\in X$,
	\begin{align*}
	d(h(x_1),h(x_2))\leq 2cCd(x_1,x_2).
	\end{align*}
	Since $h$ is continuous, a simple density argument implies that the above inequality holds for all $x_1,x_2\in X$, from which it follows that $h$ is $C'$-Lipschitz, quantitatively.
\end{proof}

The following theorem is due to Heinonen and Rickman~\cite[Theorem 6.8]{hr02}. Note that in~\cite[Theorem 6.8]{hr02}, the spaces were assumed to be generalized $n$-manifolds of type $A$. But these assumptions were only used to deduce the fact that $\mathscr{H}^Q\big(f(\mathcal{B}_f)\big)=0$.
\begin{theorem}\label{thm:Heinonen-Rickman}
	Let $f\colon X\to Y$ be an $L$-BLD mapping between two Ahlfors $Q$-regular generalized $n$-manifolds with $\mathscr{H}^Q\big(f(\mathcal{B}_f)\big)=0$. Assume that $X$ is $c_X$-quasiconvex and that $Y$ is $c_Y$-quasiconvex. If $x_0\in X$, $r>0$, and if $\lambda>1$ such that the ball $B(x_0,\lambda r)$ has compact closure in $X$, then 
	\begin{align*}
	N(y,f,B(x_0,r))\leq \big(Lc_X\big)^Q\frac{\mathscr{H}^Q\big(B(x_0,\lambda r)\big)}{\mathscr{H}^Q\big(B(y,(\lambda-1)r/Lc_Y)\big)}
	\end{align*}
	for all $y\in Y$.
\end{theorem}

The following lemma is an easy consequence of Theorem~\ref{thm:bounded geometry III}.
\begin{lemma}\label{lemma:Lip with positive Jacobian is QR}
	Let $f\colon X\to Y$ be a branched covering between two Ahlfors $Q$-regular $Q$-Loewner metric spaces. If $f$ is locally $M$-Lipschitz and if $J_f>c$ a.e. in $X$, then $f$ is metrically $H$-quasiregular, quantitatively.
\end{lemma}
\begin{proof}
	Since $f$ is locally $M$-Lipschitz, $f\in N^{1,Q}_{loc}(X,Y)$. Note also that $M$ is a $Q$-weak upper gradient of $f$ (see e.g.~\cite[Lemma 6.2.6]{hkst15} or~\cite[Proposition 10.2]{hk00}) and so we have the following pointwise estimates:
	\begin{align*}
	g_f(x)^Q\leq M^Q= \big(\frac{M^Q}{c}\big)c\leq CJ_f(x)\quad \text{a.e. in }X.
	\end{align*}
	This means that $f$ is analytically $C$-quasiregular with exponent $Q$. Since our spaces are Ahlfors $Q$-regular and $Q$-Loewner, by Theorem~\ref{thm:bounded geometry III}, $f$ is metrically $H$-quasiregular, quantitatively.
\end{proof}

For a metric space $Z=(Z,d_Z)$, we use the notation $\hat{Z}=(Z,\hat{d}_Z)$, where $\hat{d}_Z$ is the internal/length metric on $Z$. For a mapping $\gamma$ from an interval to the set $Z$, we donote by $l_Z(\gamma)$ the length of $\gamma$ with respect to $d_Z$ and $l_{\hat{Z}}(\gamma)$ the length of $\gamma$ with respect to $\hat{d}_{Z}$. The following simple lemma will be used in the proof of Theorem~\ref{thm:convergence result}.
\begin{lemma}\label{lemma:BLD via length metric}
	Let $X$ and $Y$ be two locally quasiconvex metric spaces. Then $f\colon X\to Y$ is $L$-BLD if and only if $g\colon \hat{X}\to \hat{Y}$ is $L$-BLD, where $g=f$ on the set $X$.
\end{lemma}
\begin{proof}
 Note first that since our metric spaces are locally quasiconvex, $\hat{d}_X$ and $d_X$ induces the same topology on $X$. Similarly, $\hat{d}_Y$ and $d_Y$ induces the same topology on $Y$. This implies that $f\colon X\to Y$ is a branched covering if and only if $g\colon \hat{X}\to \hat{Y}$ is. To prove the claim, we need to show the bi-Lipschitz condition on lengths of curves. Note also that for an arbitrary metric $d$, $\gamma$ is a rectifiable curve with respect to $d$ if and only if it is with respect to the internal metric, and in that case, the lengths coincide.
	
	Suppose now $f\colon X\to Y$ is $L$-BLD and we want to show that $g\colon \hat{X}\to \hat{Y}$ is $L$-BLD. To this end, fix a curve $\gamma\colon I\to \hat{X}$. If $\gamma$ is rectifiable with respect to $\hat{d}_X$, then the BLD condition on $f$ implies that
	\begin{align*}
		L^{-1}l_{\hat{X}}(\gamma)=L^{-1}l_{X}(\gamma)\leq l_Y(f\circ \gamma)\leq Ll_X(\gamma)=Ll_{\hat{X}}(\gamma)<\infty,
	\end{align*} 
	which in particular implies that the curve $f\circ \gamma$ is rectifiable with respect to $d_Y$ and hence has the same length with respect to $\hat{d}_Y$. Thus we have
	\begin{align*}
		L^{-1}l_{\hat{X}}(\gamma)\leq l_{\hat{Y}}(g\circ \gamma)\leq Ll_{\hat{X}}(\gamma)
	\end{align*}
	for all rectifiable curves $\gamma$ in $\hat{X}$. If $\gamma$ is not rectifiable with respect to $\hat{d}_X$, then it is also not rectifiable with respect to $d_X$. Consequently, we have
	\begin{align*}
		\infty=L^{-1}l_{\hat{X}}(\gamma)=L^{-1}l_{X}(\gamma)\leq l_Y(f\circ \gamma).
	\end{align*}
	This means that the curve $f\circ \gamma$ is not rectifiable with respect to $d_Y$ and thus is not rectifiable with respect to $d_{\hat{Y}}$. So we have $l_{\hat{Y}}(f\circ \gamma)=\infty=l_{\hat{X}}(\gamma)$.
	
	The proof of the reverse direction is entirely the same and thus it is omitted here.
\end{proof}

\subsection{Proof of the main results}

\begin{proof}[Proof of Theorem~\ref{thm:main thm BLD}]
	We only need to show that 4) implies 1) since the other implications were already proved in~\cite[Proposition 3.2]{gw15}. Let us assume that all the conditions in 4) are satisfied and additionally that $X$ and $Y$ are length spaces (this is a harmless assumption since our considerations here are bi-Lipschitzly invariant). Let $f=\pi\circ g$ be the pullback factorization as in Section~\ref{subsec:Canonical factorization}. By our discussions in Section~\ref{subsec:Fine properties of the pullback factorization}, it suffices to show that $g\colon X\to X^f$ is bi-Lipschitz, quantitatively.
	
	Since $f\colon X\to Y$ is metrically $H$-quasiregular, by Theorem~\ref{thm:bounded geometry I}, $H_f^*<\infty$ everywhere in $X$ and $H_f^*\leq H_1$ $\mu$-a.e. in $X$, for some positive constant $H_1$ depending quantitatively on on $H$ and the data of the spaces. By Proposition~\ref{prop:characterization of QR via pullback factorization II}, we also have $H_g^*(x)=H_f^*(x)<\infty$ everywhere in $X$ and $H_g^*(x)\leq H'$ $\mu$-a.e. in $X$ (with some quantitative positive constant $H'$). Since by Lemma~\ref{lemma:Condition N+N inverse}, $f$ satisfies both Condition $N$ and Condition $N^{-1}$, so is $g$. It follows from~\cite[Proof of Theorem 1.6]{w14} that $H_{g^{-1}}(y)<\infty$ for all $y\in X^f$ and $H_{g^{-1}}(y)\leq H'$ for $\lambda$-a.e. $y\in X^f$, or equivalently, $g^{-1}\colon X^f\to X$ is metrically $H'$-quasiconformal. As a consequence, we have
	\begin{align}\label{eq:chain rule}
	J_g\cdot J_{g^{-1}}\circ g=1\qquad \mu\text{-a.e. in } X.
	\end{align}
	On the other hand, since $f$ is locally $M$-Lipschitz and $X$ is a length space, $g$ is locally $M$-Lipschitz as well. In particular, the constant function $M$ is an upper gradient of $g$. By~\eqref{eq:chain rule} and the $\mu$-a.e. lower bound on $J_g=J_f$, we conclude that $J_{g^{-1}}\leq \frac{1}{c}$ $\lambda$-a.e. in $X^f$. 
	Note that by Theorem~\ref{thm:bounded geometry III} and Theorem~\ref{thm:equivalence of G and A}, $g^{-1}$ is also  analytically $K$-quasiconformal with exponent $Q$, quantitatively, i.e., $g^{-1}\in N_{loc}^{1,Q}(X^f,X)$,  $|\nabla g^{-1}|^Q\leq KJ_{g^{-1}}$ $\lambda$-a.e. in $X^f$ (with $K$ depending only on $H$, $c$ and the data of the spaces). Thus, some constant function $c'$, which depends only on $K$, $Q$ and $c$, is a $Q$-weak upper gradient of $g^{-1}$. Hence, by Lemma~\ref{lemma:constant weak ug implies Lipschitz}, $g^{-1}$ is $L'$-Lipschitz, quantitatively. In other words, we have shown that $g$ is $L$-bi-Lipschitz, quantitatively. This completes our proof of the theorem.
\end{proof}


\begin{proof}[Proof of Theorem~\ref{thm:convergence result}]
	We first claim that $f$ is discrete. Since the proof of this claim is very similar to~\cite[Proof of Theorem 1.4]{l15b}, we only point out the difference. The key observation is that under the assumption of Theorem~\ref{thm:convergence result}, it follows from~\cite[Corollary 1.2 and Remark 3.3 i)]{gw15} that $\mathscr{H}^Q(f_i(\mathcal{B}_{f_i}))=0$ for each $i\in \mathbb{N}$. Thus we may apply Theorem~\ref{thm:Heinonen-Rickman} to each of the BLD mappings $f_i$ to run the proof of Theorem 1.4 as in~\cite{l15b}. 
		
	Secondly, by~\cite[Theorem 1.1]{l15b} and the fact that $f$ is a discrete $L$-LQ mapping, $f$ will be $L$-BLD if the metric spaces $X$ and $Y$ were length spaces. Let $\hat{d}_X$ and $\hat{d}_Y$ be the internal metrics on $X$ and $Y$.  Lemma~\ref{lemma:BLD via length metric} implies that a mapping $g\colon (X,d_X)\to (Y,d_Y)$ is $L$-BLD if and only if the mapping $\hat{g}\colon (X,\hat{d}_X)\to (Y,\hat{d}_Y)$, defined by $\hat{g}:=g$, is $L$-BLD. Thus, we have $f_i\colon (X,\hat{d}_X)\to (Y,\hat{d}_Y)$, $i=1,2,\dots$ form a sequence of $L$-BLD mappings between two complete, locally compact length spaces, converging locally uniformly to a discrete $L$-LQ mapping $\hat{f}\colon (X,\hat{d}_X)\to (Y,\hat{d}_Y)$ defined by $\hat{f}:=f$. It follows from~\cite[Theorem 1.1]{l15b} that $\hat{f}\colon (X,\hat{d}_X)\to (Y,\hat{d}_Y)$ is an $L$-BLD mapping, or equivalently, $f\colon X\to Y$ is $L$-BLD. 
\end{proof}

\begin{remark}\label{rmk:on removing quasiconvexity}
i).	It is clear from the proof of Theorem~\ref{thm:convergence result} that we only need the following two facts. 
	
	First of all, a mapping $g\colon (X,d_X)\to (Y,d_Y)$ is $L$-BLD if and only if the mapping $\hat{g}\colon (X,\hat{d}_X)\to (Y,\hat{d}_Y)$, defined by $\hat{g}:=g$, is $L$-BLD. For this, it suffices to assume that $X$ and $Y$ are rectifiably connected, i.e., each two points in $X$ (and $Y$) can be joined by a rectifiable curve in $X$ (and $Y$) and that $(X,\hat{d}_X)$ and $(Y,\hat{d}_Y)$ have the same topology as $(X,d_X)$ and $(Y,d_Y)$, respectively.
	
	Secondly, we need Theorem~\ref{thm:Heinonen-Rickman} to run the proof of Theorem 1.3 in~\cite{l15b}. For this, one essentially needs the fact that the image of the branch set of a BLD mapping has zero Hausdorff measure, and thus by~\cite[Corollary 1.2]{gw15}, it suffices to assume that $X$ is (locally) linearly locally contractiable and that $Y$ has (local) bounded turning. 
	
ii). The proof of Theorem~\ref{thm:convergence result} implies the following stronger statement: \emph{if $(X_j,x_j)$ and $(Y_j,y_j)$ are two sequences of pointed Ahlfors $Q$-regular generalized $n$-manifolds with uniform Ahlfors regularity constant. Assume additionally that all the $X_j$ are uniformly locally linearly locally contractible and all $X_j$ and $Y_j$ are uniformly locally quasiconvex. If the sequence of pointed mapping packages $\big((X_j,x_j),(Y_j,y_j),f_j\big)$ converges locally uniformly to a package $\big((X,x_0),(Y,y_0),f\big)$ with each $f_j$ being $L$-BLD, then $f$ is $L$-BLD.} For this, see~\cite[Proof of Theorem 1.4]{l15b}.
\end{remark}

\newpage
%
%
%



\end{document}